\documentclass[10pt]{amsart}%{article}
\usepackage{amsmath,amsthm, epsfig, amscd}
\usepackage{psfrag}
\usepackage[psamsfonts]{amssymb}
\usepackage[all]{xy}
\usepackage{mathrsfs}

\numberwithin{equation}{section}

\newtheorem{Thm}{Theorem}[section]
\newtheorem{Prop}[Thm]{Proposition}
\newtheorem{Lem}[Thm]{Lemma}
\newtheorem{Cor}[Thm]{Corollary}
\newtheorem{Conj}[Thm]{Conjecture}

\theoremstyle{remark}
\newtheorem{Rem}[Thm]{Remark}

\theoremstyle{definition}
\newtheorem{Def}[Thm]{Definition}
\newtheorem{Assum}[Thm]{Assumption}

\newtheorem{Exa}[Thm]{Example}

\newcommand{\mysection}[2]{%
\vspace{2mm}\section{\bf #1}\label{#2}
}

\newcommand{\fig}[1]
        {\raisebox{-0.5\height}
                 {\includegraphics{#1}}
        }

\newcommand{\cfig}[1]
        {\begin{center}\fig{#1}\end{center}
        }
\def\Z{{\mathbb Z}}
\def\R{{\mathbb R}}
\def\Q{{\mathbb Q}}

\def\calA{\mathscr{A}}

\def\calC{\mathscr{C}}
\def\calD{\mathscr{D}}

\def\calF{\mathscr{F}}
\def\calG{\mathscr{G}}

\def\calM{\mathscr{M}}
\def\calN{\mathscr{N}}

\def\calR{\mathscr{R}}
\def\calS{\mathscr{S}}

\def\calU{\mathscr{U}}

\def\calX{\mathscr{X}}

\def\deg{\mathrm{deg}}
\def\Hom{\mathrm{Hom}}

\def\tcoprod{\textstyle\coprod}

\newcommand{\mapright}[1]{
	\smash{\mathop{
		\hbox to 1cm{\rightarrowfill}}\limits^{#1}}}

\newcommand{\mapleft}[1]{
	\smash{\mathop{
		\hbox to 1cm{\leftarrowfill}}\limits^{#1}}}

\def\In{\mathrm{In}}
\def\Out{\mathrm{Out}}

\def\Br{\mathrm{Br}}
\def\Se{\mathrm{Se}}
\def\Comp{\mathrm{Comp}}
\def\End{\mathrm{End}}
\def\Tr{\mathrm{Tr}}

\def\asum#1#2{\sum_{{{#1}\atop{#2}}}}
\def\Conf{C}
\def\bvec#1{\mbox{\boldmath{$#1$}}}
\def\bcalM{\overline{\calM}}
\def\bcalN{\overline{\calN}}

\def\wgamma{\widetilde{\gamma}}
\def\bcalD{\overline{\calD}}
\def\bcalA{\overline{\calA}}

\def\Dom{\mathrm{Domain}}
\def\bbcalM#1#2#3{{\calM'(#1;#2,#3)}}
\def\codim{\mathrm{codim}\,}
\def\Tan{\mathrm{Tan}}
\def\ve{\varepsilon}
\def\bConf{\overline{C}}
\def\loc{\mathrm{local}}
\def\wDelta{\widehat{\Delta}}
\def\bM{\mathsf{M}}
\def\wcalD{\widetilde{\calD}}
\def\wcalA{\widetilde{\calA}}

\begin{document}

\title[A generalization of Fukaya's invariant of 3-manifolds I]{Higher order generalization of Fukaya's Morse homotopy invariant of 3-manifolds I. Invariants of homology 3-spheres}
\author{Tadayuki Watanabe}
\address{Department of Mathematics, Shimane University,
1060 Nishikawatsu-cho, Matsue-shi, Shimane 690-8504, Japan}
\email{tadayuki@riko.shimane-u.ac.jp}
\date{\today}

{\noindent\footnotesize {\rm Preprint} (2012)}\par\vspace{15mm}
\maketitle
\vspace{-6mm}
\setcounter{tocdepth}{2}
\begin{abstract}
{We give a generalization of Fukaya's Morse homotopy theoretic approach for 2-loop Chern--Simons perturbation theory to 3-valent graphs with arbitrary number of loops at least 2. We construct a sequence of invariants of integral homology 3-spheres with values in a space of 3-valent graphs (Jacobi diagrams or Feynman diagrams) by counting graphs in an integral homology 3-sphere satisfying certain condition described by a set of ordinary differential equations.}
\end{abstract}
\par\vspace{3mm}
%\tableofcontents

%%%%%%%%%%%%%%%%%%%%%%%%%%%%%%
%%%%%%%%%%%%%%%%%%%%%%%%%%%%%%
%%%%%%%%%%%%%%%%%%%%%%%%%%%%%%
\def\baselinestretch{1.06}\small\normalsize

\mysection{Introduction}{s:intro}

After Witten's discovery of path integral invariants of knots and 3-manifolds (\cite{Wi2}), several mathematical constructions of universal invariant for homology 3-spheres appeared, e.g. perturbative Chern--Simons theory $Z^{\mathrm{CS}}$ of Axelrod--Singer \cite{AS} and Kontsevich \cite{Ko1}, and a combinatorial invariant $Z^{\mathrm{LMO}}$ of Le--Murakami--Ohtsuki \cite{LMO}. Here a homology 3-sphere denotes a closed 3-manifold $M$ with $H_*(M;\Z)\cong H_*(S^3;\Z)$. These invariants take values in a space of graphs called Jacobi diagrams or Feynman diagrams (\cite{BN, Ko2}), and are known to be universal among Ohtsuki's finite type invariants for rational homology 3-spheres (\cite{BGRT, KT, Les2}). $Z^{\mathrm{CS}}$ is defined by integration over spaces of configurations of points on a 3-manifold. $Z^{\mathrm{LMO}}$ is constructed from Kontsevich's link invariant \cite{Ko2} by ingenious combinatorial arguments.

This article is concerned with Fukaya's graph counting invariant of 3-manifolds, developed in \cite{Fuk2} via Morse homotopy theory (\cite{BC, Fuk1}) and conjectured to coincide with 2-loop perturbative Chern--Simons theory. Fukaya considered triads $\vec{f}=(f_1,f_2,f_3)$ of Morse functions on a 3-manifold $M$ and for pairs $(M,\zeta_i)$, $i=1,2$, of $M$ and for acyclic flat Lie algebra bundles $\zeta_i$ on $M$, he defined some number $\hat{Z}_2(\vec{f};\zeta_i)$. Roughly, it counts with weights the ways that the $\Theta$-graph can be immersed such that edges follow gradient lines. The weights are determined by the holonomies taken along the edges. Fukaya showed that the difference $\hat{Z}_2(\vec{f};\zeta_1)-\hat{Z}_2(\vec{f};\zeta_2)$ depends only on $(M,\zeta_1,\zeta_2)$. 

Fukaya discusses in \cite{Fuk2} some heuristic argument involving the Witten deformation of de Rham complex (\cite{Wi1, Bo}) which suggests that his invariant coincides with the 2-loop part of perturbative Chern--Simons theory. Fukaya also discusses conjectural relation with open string theory on the cotangent bundle of a manifold.

The aim of this article is to construct graph-valued invariants of $\Z$-homology 3-spheres via Morse homotopy theory, as a higher order generalization of \cite{Fuk2}. We generalize the idea of Fukaya to graphs with the first Betti numbers $\geq 2$ for homology 3-sphere $M$ with the trivial connection and generalize Fukaya's conjecture which asks if his invariant coincides with perturbative Chern--Simons theory. To give an explicit formula for our invariant for all orders, we introduce an appropriate graph complex for Morse homotopy theory being based on Kontsevich's graph complex in \cite{Ko1}. 

As in \cite{Fuk2}, the proof that our invariant $\widehat{Z}_{2k,3k}$ is well-defined is done by a topological field theoretic argument for a 1-parameter family of smooth functions on $M$ without higher singularities. Namely, the difference of $\widehat{Z}_{2k,3k}$ for two auxiliary choices is given by the contribution of the 0-dimensional moduli spaces at the endpoints of a 1-parameter family. The moduli spaces of flow graphs generalized suitably to 1-parameter family gives a possibly non-compact cobordism between the 0-dimensional moduli spaces on the endpoints. The cobordism may have inner ends. By counting the contributions of the inner ends in the cobordism, we may obtain the difference of $\widehat{Z}_{2k,3k}$. To make the difference trivial, or the contributions of the inner ends cancel with each other, we consider some linear equations (the IHX relation) among coefficients for the counts of the 0-dimensional moduli spaces. The point is that the proof is reduced to checking that the sum of weighted counts of flow graphs is 0. In this paper, we consider graphs for all orders, so we attempt to give a general description of the structure of a smooth manifold of a moduli space of flow graphs and of arguments of orientations etc. in a similar fashion as \cite{BH, We}.

The moduli space of flow graphs will be described as the intersections of several submanifolds of a configuration space of $M$ or of the direct product of a configuration space of $M$ with $[0,1]$. We confirm the invariance of $\widehat{Z}_{2k,3k}$ one at a time by using a Cerf theoretic method as in \cite{Ce, Hu}. 

Also, unlike in \cite{Fuk2}, we consider only the trivial connection contribution and we do not take the difference of terms for two flat connections as in \cite{Fuk2}. To do so, we need to introduce an `anomaly correction' term appropriately. We define an anomaly term $Z_{2k,3k}^\mathrm{anomaly}$ by taking some linear combination of the numbers of infinitesimal flow graphs in a vector bundle over a compact 4-manifold $W$ with $\partial W=M$. The key point for the correction term to be well-defined is the spin cobordism invariance of the anomaly term $Z_{2k,3k}^\mathrm{anomaly}$. The spin cobordism invariance allows us to define an analogue of the signature defect, which includes $Z_{2k,3k}^\mathrm{anomaly}$ instead of the relative $L$-class, and it gives the desired correction term. 

\subsection{Organization} 
The organization of the present paper is as follows. In \S{2}, we give definitions of Fukaya's moduli spaces $\calM_\Gamma$ of flow graphs and our invariant. 

From \S{3} to \S{5}, we give some basics for the trajectory spaces. 
In \S{3}, we study the moduli space of gradient trajectories between two points and construct its compactification $\bcalM_2(f)$. In \S{4}, we define a compactification $\bcalM_\Gamma$ of $\calM_\Gamma$ using $\bcalM_2(f)$. In \S{5}, we study (co)orientations of the moduli spaces. 

From \S{6} to \S{7}, we show that our invariant depends only on a sequence of Morse functions and metrics on $M$. 
In \S{6}, we show that the principal term $Z_{2k,3k}$ is independent of combinatorial propagator. In \S{7}, we show that the correction term $Z_{2k,3k}^\mathrm{anomaly}(\vec{\gamma}_W)-\mu_k\,\mathrm{sign}\,W$ is independent of the choice of 4-cobordism $(W,\vec{\gamma}_W)$. 

In the final \S{10}, we shall show that our invariant is also independent of the choice of Morse functions and metrics on $M$ and complete the proof of the main theorem. \S{8} and \S{9} are preliminaries for \S{10}, which give basics for the trajectory spaces in 1-parameter family, which are mainly analogues of the results in \S{3} to \S{5}. In \S{8}, we consider the compactification for the moduli space of flow graphs in 1-parameter family of smooth functions to construct cobordisms. In \S{9}, we study (co)orientations of the moduli spaces in 1-parameter family. In \S{10}, in accordance with the results in previous sections, we check the invariance of our invariant by a cobordism argument. For each of the four types of bifurcations that may occur in a generic 1-parameter family, we confirm the invariance one at a time. 

In Appendix, we describe some facts on smooth manifolds with corners, convention for orientation, the chain complex of endomorphisms of an acyclic chain complex and the definition of blow-up.

\subsection{Conventions}
We denote by $C^r(M)$ the space of $C^r$ functions $f:M\to \R$ on a manifold $M$ for sufficiently large $r$ and we equip $C^r(M)$ the Whitney $C^r$-topology. By smooth maps or smooth manifolds we mean $C^r$ maps or $C^r$ manifolds for sufficiently large $r$. For a $C^r$ function $f$ on a manifold $M$, we denote by $\Sigma(f)$ the subset of $M$ of critical points of $f$. Let $\Sigma^1(f)$ denote the subset of $\Sigma(f)$ consisting of Morse singularities. For a Morse singularity $p\in \Sigma(f)$, we denote by $i(p)$ the Morse index of $p$. For a Morse function $f$, a critical point $p$ of $f$ and a Riemannian metric $\mu$ on a manifold, we denote by $\calD_p(f)=\calD_p(f;\mu)$ (resp. $\calA_p(f)=\calA_p(f;\mu)$) the descending manifold (resp. ascending manifold) of $(f,\mu)$ at $p$. 

We denote by $\Gamma(E)$ the space of sections of a fiber bundle $E\to B$. 

For a sequence of submanifolds $A_1,A_2,\ldots,A_r\subset W$ of a smooth Riemannian manifold $W$, we say that the intersection $A_1\cap A_2\cap \cdots\cap A_r$ is {\it transversal} if for each point $x$ in the intersection, the subspace $N_xA_1+N_xA_2+\cdots+N_xA_r\subset T_xW$ spans the direct sum $N_xA_1\oplus N_xA_2\oplus\cdots\oplus N_xA_r$, where $N_xA_i$ is the orthogonal complement of $T_xA_i$ in $T_xW$ with respect to the Riemannian metric. 

%\tableofcontents

%\clearpage
%%%%%%%%%%%%%%%%%%%%%%%%%%%%%%
%%%%%%%%%%%%%%%%%%%%%%%%%%%%%%
\mysection{Definition of the invariant}{s:def_inv}

In this section, the definition of Fukaya's moduli space of flow graphs in a manifold is recalled and the definition of our invariant is given.

%%%%%%%%%%%%%%%%%%%%%%%%%%%%%%
\subsection{Graphs}

By a {\it graph}, we mean a finite graph with each edge oriented, i.e. an ordering of the boundary vertices of an edge is fixed. We identify a graph with its geometric realization. For an oriented edge $e$ with orientation $(v_1,v_2)$, we call $v_1$ (resp. $v_2$) the input (resp. output) vertex of $e$. In diagrams we represent edge orientations by arrows directed toward the output vertices, as in Figure~\ref{fig:edges}. For a graph $\Gamma$, let 
\[ \begin{split}
\In(\Gamma)&=\{\mbox{univalent vertices of $\Gamma$ that are inputs}\},\\
\Out(\Gamma)&=\{\mbox{univalent vertices of $\Gamma$ that are outputs}\},\\
W(\Gamma)&=\{\mbox{vertices of $\Gamma$ of valence $\leq 2$ (``white vertices'')}\},\\
B(\Gamma)&=\{\mbox{vertices of $\Gamma$ of valence $\geq 3$ (``black vertices'')}\}.\\
\end{split}\]

We define an {\it admissible graph} to be a pair $(\Gamma,\rho)$, where
\begin{enumerate}
\item $\Gamma$ is a graph with $|\In(\Gamma)|=|\Out(\Gamma)|$, 
\item $\rho:\In(\Gamma)\to \Out(\Gamma)$ is a fixed bijection, and 
\item each bivalent vertex has exactly one incoming and one outgoing edges.
\item $\Gamma$ does not have a self-loop.
\end{enumerate}
We will omit $\rho$ when referring to an admissible graph. For an admissible graph $\Gamma$, we consider the following (sets of) edges:
\begin{enumerate}
\item A {\it compact edge} is an edge connecting two black vertices. 
\item A {\it separated edge} is a pair of edges $\{(a,x),(y,b)\}$ with $x,y\in B(\Gamma)$, $a\in \In(\Gamma)$, $b\in \Out(\Gamma)$ such that $b=\rho(a)$. 
\item A {\it broken edge} is a pair of edges $\{(x,a),(a,y)\}$ with $x,y\in B(\Gamma)$, $a\in W(\Gamma)$. 
\item A {\it broken separated edge} is either a triple $\{(a,b),(b,x),(y,c)\}$ or a triple $\{(a,x),(y,b),(b,c)\}$, with $x,y\in B(\Gamma)$, $a,b,c\in W(\Gamma)$ such that $c=\rho(a)$. 
\end{enumerate}
See Figure~\ref{fig:edges}. Let $\Comp(\Gamma)$, $\Se(\Gamma)$, $\Br(\Gamma)$, $\Se'(\Gamma)$ be the set of compact, separated, broken, broken separated edges respectively. Let $E(\Gamma)=\Comp(\Gamma)\cup\Se(\Gamma)\cup\Br(\Gamma)\cup\Se'(\Gamma)$.
\begin{figure}
\fig{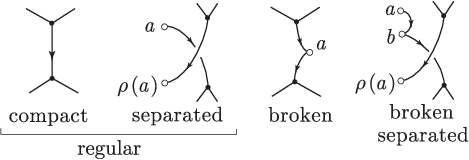}
\caption{}\label{fig:edges}
\end{figure}

A {\it labeled graph} is an admissible graph $\Gamma$ equipped with bijections $\alpha:\{1,2,\ldots,n\}\to B(\Gamma)$ and $\beta:\{1,2,\ldots,\ell\}\to E(\Gamma)$, where $n=|B(\Gamma)|$ and $\ell=|E(\Gamma)|$. Let $C_*^{(i)}=(C_*^{(i)},\partial^{(i)})$, $C_*^{(i)}=\Z^{P_*^{(i)}}$, $i=1,2,\ldots,\ell$, be a sequence of acyclic chain complexes over $\Z$ with finite bases. For a sequence $\vec{C}=(C_*^{(1)},\ldots,C_*^{(\ell)})$, we define a {\it $\vec{C}$-colored graph} as a labeled graph $\Gamma=(\Gamma,\alpha,\beta)$ such that on each white vertex of $\beta(i)$ a basis element $p\in P_*^{(i)}$ is attached for each $i$. Later we will substitute the Morse complex of a Morse pair to each $C_*^{(i)}$. Then $P_*^{(i)}$ will correspond to the set of critical points of a Morse function.

For each edge $e=\beta(i)\in E(\Gamma)$ in a $\vec{C}$-colored graph, we define its {\it degree} by
\[ \deg(e)=\left\{\begin{array}{ll}
	1 & \mbox{if $e\in\Comp(\Gamma)$}\\
	i(p)-i(q) & \mbox{if $e\in\Se(\Gamma)$, }\\
	0 & \mbox{if $e\in\Br(\Gamma)$}\\
	i(p)-i(q)-1 & \mbox{if $e\in \Se'(\Gamma)$,}\\
	\end{array}\right. \]
where $i(x)$ denotes the degree of $x$ and where $p\in P_*^{(i)}$ is on the input, $q\in P_*^{(i)}$ is on the output of $e$. We define the {\it degree} of a $\vec{C}$-colored graph by $\deg(\Gamma)=(\deg(\beta(1)),\ldots,\deg(\beta(\ell)))$. We will call a $\vec{C}$-colored graph with degree $\vec{\eta}=(\eta_1,\ldots,\eta_\ell)$, with $n$ black vertices and with $|\Comp(\Gamma)|+|\Se(\Gamma)|=m$ a {\it $\vec{C}$-colored graph of type $(n,m,\vec{\eta})$}. We define the {\it closure} $\widehat{\Gamma}$ of $\Gamma$ as the graph obtained from $\Gamma$ by identifying white vertices of each input/output pair $(a,\rho(a))$. An example of a $\vec{C}$-colored graph of type $(4,6,(1,1,1,1,1,1))$ is given in (\ref{eq:ex_graph}).

%%%%%%%%%%%%%%%%%%%%%%%%%%%%%%
\subsection{The space $\calG_{n,m,\vec{\eta}}(\vec{C})$}\label{ss:space_G}

Let $\calG_{n,m,\vec{\eta}}^0(\vec{C})$ be the set of pairs $(\Gamma,o)$, where
\begin{enumerate}
\item $\Gamma$ is a $\vec{C}$-colored graph of type $(n,m,\vec{\eta})$ with connected closure $\widehat{\Gamma}$,
\item $o$ is an orientation of the real vector space
\[ \R^{B(\Gamma)}
	\oplus\bigoplus_{e\in\Comp(\Gamma)}\R^{H(e)}, \]
where $H(e)=\{e_+,e_-\}$ is the two-element set of `half-edges', namely $e_-=\varphi^{-1}[0,\frac{1}{2}]$ and $e_+=\varphi^{-1}[\frac{1}{2},1]$ for an orientation preserving homeomorphism $\varphi:e\to [0,1]$.
\end{enumerate}
Let $\calG_{n,m,\vec{\eta}}(\vec{C})$ be the vector space over $\Q$ spanned by $\calG_{n,m,\vec{\eta}}^0(\vec{C})$, quotiented by the relation $(\Gamma,-o)=-(\Gamma,o)$. 
The bijection $\alpha$ and the edge orientation of a labelled graph $\Gamma$ define a canonical graph orientation $o(\Gamma)$, as
\begin{equation}\label{eq:graph_ori}
 o(\Gamma)=\alpha(1)\wedge\cdots\wedge\alpha(n)\wedge\bigwedge_{e\in \Comp(\Gamma)}(e_+\wedge e_-),
\end{equation}
where $e$ is oriented as $(e_-,e_+)$.

We denote by $\calG_{n,m}^0(\vec{C})$ the subset of $\calG_{n,m,(1,\ldots,1)}^0(\vec{C})$ consisting of graphs without bivalent vertices such that $\ell=m$, i.e. $\vec{C}=(C_*^{(1)},\ldots,C_*^{(m)})$. Let $\calG_{n,m}(\vec{C})$ be the span of $\calG_{n,m}^0(\vec{C})$ over $\Q$. Let $\calG_{n,m,(1,\ldots,1)}^\mathrm{comp,0}(\vec{C})$ be the subset of $\calG_{n,m,(1,\ldots,1)}^0(\vec{C})$ consisting of graphs with only compact edges. Since the sequence of complexes $\vec{C}$ is unnecessary to represent a graph in $\calG_{n,m,(1,\ldots,1)}^\mathrm{comp,0}(\vec{C})$, there are canonical bijections between $\calG_{n,m,(1,\ldots,1)}^\mathrm{comp,0}(\vec{C})$ for different sequences $\vec{C}$. Identifying $\calG_{n,m,(1,\ldots,1)}^\mathrm{comp,0}(\vec{C})$ for all $\vec{C}$ by the canonical bijections, we simply write 
\[\calG_{n,m}^0=\calG_{n,m,(1,\ldots,1)}^\mathrm{comp,0}(\vec{C}) \]
and we define $\calG_{n,m}$ to be the vector space over $\Q$ spanned by $\calG_{n,m}^0$. 

%%%%%%%%%%%%%%%%%%%%%%%%%%%%%%
\subsection{Assumption on Morse functions}\label{ss:morse_func}

We make an assumption on Morse functions, as in \cite{Les3}, \cite[\S{4.1}]{Sh}\footnote{In an earlier version of the present paper, we did not make such an assumption. But the referee pointed out that without this assumption, there may be some boundary strata in the trajectory spaces which may violate the invariance of $\widehat{Z}_{2k,3k}$. Considering a homology sphere with one point removed as the connected sum of $\R^d$ with a homology sphere is originally due to Kontsevich (\cite{Ko1}).}. Let $M$ be a $d$-dimensional homology sphere with a distinguished point $\infty_M\in M$. We consider $S^d$ as the one point compactification $\R^d\cup \{\infty\}$. Let $U_\infty$ be the open ball around $\infty$:
\[ U_\infty=\{x\in \R^d\,;\,\|x\|> R \}\cup\{\infty\}\subset S^d \]
for some large $R$. Fix a small open ball $U_\infty'\subset M$ including $\infty_M$ and a diffeomorphism $\varphi_\infty:U_\infty'\to U_\infty$ which sends $\infty_M$ to $\infty$. We consider a Morse function on $M_0=M-\{\infty_M\}$ and a Riemannian metric $\mu$ on $M_0$ that are {\it standard near $\infty_M$}. We say that a function $f:M_0\to \R$ is standard near $\infty_M$ if $f|_{U_\infty'-\{\infty_M\}}$ agrees with the pullback of a rank one linear map $\R^d\to \R$ by $\varphi_\infty$. Similarly, we say that a Riemannian metric $\mu$ on $M_0$ is standard near $\infty_M$ if the restriction of $\mu$ to $T(U_\infty'-\{\infty_M\})$ agrees with the pullback of the standard metric on $\R^d$ by $\varphi_\infty$. Let $C^r_{\varphi_\infty}(M_0)$ denote the subspace of $C^r(M_0)$ consisting of functions that are standard near $\infty_M$ with respect to $\varphi_\infty$.

\begin{Assum}\label{assum:infty}
Fix a sufficiently large integer $r>0$. Throughout this paper, a Morse function on $M_0$ is always a $C^r$ Morse function $f:M_0\to \R$ that is standard near $\infty_M$ and a Riemannian metric $\mu$ on $M_0$ is always a Riemannian metric on $M_0$ that is standard near $\infty_M$.
\end{Assum}

%%%%%%%%%%%%%%%%%%%%%%%%%%%%%%
\subsection{Fukaya's moduli space $\calM_\Gamma$}\label{ss:M_G}

Suppose given a sequence $\vec{f}=(f_1,f_2,\ldots,f_m)$ of Morse functions on $M_0$ and a Riemannian metric $\mu$ on $M_0$. Suppose that $(f_i,\mu)$ is Morse--Smale for each $i$, namely, all the intersections between the descending manifolds and the ascending manifolds are transversal. We choose an orientation $o(\calD_p(f_i))$ of $\calD_p(f_i)$ arbitrarily for each critical point $p$ of $f_i$ and orient $\calA_p(f_i)$ by $o(\calA_p(f_i))=*o(\calD_p(f_i))$ near $p$, where $*$ is the Hodge star operator. Let $C^{(i)}=(C^{(i)}=\Z^{P^{(i)}},\partial^{(i)})$ be the Morse complex associated to $(f_i,\mu)$, namely, $C^{(i)}$ is the free $\Z$-module generated by the (finite) set $P_*^{(i)}$ of critical points of $f_i$ and $\partial^{(i)}:C^{(i)}_{k+1}\to C^{(i)}_k$ is defined by
\[\begin{split}
	\partial^{(i)}p=\sum_{q\in P_k^{(i)}}\#\bbcalM{f_i}{p}{q}\cdot q,\quad \bbcalM{f_i}{p}{q}=(\calD_p(f_i)\pitchfork \calA_q(f_i))\pitchfork L_p,\\
\end{split}\]
where $L_p$ is a level surface of $f_i$ that lies just below the level of $p$ and $\bbcalM{f_i}{p}{q}$ is an oriented 0-manifold whose orientation is derived from those of $\calD_p(f_i)$ and $\calA_q(f_i)$. More precisely, $\calD_p(f_i)\pitchfork \calA_q(f_i)$ is a disjoint union of flow lines of $-\mathrm{grad}\,f_i$. At each point $b\in \bbcalM{f}{p}{q}$, the wedge $o^*_M(\calD_p(f_i))_b\wedge o^*_M(\calA_q(f_i))_b\in \bigwedge^{d-1}T^*_b L_p\subset \bigwedge^{d-1}T^*_b M$ defines a coorientation of the flow line passing through $b$ (see Appendix~\ref{s:ori} (\ref{eq:coori_int})). Hence there exists a sign $\ve_{f_i}(p,q)_b=\pm 1$ such that
\[ o^*_M(\calD_p(f_i))_b\wedge o^*_M(\calA_q(f_i))_b \sim \ve_{f_i}(p,q)_b \,\, \iota(-\mathrm{grad}\,f_i)o(M)_b. \]
The sign $\ve_{f_i}(p,q)_b$ does not depend on the choice of $L_p$. 

Then the incidence coefficient is defined by
\[ \#\bbcalM{f_i}{p}{q}=\sum_{b\in \bbcalM{f_i}{p}{q}} \ve_{f_i}(p,q)_b. \]
It is known that $(C^{(i)},\partial^{(i)})$ above is a chain complex called a Morse complex (e.g. \cite{Bo}, see also Corollary~\ref{cor:dd=0}). Moreover, $(C^{(i)},\partial^{(i)})$ is acyclic by Assumption~\ref{assum:infty}. We put $ \vec{C}=(C^{(1)},C^{(2)},\ldots,C^{(m)})$.

Before recalling a general definition of Fukaya's moduli space $\calM_\Gamma(\vec{f})$, we give an example. Consider the following graph.
\begin{equation}\label{eq:ex_graph}
 \Gamma=\fig{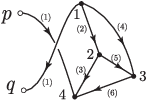}\in\calG_{4,6}^0(\vec{C}). 
\end{equation}
Let $\Phi_f^t:M_0\to M_0$, $t\in\R$, be the 1-parameter group of diffeomorphisms associated to the negative gradient $-\mathrm{grad}\,f$ considered with respect to a Riemannian metric $\mu$ on $M_0$. For $\vec{f}=(f_1,f_2,f_3,f_4,f_5,f_6)$, let $\calM_\Gamma(\vec{f})$ be the space of points $(x_1,x_2,x_3,x_4)\in M_0^4$ such that
\begin{enumerate}
\item there exist $t_2,t_3,t_4,t_5,t_6\in(0,\infty)$ such that $\Phi_{f_2}^{t_2}(x_1)=x_2$, $\Phi_{f_3}^{t_3}(x_2)=x_4$, $\Phi_{f_4}^{t_4}(x_1)=x_3$, $\Phi_{f_5}^{t_5}(x_2)=x_3$, $\Phi_{f_6}^{t_6}(x_3)=x_4$,
\item $\displaystyle\lim_{t\to -\infty}\Phi_{f_1}^{t}(x_4)=p$, $\displaystyle\lim_{t\to +\infty}\Phi_{f_1}^t(x_1)=q$ (or $x_4\in\calD_p(f_1)$, $x_1\in\calA_q(f_1)$).
\end{enumerate}

Now we give a general definition of $\calM_\Gamma(\vec{f})$, which is a straightforward generalization of the example above. For $\Gamma=(\Gamma,\alpha,\beta)\in\calG_{n,m}^0(\vec{C})$, we define the source and the target maps
\def\source{\sigma}
\def\target{\tau}
\[ \source:\{1,2,\ldots,m\}\to \{1,2,\ldots,n\},\quad \target:\{1,2,\ldots,m\}\to \{1,2,\ldots,n\} \]
as $\source(k)=\alpha^{-1}(\mbox{source of $\beta(k)$})$, $\target(k)=\alpha^{-1}(\mbox{target of $\beta(k)$})$. For each $i\in\{1,2,\ldots,n\}$, we define the subsets $\In_i(\Gamma)=\{k_1^i,k_2^i,\ldots,k_{r_i}^i\}$, $\In_i^\infty(\Gamma)=\{\bar{k}_1^i,\bar{k}_2^i,\ldots,\bar{k}_{\bar{r}_i}^i\}$, $\Out_i^\infty(\Gamma)=\{\bar{\ell}_1^i,\bar{\ell}_2^i,\ldots,\bar{\ell}_{\bar{s}_i}^i\}$ of the set of labels $\{1,2,\ldots, m\}$ of edges as the subsets consisting of labels of edges such that
\begin{align*}
 k_j^i\in \In_i(\Gamma)\Leftrightarrow \target(k_j^i)&=i
	\mbox{ and $\beta(k_j^i)\in\Comp(\Gamma)$},\\
 \bar{k}_j^i\in \In_i^\infty(\Gamma)\Leftrightarrow \target(\bar{k}_j^i)&=i
	\mbox{ and $\beta(\bar{k}_j^i)\in\Se(\Gamma)$},\\
 \bar{\ell}_j^i\in \Out_i^\infty(\Gamma)\Leftrightarrow\source(\bar{\ell}_j^i)&=i
	\mbox{ and $\beta(\bar{k}_j^i)\in\Se(\Gamma)$}.
\end{align*}
For example, $\In_i(\Gamma)$ is the subset of labels of incoming compact edges at the $i$-th vertex and $\In_i^\infty(\Gamma)$ is the subset of labels of incoming separated edges at the $i$-th vertex. See Figure~\ref{fig:in-out}.
\begin{figure}
\fig{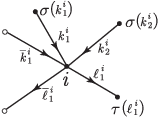}
\caption{Labels for edges incident to the $i$-th vertex. $\tau(k_1^i)=\tau(k_2^i)=i=\sigma(\ell_1^i)$.}\label{fig:in-out}
\end{figure}

\begin{Def}
For $\vec{f}=(f_1,\ldots,f_m)$ and $\Gamma$ without bivalent vertices, let $\calM_\Gamma(\vec{f})=\calM_\Gamma(\vec{f};\mu)$ be the space of points $(x_1,\ldots,x_n)\in M_0^n$ such that
\begin{enumerate}
\item for all $1\leq i\leq n$, $1\leq j\leq r_i$ such that $k_j^i\in \In_i(\Gamma)$, there exists $t_{k_j^i}\in(0,\infty)$ such that $\Phi_{f_{k_j^i}}^{t_{k_j^i}}(x_{\source(k_j^i)})=x_i$,

\item $\displaystyle\lim_{t\to -\infty}\Phi_{f_{\bar{k}_j^i}}^t(x_i)=p_{\bar{k}_j^i}$ for $1\leq i\leq n$, $1\leq j\leq \bar{r}_i$ such that $\bar{k}_j^i\in\In_i^\infty(\Gamma)$, where $p_{\bar{k}_j^i}\in P^{(\bar{k}_j^i)}$,

\item $\displaystyle\lim_{t\to +\infty}\Phi_{f_{\bar{\ell}_j^i}}^t(x_i)=q_{\bar{\ell}_j^i}$ for $1\leq i\leq n$, $1\leq j\leq \bar{s}_i$ such that $\bar{\ell}_j^i\in\Out_i^\infty(\Gamma)$, where $q_{\bar{\ell}_j^i}\in P^{(\bar{\ell}_j^i)}$.
\end{enumerate}
\end{Def}

\begin{Rem}\label{rem:cp-nonsing}
Since $\Phi_f^t(p)=p$ ($\forall t$) for a critical point $p$ of $f$, we allow for a point $(x_1,\ldots,x_n)$ of $\calM_\Gamma(\vec{f})$ that some $x_i$ coincides with a critical point of some $f_j$. We will see later that such a point is not a singular point of $\calM_\Gamma(\vec{f})$.
\end{Rem}

%%%%%%%%%%%%%%%%%%%%%%%%%%%%%%%%%%%%%%
\subsection{The count of $\calM_\Gamma$}\label{ss:count}

\begin{Prop}[page 49 of \cite{Fuk2}]\label{prop:M_mfd}
Suppose that $\Gamma\in\calG_{n,m,\vec{\eta}}^0(\vec{C})$ has no bivalent vertex, i.e. $E(\Gamma)=\Comp(\Gamma)\cup \Se(\Gamma)$. For a generic choice of $\vec{f}$, the space $\calM_\Gamma(\vec{f})$ is a $C^{r-1}$ smooth manifold of dimension $(n-m)d+\sum_{i=1}^m\eta_i$.
Moreover, $\vec{f}$ can be chosen so that this property is satisfied simultaneously for all graphs $\Gamma$ in $\calG_{n,m,{\vec{\eta}}}^0(\vec{C})$ for a fixed triple $m,n,\sum_{i=1}^m\eta_i$.
\end{Prop}
The proof of Proposition~\ref{prop:M_mfd} will be given in \S\ref{ss:transversality}. The reason for the dimension is roughly that an edge $e$ of degree $i(e)$ yields a $(d-i(e))$-dimensional constraint. Since $\dim\,M_0^n=nd$, the dimension of the moduli space should be $nd-\sum_{e\in E(\Gamma)}(d-i(e))=nd-md+\sum_ei(e)$. The reason for the class $C^{r-1}$ is that the solution for the differential equation $\dot{\gamma}(t)=-(\mathrm{grad}\,f)_{\gamma(t)}$ for a $C^r$ Morse function is of class $C^{r-1}$. 

As in \cite{Fuk2}, we will need a compactification of the moduli space $\calM_\Gamma(\vec{f})$ for $\Gamma$ with only trivalent black vertices. If $\Gamma$ has only trivalent black vertices and does not have bivalent vertices, then $n=2k$ and $m=3k$. For simplicity, we take a convenient metric for each Morse function. Namely, for $\vec{f}=(f_1,f_2,\ldots,f_{3k})$, we take a sequence $\vec{\mu}=(\mu_1,\mu_2,\ldots,\mu_{3k})$ of Riemannian metrics on $M_0$ such that for each $i$ the pair $(f_i,\mu_i)$ is Morse--Smale and that $\mu_i$ is Euclidean near $\Sigma(f_i)$ with respect to the coordinate of the Morse lemma. 

\begin{Prop}\label{prop:M_G-1-mfd}
Suppose $d=3$ and that $(\vec{f},\vec{\mu})$ is as above and is generic as in Proposition~\ref{prop:M_mfd}. Suppose that $\Gamma\in \calG_{2k,3k,\vec{\eta}}^0(\vec{C})$ is such that $|E(\Gamma)|=3k$ and such that $0\leq \dim{\calM_\Gamma(\vec{f})}=(\eta_1-1)+(\eta_2-1)+\cdots+(\eta_{3k}-1)\leq 1$. Then $\calM_\Gamma(\vec{f})$ has a natural compactification to a smooth manifold $\overline{\calM}_\Gamma(\vec{f})$ with boundary, whose boundary consists of flow graphs with a once broken trajectory or with a subgraph collapsed to a point.
\end{Prop}

The proof of Proposition~\ref{prop:M_G-1-mfd} will be given in \S\ref{ss:comp_fuk}. Proposition~\ref{prop:M_mfd} implies that for $\Gamma$ as in Proposition~\ref{prop:M_G-1-mfd} with $\eta_1=\cdots=\eta_{3k}=1$, we have $\dim\calM_\Gamma(\vec{f})=0$. In fact, $\bcalM_\Gamma(\vec{f})=\calM_\Gamma(\vec{f})$ in this case. We count points in the finite set $\calM_\Gamma(\vec{f})$ with signs as follows. Let $(x_1,\ldots,x_{2k})\in\calM_\Gamma(\vec{f})$. For each edge $e\in E(\Gamma)$, we assign a vector 
\[ v_e\in \textstyle \bigwedge^2 (T_xM_0\oplus T_yM_0), \]
where $x=x_{\sigma(i)}$, $y=x_{\tau(i)}$, $i=\beta^{-1}(e)$, as follows. 

If $e\in \Comp(\Gamma)$, let $e_1,e_2,e_3$ be an orthonormal basis of $T_xM_0$ such that $e_1\wedge e_2\wedge e_3$ gives the orientation of $M_0$ and $e_1$ is a positive multiple of $-(\mathrm{grad}\,f_i)_x$. There is $t_0>0$ such that $y=\Phi_{f_i}^{t_0}(x)$. The flow $\Phi_{f_i}^{t_0}$ induces a diffeomorphism from a neighborhood of $x$ to that of $y$. Let $e_i'=d\Phi_{f_i}^{t_0}(e_i)\in T_yM_0$ ($i=1,2,3$). Let 
\[ \begin{split}
  v_e&=n_2\wedge n_3,\,\,\mbox{where}\\
  n_2&=e_2+\frac{1}{\Delta}
  \left(
  \left|\begin{array}{cc} 
  a_{12} & a_{13} \\
  a_{32} & a_{33}
  \end{array}\right|e_1'
  -\left|\begin{array}{cc} 
  a_{11} & a_{13} \\
  a_{31} & a_{33}
  \end{array}\right|e_2'
  +\left|\begin{array}{cc} 
  a_{11} & a_{12} \\
  a_{31} & a_{32}
  \end{array}\right|e_3'
  \right),\\
  n_3&=e_3+\frac{1}{\Delta}
  \left(
  -\left|\begin{array}{cc} 
  a_{12} & a_{13} \\
  a_{22} & a_{23}
  \end{array}\right|e_1'
  +\left|\begin{array}{cc} 
  a_{11} & a_{13} \\
  a_{21} & a_{23}
  \end{array}\right|e_2'
  -\left|\begin{array}{cc} 
  a_{11} & a_{12} \\
  a_{21} & a_{22}
  \end{array}\right|e_3'
  \right),\\
  a_{ij}&=\langle e_i',e_j'\rangle,\quad \Delta=\det(a_{ij}).
\end{split}\]

If $e\in \Se(\Gamma)$, let $p$ and $q$ be the critical points of $f_i$ that are the input and the output of the $i$-th edge in the flow graph. Let $e_1,e_2,e_3$ be an orthonormal basis of $T_xM_0$ such that $T_x\calA_q(f_i)=\langle e_1,\ldots,e_r\rangle$, $T_x\calA_q(f_i)^{\perp}=\langle e_{r+1},\ldots,e_3\rangle$ and $e_1\wedge\cdots\wedge e_r$ and $e_1\wedge e_2\wedge e_3$ give the orientations of $\calA_q(f_i)$ and $M_0$ respectively. Similarly, let $e_1',e_2',e_3'$ be an orthonormal basis of $T_yM_0$ such that $T_y\calD_p(f_i)=\langle e_1',\ldots,e_s'\rangle$, $T_y\calD_p(f_i)^{\perp}=\langle e_{s+1}',\ldots,e_3'\rangle$ and $e_1'\wedge\cdots\wedge e_s'$ and $e_1'\wedge e_2'\wedge e_3'$ give the orientations of $\calD_p(f_i)$ and $M_0$ respectively. Then we define 
\[ v_e=(e_{r+1}\wedge\cdots\wedge e_3)\wedge (e_{s+1}'\wedge\cdots\wedge e_3'),\]
which belongs to $\bigwedge^2 T_{(x,y)}(M_0\times M_0)$ if $\eta_i=i(p)-i(q)=1$. 

Let $V(x_1,\ldots,x_{2k})=\bigwedge_{e\in E(\Gamma)} v_e \in \textstyle\bigwedge^{6k}T_{(x_1,\ldots,x_{2k})}(M_0^{2k})$. Since $\dim{M_0^{2k}}=6k$, there is a nonzero real number $\alpha$ such that $V(x_1,\ldots,x_{2k})=\alpha\, O_{x_1}\wedge \cdots\wedge O_{x_{2k}}$, where $O_{x_j}\in \bigwedge^3 T_{x_j}M_0$ gives the unit volume. For $(x_1,\ldots,x_{2k})\in \calM_\Gamma(\vec{f})$, we define
\[ \ve(x_1,\ldots,x_{2k})=\left\{\begin{array}{ll}
1 & \mbox{if $\alpha>0$}\\
-1 & \mbox{if $\alpha<0$}
\end{array}\right. \]
For a generic pair $(\vec{f},\vec{\eta})$ as in Proposition~\ref{prop:M_G-1-mfd}, the coefficient $\alpha$ is always nonzero for all points of $\calM_\Gamma(\vec{f})$. We define
\[ \#\calM_\Gamma(\vec{f})=\sum_{(x_1,\ldots,x_{2k})\in \calM_\Gamma(\vec{f})} \ve(x_1,\ldots,x_{2k})\in \Z. \]

%%%%%%%%%%%%%%%%%%%%%%%%%%%%%%
\subsection{Principal term $Z_{2k,3k}$}\label{ss:def_Z}

The space $\calG_{2k,3k}$ (\S\ref{ss:space_G}) is spanned by 3-valent graphs with only compact edges. Let $\calR_{2k,3k}\subset \calG_{2k,3k}$ be the subspace spanned by the {\it IHX relation} and the {\it label change relation}. The IHX relation is shown in Figure~\ref{fig:IHX} and the label change relation is generated by the following elements
\begin{enumerate}
\item $(\Gamma,o(\Gamma))+(\Gamma',o(\Gamma'))$ for labeled graphs $\Gamma$ and $\Gamma'$, where $\Gamma'$ is obtained from $\Gamma$ by a swap of a pair of vertex labels or by an inversion of the orientation of an edge. 
\item $(\Gamma,o(\Gamma))-(\Gamma',o(\Gamma'))$ for labeled graphs $\Gamma$ and $\Gamma'$, where $\Gamma'$ is obtained from $\Gamma$ by a swap of a pair of labels for compact edges. 
\end{enumerate}
We define the space $\calA_{2k,3k}$ to be the quotient space of $\calG_{2k,3k}$ by $\calR_{2k,3k}$. We denote by $[\Gamma]$ the equivalence class in $\calA_{2k,3k}$ represented by $\Gamma\in\calG^0_{2k,3k}$. 
\begin{figure}
\fig{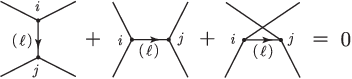}
\caption{The IHX relation. $i$ and $j$ are labels of vertices, $(\ell)$ is a label of an edge.}\label{fig:IHX}
\end{figure}

Let $\Gamma$ be a $\vec{C}$-labeled graph with $p_i\in P_*^{(i)}$ on the input and $q_i\in P_*^{(i)}$ on the output. Let $k_i=i(p_i)-i(q_i)$. For a sequence $\vec{h}=(h^{(1)},\ldots,h^{(m)})$ of degree $k_i$ endomorphisms $h^{(i)}\in\End_{k_i}(C_*^{(i)})$, $i=1,2,\ldots,m$, we define the trace of $\Gamma$ by
\[ \Tr_{\vec{h}}\left(\fig{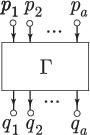}\right)
	=\prod_{i=1}^a(-h_{q_ip_i}^{(i)})\times \fig{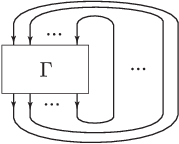}, \]
where $h^{(i)}q_i=\sum_{{{p_i\in P_*^{(i)}}\atop{i(p_i)=i(q_i)+k_i}}}h_{q_ip_i}^{(i)}p_i$. In particular, since each $(C_*^{(i)},\partial^{(i)})$ is acyclic, there exists an endomorphism $g^{(i)}:C_*^{(i)}\to C_{*+1}^{(i)}$ of homogeneous degree 1 such that $\partial^{(i)}g^{(i)}+g^{(i)}\partial^{(i)}=\mathrm{id}$. (See Appendix~\ref{s:acyclic_cpx}. Following \cite{Fuk2}, we call such an endomorphism a {\it combinatorial propagator}). As a special case of the above definition, the trace by combinatorial propagators $\vec{g}=(g^{(1)},\ldots,g^{(3k)})$ defines a linear map $\Tr_{\vec{g}}:\calG_{2k,3k}(\vec{C})\to \calA_{2k,3k}$.
\begin{Def}\label{def:Z}
We define
\[ Z_{2k,3k}(\vec{f})=\sum_{\Gamma\in\calG_{2k,3k}^0(\vec{C})}\#\calM_\Gamma(\vec{f})\,\Tr_{\vec{g}}(\Gamma)
	\in \calA_{2k,3k}, \]
where the sum is taken over all $\vec{C}$-colored graphs in $\calG_{2k,3k}^0(\vec{C})$, each equipped with canonical orientation. 
\end{Def}

%%%%%%%%%%%%%%%%%%%%%%%%%%%%%%
\subsection{Moduli space of infinitesimal flow graphs}\label{ss:infinitesimal}

In the rest of this section, we define the correction term which turns $Z_{2k,3k}$ into a topological invariant. Let $X$ be an oriented smooth Riemannian manifold and $\Gamma$ be a graph with only compact edges. Suppose that $\Gamma$ has $n$ vertices and $m$ edges. We shall consider the moduli space of linear flow graphs in an oriented linear $\R^3$-bundle $\pi:E\to X$ for such a graph $\Gamma$. Let $P\to X$ be the orthonormal $SO_3$-frame bundle associated to $\pi$ and
\[ \Conf_n^\loc(\R^3)=\Bigl\{(y_1,\ldots,y_n)\in(\R^3)^n;y_1=0,\sum_{\ell=2}^n\|y_\ell\|^2=1,y_i\neq y_j\mbox{ if }i\neq j\Bigr\}. \]
Let $\pi^\circ:E^\circ=P\times_{SO_3}(\R^3-\{0\})\to X$, $S(\pi):S(E)=P\times_{SO_3}S^2\to X$, $\Conf_n^\loc(\pi):\Conf_n^\loc(E)=P\times_{SO_3}\Conf_n^\loc(\R^3)\to X$ be the bundles associated to $\pi$.
Such a bundle appears in a boundary strata of compactified configuration space (see \S\ref{ss:FM}). The normalization $v\mapsto v/\|v\|$ induces a natural map $\nu:E^\circ \to S(E)$. The Gauss map $\phi_{ij}:\Conf_n^\loc(\R^3)\to S^2$, which takes $(y_1,\ldots,y_n)$ to $\frac{y_j-y_i}{\|y_j-y_i\|}$, induces a well-defined morphism $\widetilde{\phi}_{ij}:\Conf_n^\loc(E)\to S(E)$.

Now suppose that a section $\gamma:X\to E^\circ$ of $\pi^\circ$ is given. Then $\bar{\gamma}=\nu\circ \gamma:X\to S(E)$ is a section of $S(\pi)$. Since $\widetilde{\phi}_{ij}$ is transversal to $\bar{\gamma}(X)$ on each fiber, the subset
\[ \Theta_\ell(\gamma)=\widetilde{\phi}_{ij}^{-1}(\bar{\gamma}(X))\subset \Conf_n^\loc(E) \]
forms a smooth subbundle of $\Conf_n^\loc(\pi)$ where $\ell$ is such that $i=\sigma(\ell)$ and $j=\tau(\ell)$. 
\begin{Def}\label{def:Mloc}
For a sequence $\vec{\gamma}=(\gamma_1,\gamma_2,\ldots,\gamma_m)$ of sections of $\pi^\circ$, we define
\[
 \calM_\Gamma^\loc(\vec{\gamma})
	=\bigcap_{\ell=1}^m \Theta_\ell(\gamma_\ell)\subset \Conf_n^\loc(E)
\]
for a sequence $\vec{\gamma}=(\gamma_1,\ldots,\gamma_m)$ of sections of $\pi^\circ$. If the intersection is transversal, in other words, if $\bigwedge_{\ell=1}^m o^*_{\Conf_n^\loc(E)}(\Theta_\ell(\gamma_\ell))\neq 0$ at every point of $\calM_\Gamma^\loc(\vec{\gamma})$, this formula also defines a co-orientation of $\calM_\Gamma^\loc(\vec{\gamma})$. 
\end{Def}
There is a compactification $\bConf_n^\loc(\R^3)$ of $\Conf_n^\loc(\R^3)$, which is naturally an $SO_3$-space. See \S\ref{ss:FM}. Let $\bConf_n^\loc(\pi):\bConf_n^\loc(E)\to X$ be the $\bConf_n^\loc(\R^3)$-bundle associated to $\pi$. The interior of $\bConf_n^\loc(\R^3)$ is identified with $\Conf_n^\loc(\R^3)$. 
Let $\overline{\Theta}_\ell(\gamma)\subset \bConf_n^\loc(E)$ be the closure of $\Theta_\ell(\gamma)$. Let
\begin{equation}\label{eq:intersection}
 \bcalM_\Gamma^\loc(\vec{\gamma})=\bigcap_{\ell=1}^m \overline{\Theta}_\ell(\gamma_\ell). 
\end{equation}

\begin{Lem}\label{lem:M-local_mfd}
For a generic choice of $\vec{\gamma}$, the moduli space $\bcalM_\Gamma^\loc(\vec{\gamma})$ is a submanifold of $\bConf_n^\loc(E)$ of codimension $2m$. If $X$ is compact, then so is $\bcalM_\Gamma^\loc(\vec{\gamma})$.
\end{Lem}
\begin{proof}
Note that $\overline\Theta_{k}(\gamma)$ is a submanifold of codimension $2$. By using the transversality theorem, the set of sections $\vec{\gamma}$ can be inductively deformed in $\Gamma(\pi)^m$ slightly so that the intersection (\ref{eq:intersection}) is transversal. Thus for a generic choice of $\vec{\gamma}$, $\bcalM_\Gamma^\loc(\vec{\gamma})$ is a submanifold. The second assertion is immediate.
\end{proof}

When $\vec{\gamma}$ is generic as in Lemma~\ref{lem:M-local_mfd} and $X$ is compact and $\dim\bcalM_\Gamma^\loc(\vec{\gamma})=0$, we define $\#\bcalM_\Gamma^\mathrm{local}(\vec{\gamma})$ to be the number of components counted with signs, which are determined by the coorientations of the intersections. Here we fix the orientation $o(\bConf_n^\loc(\R^3))$ to be the one on the unit sphere induced from that of the Euclidean space $(\R^3)^{n-1}$. Then we orient $\bConf_n^\loc(E)$ by 
\[ o(\bConf_n^\loc(E))=o(X)\wedge o(\bConf_n^\loc(\R^3)). \]

%%%%%%%%%%%%%%%%%%%%%%%%%%%%%%%%%%%%%%%%%%
\subsection{Anomaly term $Z_{2k,3k}^\mathrm{anomaly}$}\label{ss:anomaly}

Here, we shall define the term $Z_{2k,3k}^\mathrm{anomaly}(\vec{\gamma}_W)$ for a sequence $\vec{\gamma}_W$ of sections of a vector bundle $T^vW$ over a spin 4-manifold $W$ with $\partial W\cong M$. To do this we shall first find a trivialization of $TW$ and consider its trivial subbundle $T^vW$. 

%%%%%
\subsubsection{Framing on spin cobordism}\label{ss:framing}

For a $k$-manifold $X$, a {\it framing} on $X$ is a trivialization $\tau_X:TX\to X\times \R^k$. More generally, we will also call a trivialization of a vector bundle a framing. We will identify a framing with a finite set of sections of a vector bundle that are fiberwise linearly independent. Here we shall fix framings on $M_0$ and on a spin 4-manifold $W$ with $\partial W=M$ in a sense compatible with each other. Recall that a spin structure on a vector bundle $E$ over a CW-complex $B$ is a homotopy class of framings on the 1-skeleton of $E$ which can be extended to the 2-skeleton (\cite{Mi1}). A spin structure on a tangent bundle of a manifold $X$ is called a spin structure on $X$. Since the group $\Omega_3^\mathrm{spin}$ of spin cobordism classes of spin 3-manifolds is trivial, one can find a compact spin 4-manifold $W$ with $\partial W=M$ and with a spin structure that is compatible with the (canonical) spin structure of $M$. 

We choose a framing $\tau_M$ on $TM_0$, which exists for any $M$. We fix $\tau_M$ such that it agrees on $U_\infty'-\{\infty_M\}$ with the pullback of the standard one $\tau_{\R^3}$ on $U_\infty-\{\infty\}$ by $d\varphi_\infty^{-1}$. One may check that such a framing really exists by the obstruction theory for extending sections. Let $\overline{U}_\infty$ be the closure of $U_\infty\subset S^3$ and let 
\[ \bM=(M-U_\infty')\cup_{\partial} ([0,1]\times \partial \overline{U}_\infty)\cup_\partial -(S^3-U_\infty), \]
where $\partial(M-U_\infty')$ is identified with $\{0\}\times \overline{U}_\infty$ by $\varphi_\infty$ and $\partial (S^3-U_\infty)=\partial\overline{U}_\infty$ is identified with $\{1\}\times \partial\overline{U}_\infty$. Then $\bM$ is diffeomorphic to $M$. We construct a rank 3 vector bundle $T^v\bM$ on $\bM$ as follows. Consider $[0,1]\times \partial \overline{U}_\infty$ as a part of $[0,1]\times \overline{U}_\infty$. Let $T^v([0,1]\times\overline{U}_\infty)$ be the pullback of $T\overline{U}_\infty$ by the projection $[0,1]\times \overline{U}_\infty\to \overline{U}_\infty$. Let $T^v([0,1]\times\partial\overline{U}_\infty)$ be the restriction of $T^v([0,1]\times\overline{U}_\infty)$ to $[0,1]\times\partial\overline{U}_\infty$. We define
\[ T^v\bM = T(M-U_\infty')\cup T^v([0,1]\times\partial\overline{U}_\infty)\cup T(-(S^3-U_\infty)). \]
The rank 4 vector bundle $T([0,1]\times \overline{U}_\infty)$ restricts on $\{0,1\}\times\partial\overline{U}_\infty$ to the restrictions of $\ve^1\oplus TM$ and $\ve^1\oplus T(-S^3)$, where $\ve^1$ denotes the trivial line bundle. Thus by extending $T([0,1]\times \overline{U}_\infty)|_{[0,1]\times\partial \overline{U}_\infty}$ by the restrictions of $\ve^1\oplus T(M-U_\infty')$ and $\ve^1\oplus T(-(S^3-U_\infty))$, we obtain a $\R^4$-bundle over $\bM$ of the form $\ve^1\oplus T^v\bM$. 

Let $n$ be a framing of $\ve^1$. The 4-framings $n\oplus \tau_M$ and $n\oplus \tau_{\R^3}$ of $\ve^1\oplus T(M-U_\infty')$ and $\ve^1\oplus T(-(S^3-U_\infty))$ respectively extend over $\ve^1\oplus T^v\bM$ by using the product structure. We denote by $\tau_M'$ the resulting 4-framing of $\ve^1\oplus T^v\bM$. 

The following lemma follows from Lemma~2.3 of \cite{KM}, Lemma~2.40 of \cite{Les1} and from the proof of \cite[Theorem~2.6]{KM}.
\begin{Lem}\label{lem:KM}
\begin{enumerate}
\item There exists a compact spin 4-manifold $W$ with corners with $\partial W=\bM$ as the spin boundary such that $\chi(W)=1$. 
\item Let $W$ be as in (1). The 4-framing $\tau_M'$ extends to a framing of $TW$ if and only if $p_1(TW;\tau_M')=0$, where $p_1(TW;\tau_M')\in \Z$ denotes the relative Pontrjagin number. Moreover, there exists a framing $\tau_M$ of $M_0$ that is standard near $\infty_M$ and that satisfies $p_1(TW;\tau_M')=0$. 
\end{enumerate}
\end{Lem}

%%%%%
\subsubsection{Generalized Morse sections}\label{ss:H-bundle}

Let $\pi:E\to X$ be a linear $\R^d$-bundle over a compact manifold $X$ possibly with corners with $\dim{X}=N\geq d$. We say that a smooth section $\gamma:X\to E$ is {\it generalized Morse (GM)} if for each point $x\in \gamma^{-1}(0)$, there is a local coordinate $(y_1,\ldots,y_N)$ around $x$ on an open neighborhood $U$ of $x$ and a trivialization $\varphi:\pi^{-1}(U)\to U\times \R^d$ such that either of the following holds\footnote{This condition is not a generic one if $N\geq d+2$. Thus this gives a stronger restriction than the transversality to the zero section. This restriction is placed to determine all the singularities.}.
\[ \begin{split}
  (1)\quad \gamma(y_1,\ldots,y_N)&=\varphi^{-1}(y_1,\ldots,y_N,\pm y_1,\ldots,\pm y_d)\quad \forall (y_1,\ldots,y_N)\in U\\
  (2)\quad \gamma(y_1,\ldots,y_N)&=\varphi^{-1}(y_1,\ldots,y_N,y_1^2-y_{d+1},\pm y_2,\ldots,\pm y_d)\quad \forall (y_1,\ldots,y_N)\in U
\end{split}\]
When $\gamma$ is GM, we call a point $x\in \gamma^{-1}(0)$ having local form (1) (resp. (2)) a {\it Morse singularity} (resp. {\it birth-death singularity}) of $\gamma$. We write $\Sigma(\gamma)=\gamma^{-1}(0)$ and let $\Sigma^1(\gamma)$ (resp. $\Sigma^2(\gamma)$) be the subset of $\Sigma(\gamma)$ consisting of Morse singularities (resp. birth-death singularities). An obvious example with $\Sigma^2(\gamma)=\emptyset$ is the section $M_0\to TM_0$ given by the gradient of a Morse function. The following lemma is an immediate consequence of results of K. Igusa (\cite[Lemma~2.8]{Ig1} and \cite[Appendix~2]{Ig2}).

\begin{Lem}\label{lem:GM_section}
Let $\pi:E\to X$ be as above. Suppose that the restriction of a smooth section $\gamma:X\to E$ to $\partial X$ is GM. Then there is a homotopy of $\gamma$ relative to $\partial X$ whose result is GM. Hence $\Sigma^2(\gamma)$ is a codimension 1 submanifold of $\Sigma(\gamma)$. 
\end{Lem}

%%%%%
\subsubsection{Definition of $Z_{2k,3k}^\mathrm{anomaly}$}\label{ss:def_anomaly}

Now let $(W,\tau_M')$ be a pair satisfying the condition of Lemma~\ref{lem:KM}. One can find a (non-unique) 4-framing of $TW$ and let $\tau_W^v$ be its sub 3-framing $\tau_W^v$ of $TW$ that extends $\tau_M$. The 3-framing $\tau_W^v$ spans a rank 3 subbundle $T^vW$ of $TW$. For each $i\in\{1,2,\ldots,3k\}$, let $\gamma_i$ be a GM section of $T^vW$ extending $-\mathrm{grad}\,f_i\in \Gamma(T(M-U_\infty'))$ and put $\vec{\gamma}_W=(\gamma_1,\ldots,\gamma_{3k})$. 
\begin{Def}\label{def:Za}
We define
\[ Z_{2k,3k}^\mathrm{anomaly}(\vec{\gamma}_W)=\sum_{\Gamma\in\calG_{2k,3k}}\#\calM_{\Gamma}^\loc(\vec{\gamma}_W)\,[\Gamma]\in\calA_{2k,3k}, \]
where the moduli space $\calM_{\Gamma}^\loc(\vec{\gamma}_W)$ is considered inside the trivial $\bConf_{2k}^\loc(\R^3)$-bundle over $\bigcap_{j=1}^{3k}(W-\Sigma(\gamma_j))$ associated to the restriction of the $\R^3$-bundle $T^vW$. 
\end{Def}

\begin{Prop}\label{prop:anomaly_welldefined}
\begin{enumerate}
\item For a generic choice of the GM extension $\vec{\gamma}_W$ of $-\mathrm{grad}\,\vec{f}=(-\mathrm{grad}\,f_1,\ldots,-\mathrm{grad}\,f_{3k})$, the number $\#\calM_\Gamma^\loc(\vec{\gamma}_W)$ is finite.

\item Let $W$ and $W'$ be compact, connected, spin 4-manifolds with corners with $\partial W=\partial W'=\bM$, $\chi(W)=\chi(W')=1$ and suppose that $\vec{\gamma}_W|_{\bM}=\vec{\gamma}_{W'}|_{\bM}$. Then for each $k\geq 1$ there exists a constant $\mu_k\in\calA_{2k,3k}$ such that
\[ Z_{2k,3k}^\mathrm{anomaly}(\vec{\gamma}_W)-\mu_k\,\mathrm{sign}\,W
	=Z_{2k,3k}^\mathrm{anomaly}(\vec{\gamma}_{W'})-\mu_k\,\mathrm{sign}\,W'. \]
Hence $Z_{2k,3k}^\mathrm{anomaly}(\vec{\gamma}_W)-\mu_k\,\mathrm{sign}\,W$ does not depend on the choice of $(W,\vec{\gamma}_W)$ such that $\partial W=\bM$, $\chi(W)=1$, $\vec{\gamma}_W|_{M-U_\infty'}=-\mathrm{grad}\,\vec{f}$.
\end{enumerate}
\end{Prop}
\begin{proof}[Proof of Proposition~\ref{prop:anomaly_welldefined} (1)]
Put $\vec{\gamma}=(\gamma_1,\ldots,\gamma_{3k})=\vec{\gamma}_W$. Since for the GM extension $\gamma_i$ the singularity set $\Sigma(\gamma_i)=\gamma_i^{-1}(0)$ is a compact smooth 1-submanifold of a 4-manifold $W$, we may assume that $\Sigma(\gamma_i)\cap\Sigma(\gamma_j)=\emptyset$ if $i\neq j$, and moreover that they are separated by small open tubular neighborhoods. We shall show that the projection of the 0-dimensional moduli space $\calM_\Gamma^\loc(\vec{\gamma})$ on $W$ is disjoint from a neighborhood of $\Sigma(\gamma_i)$ for each $i$ and hence from a neighborhood of $\coprod_{j=1}^{3k}\Sigma(\gamma_j)$. 

Let $\Gamma'$ be the graph obtained from $\Gamma$ by replacing $E(\Gamma)$ with $E(\Gamma)-\{\beta(i)\}$. According to Definition~\ref{def:Mloc}, $\calM_{\Gamma}^\loc(\vec{\gamma})$ is the intersection of $\calM_{\Gamma'}^\loc(\vec{\gamma}\setminus\{\gamma_i\})$ with $\Theta_{i}(\gamma_i)$. By Lemma~\ref{lem:M-local_mfd}, $\calM_{\Gamma'}^\loc(\vec{\gamma}\setminus\{\gamma_i\})$ is a submanifold of $(W-\bigcup_j\Sigma(\gamma_j))\times\bConf_{2k}^\loc(\R^3)$ of codimension $2(3k-1)=6k-2$, i.e., a 2-submanifold if $\vec{\gamma}$ is generic. For a generic choice of $\gamma_i$, the projection of $\calM_{\Gamma'}^\loc(\vec{\gamma}\setminus\{\gamma_i\})$ on $W$ is disjoint from a neighborhood of $\Sigma(\gamma_i)$ for a dimensional reason. Hence for a generic choice of $\gamma_i$, the projection of $\calM_\Gamma^\loc(\vec{\gamma})$ on $W$ is disjoint from a neighborhood of $\Sigma(\gamma_i)$. Here we may assume that the perturbation of $\gamma_i$ for the disjunction has support in an arbitrarily small neighborhood of $\Sigma(\gamma_i)$. Since $\Sigma(\gamma_i)\cap \Sigma(\gamma_j)=\emptyset$ for $i\neq j$, the perturbations can be done for all $i$ independently and we may assume that $\calM_\Gamma^\loc(\vec{\gamma})$ is disjoint from a tubular neighborhood of $\coprod_{j=1}^{3k}\Sigma(\gamma_j)$. 

By Lemma~\ref{lem:M-local_mfd}, the restriction of $\calM_\Gamma^\loc(\vec{\gamma})$ to the complement of the tubular neighborhood of $\coprod_{j=1}^{3k}\Sigma(\gamma_j)$ is compact. Therefore, for a generic choice of $\vec{\gamma}$, $\calM_\Gamma^\loc(\vec{\gamma})$ is a compact 0-submanifold, i.e., a finite set.
\end{proof}
The proof of Proposition~\ref{prop:anomaly_welldefined} (2) will be given in \S\ref{ss:correctionterm}.

\subsection{Main result and conjectures}
\begin{Thm}\label{thm:Z_invariant}
For $k\geq 1$, 
\[ \widehat{Z}_{2k,3k}(\vec{f})
	=Z_{2k,3k}(\vec{f})-Z_{2k,3k}^\mathrm{anomaly}(\vec{\gamma}_W)+\mu_k\,\mathrm{sign}\,W\in\calA_{2k,3k}, \]
where $\mu_k$ is the constant found in Proposition~\ref{prop:anomaly_welldefined} (2), is an invariant of diffeomorphism type of $M$. 
\end{Thm}

Proof of the theorem is given in \S\ref{s:proof_main}. Theorem~\ref{thm:Z_invariant} allows us to write $\widehat{Z}_{2k,3k}(M)=\widehat{Z}_{2k,3k}(\vec{f})$. As mentioned in the introduction and the concluding remarks of \cite{Fuk2}, the 2-loop part $Z_{2,3}(M)$ is likely to coincide with the 2-loop part of the configuration space integral of Kontsevich. The generalization of this conjecture is the following, which can be considered as a higher loop analogue of a theorem of Cheeger \cite{Ch} and M\"{u}ller \cite{Mu}.
\begin{Conj}\label{conj:1}
$\widehat{Z}_{2k,3k}(M)$ agrees with (Kuperberg--Thurston's universal expression (\cite{KT}) of) the configuration space integral invariant of Kontsevich (\cite{Ko1}).
\end{Conj}
It is known that the configuration space integral invariant of Kontsevich recovers all $\Q$-valued Ohtsuki finite type invariants (\cite{Oh,KT,Les2}). Hence it is highly nontrivial. Shortly after the author proposed Conjecture~\ref{conj:1} in an earlier version of this article, T.~Shimizu gave a proof of Conjecture~\ref{conj:1} (\cite{Sh}). Shimizu also found an explicit relation of the constant $\mu_k$ to a constant $\delta_k$ considered in \cite{KT, Les2} for configuration space integrals.

The following conjecture is a corollary of this and Conjecture~\ref{conj:1}. 

\begin{Cor}[of Conjecture~\ref{conj:1}]
$\widehat{Z}_{2k,3k}(M)$ is nontrivial. Furthermore, the sequence $\{\widehat{Z}_{2k,3k}(M)\}_k$ recovers all Ohtsuki's finite type invariants (\cite{Oh}) over $\Q$.
\end{Cor}

There are analogues of our graph complex and the moduli spaces of flow graphs for circle-valued Morse theory (\cite{No, Pa}). The generalization to circle-valued Morse function would give an invariant of 3-manifolds with the first Betti number 1. We plan to discuss this in a future paper \cite{Wa2}.

\begin{Rem}
(1) C.~Lescop independently constructed in collaboration with G.~Kuperberg (\cite{Les3}) an explicit 4-chain in the configuration space of two points in a rational homology 3-sphere $M$ by a geometric consideration about Heegaard diagrams, which is reminiscent of Heegaard Floer homology. She defined an invariant of `combings' on $M$ using the explicit 4-chain and gave a combinatorial formula for the invariant. It seems that Fukaya's spaces of gradient trajectories are also included in their 4-chain. 

(2) M.~Futaki discovered in \cite{Fut} some singular phenomena that are missed in \cite{Fuk2}. In \cite{Fuk2}, the coefficients in the linear combination of graphs are defined by contracting holonomies considered along flow graphs by $\mathfrak{g}$-invariant tensors. However, Futaki observed with a concrete computation that when the dumbbell graph contribution is nontrivial (homology 3-sphere with the trivial connection is not the case), the holonomy matrix will suddenly jump at the point on which a trivalent vertex passes through a critical point and thus the invariance fails. Since we construct an invariant via an intersection theory considering only the trivial connection contribution, the coefficients in the linear combination in our definition can be given without using holonomy matrix, so the same problem does not occur. (See also Remark~\ref{rem:cp-nonsing}.)
\end{Rem}

%\clearpage
%%%%%%%%%%%%%%%%%%%%%%%%%%%%%%
%%%%%%%%%%%%%%%%%%%%%%%%%%%%%%
\mysection{Moduli space of gradient trajectories}{s:M2}

We shall construct a compactification $\bcalM_2(f)$ of the space $\calM_2(f)$ of gradient trajectories that corresponds to a compact edge in a graph, in a fashion similar to \cite{BH}. The compactification $\bcalM_2(f)$ will play a fundamental role in defining the compactification $\bcalM_\Gamma(\vec{f})$. For a Morse function $f$ and a metric $\mu$ on $M_0$, we define
\[ \calM_2(f)=\calM_2(f;\mu)=\{(x,y)\in (M_0-\Sigma(f))^2;y=\Phi_f^t(x)\mbox{ for some }t\in(0,\infty)\}. \]
It follows from a property of solutions of ordinary differential equations that $\calM_2(f)$ is a submanifold of $(M_0-\Sigma(f))^2$ of dimension $d+1$. We shall construct a natural compactification of $\calM_2(f)$. Moreover, we obtain compactifications of $\calD_p(f)$ and $\calA_p(f)$ by using the compactification of $\calM_2(f)$.

%%%%%%%%%%%%%%%%%%%%%
\subsection{A decomposition of $\calM_2(f)$}\label{sss:covering_M2}

First we make some assumptions. In the following we assume that a Morse function $f$ is chosen as in the following lemma.

\begin{Lem}[e.g. Lemma~2.8 of \cite{Mi2}]\label{lem:ordered}
For any $C^r$ Morse function $f:M_0\to \R$ that is standard near $\infty_M$, there is an arbitrarily $C^r$-small perturbation of $f$ in the subspace of $C^r_{\varphi_\infty}(M_0)$ of Morse functions such that all the critical values of the resulting Morse function are distinct. (Such a Morse function is said to be \emph{ordered}.)
\end{Lem}

The Morse lemma gives a local coordinate description of the moduli space. Let $f$ be a Morse function on $M_0$. By the Morse lemma, one can find a local coordinate $(x_1,\ldots,x_d)$ on a neighborhood $M_p$ of a critical point $p$ of $f$ and a metric $\mu$ on $M_0$ such that $f$ agrees on $M_p$ with 
\begin{equation}\label{eq:std_quad}
 h(x)=f(p)-\frac{x_1^2}{2}-\cdots-\frac{x_i^2}{2}+\frac{x_{i+1}^2}{2}+\cdots+\frac{x_d^2}{2} 
\end{equation}
and such that $\mu$ agrees with the Euclidean metric on $M_p$ with respect to the coordinate $(x_1,\ldots,x_d)$. We say that such a metric $\mu$ is {\it Euclidean near critical points}. We call a pair of $M_p$ and the coordinate $(x_1,\ldots,x_d)$ a {\it Morse model}.

Suppose that the singular set $\Sigma(f)=\{p_1,p_2,\ldots,p_N\}$ is numbered so that $f(p_k)<f(p_{k+1})$ for each $k\leq N-1$. We put $c_k=f(p_k)$. For a small number $\eta>0$ and $1\leq k\leq N-1$, let 
\[  \begin{split}
  W_k&=f^{-1}[c_{k}-\eta,c_{k+1}-\eta]\cup\{\infty_M\},\quad L_k=f^{-1}(c_k-\eta)\cup\{\infty_M\},\\
  W_N&=f^{-1}[c_N-\eta,\infty)\cup\{\infty_M\},\quad L_N=f^{-1}(c_N-\eta)\cup\{\infty_M\},\\
  W_0&=f^{-1}(-\infty,c_1-\eta]\cup\{\infty_M\}.
\end{split} \]
See Figure~\ref{fig:Wk}. For a pair of subsets $A,B$ of $M$, let $\calM_2(f;A,B)=\calM_2(f)\cap(A\times B)$.
\begin{figure}
\cfig{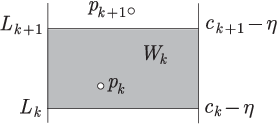}
\caption{}\label{fig:Wk}
\end{figure}
Then we have
\[ \calM_2(f)=\bigcup_{0\leq j\leq k\leq N}\calM_2(f;W_k,W_j). \]
\begin{figure}
\fig{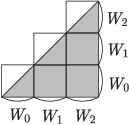}
\caption{A schematic illustration for a covering of $\calM_2(f)$. This consists of 6 squares each corresponds to $\calM_2(f;W_k,W_j)$.}\label{fig:Wk_x_Wj}
\end{figure}

For $0\leq j\leq k\leq N$, there is a natural embedding
\[ \psi_{kj}:\calM_2(f;W_k,W_j)\to W_k\times L_k\times L_{k-1}\times\cdots
\times L_{j+1}\times W_j, \]
defined by $\psi_{kj}(x,y)=(x,z_k,z_{k-1},\ldots,z_{j+1},y)$, where $z_i\in L_i$ is the unique intersection point of the flow line between $x$ and $y$ with $L_i$. Then $\calM_2(f)$ is canonically diffeomorphic to the union of the images $\psi_{kj}(\calM_2(f;W_k,W_j))$ ($0\leq j\leq k\leq N$) glued together by the diffeomorphisms
\[ \begin{split}
\psi_{kj}\circ \psi_{k+1,j}^{-1}&:
\psi_{k+1,j}(\calM_2(f;L_{k+1},W_j))\to \psi_{kj}(\calM_2(f;L_{k+1},W_j)),\\
\psi_{kj}\circ \psi_{k,j-1}^{-1}&:
\psi_{k,j-1}(\calM_2(f;W_k,L_j))\to \psi_{kj}(\calM_2(f;W_k,L_j)).
\end{split}
\]
See Figure~\ref{fig:Wk_x_Wj}. Note that $\psi_{kj}\circ \psi_{k+1,j}^{-1}$ and $\psi_{kj}\circ \psi_{k,j-1}^{-1}$ agree with the maps induced from the projections
\[ \begin{split}
\pi_{kj}:W_{k+1}\times L_{k+1}\times L_k\times \cdots\times L_{j+1}\times W_j
&\to W_k\times L_k\times L_{k-1}\times\cdots\times L_{j+1}\times W_j,\\
(x,z_{k+1},z_k,\ldots,z_{j+1},y)&\mapsto (x,z_k,z_{k-1},\ldots,z_{j+1},y),\\
\rho_{kj}:W_k\times L_k\times L_{k-1}\times \cdots \times L_{j}\times W_{j-1}
&\to W_k\times L_k\times L_{k-1}\times \cdots \times L_{j+1}\times W_j,\\
(x,z_{k},z_{k-1},\ldots,z_{j},y)&\mapsto (x,z_k,z_{k-1},\ldots,z_{j+1},y).
\end{split}\]

%%%%%%%%%%%%%%%%%%
\subsection{The definition of $\bcalM_2(f)$}\label{sss:def_M2}

Let 
\begin{equation}\label{eq:M2(k,j)}
 \bcalM_2(f;W_k,W_j)=\overline{\psi_{kj}(\calM_2(f;W_k,W_j))}\quad\mbox{(the closure)}. 
\end{equation}
Note that it is not the closure of $\calM_2(f;W_k,W_j)$ in $W_k\times W_j$ when $k>j$, but the closure in the codomain of $\psi_{kj}$. 

\begin{Lem}\label{lem:pi_rho_diffeo}
The maps $\pi_{kj}$ and $\rho_{kj}$ induce diffeomorphisms
\[ \begin{split}
&\overline{\psi_{k+1,j}(\calM_2(f;L_{k+1},W_j))}
\to \overline{\psi_{kj}(\calM_2(f;L_{k+1},W_j))},\\
&\overline{\psi_{k,j-1}(\calM_2(f;W_k,L_j))}
\to \overline{\psi_{kj}(\calM_2(f;W_k,L_j))}.
\end{split} \]
\end{Lem}
\begin{proof}
We only give a proof for $\pi_{kj}$. The smoothness of $\pi_{kj}$ is obvious. Define a smooth map $\gamma:L_{k+1}\times L_k\times\cdots\times L_{j+1}\times W_j\to L_{k+1}\times L_{k+1}\times L_k\times\cdots\times L_{j+1}\times W_j$ by $\gamma(x,z_k,\ldots,z_{j+1},y)=(x,x,z_k,\ldots,z_{j+1},y)$. The restriction of $\gamma$ to\\ $\overline{\psi_{k+1,j}(\calM_2(f;L_{k+1},W_j))}$ is a smooth inverse to $\pi_{kj}$.
\end{proof}
\begin{Def}
We define
\begin{equation}\label{eq:def_bcalM}
 \bcalM_2(f)=\bigcup_{0\leq j\leq k\leq N}\bcalM_2(f;W_k,W_j), 
\end{equation}
where the pieces are glued together by the diffeomorphisms of Lemma~\ref{lem:pi_rho_diffeo}. 
\end{Def}
It is clear from the definition that $\bcalM_2(f)$ is compact. Let
\[ \bar{b}:\bcalM_2(f)\to M\times M\]
be the continuous map that extends the natural embedding $b=\bigcup_{k,j}\psi_{kj}^{-1}:\bigcup_{j,k}\psi_{kj}(\calM_2(f;W_k,W_j))\to M\times M$ onto $\calM_2(f)$. In other words, $\bar{b}$ gives the pair of endpoints of a (possibly broken) flow line. For subsets $A$ of $W_k$ and $B$ of $W_j$, let
\begin{equation}\label{eq:bcalM(A,B)}
 \bcalM_2(f;A,B)=\overline{\psi_{kj}(\calM_2(f;A,B))}\subset A\times L_k\times\cdots \times L_{j+1}\times B. 
\end{equation}
This is consistent with (\ref{eq:M2(k,j)}). Note that this may depend on the choices of $k$ and $j$ when $A\subset L_k$ or $B\subset L_{j+1}$, but it becomes well-defined if it is considered as a subspace of $\bcalM_2(f)$. 

For a Morse pair $(f,\mu)$ and a pair $(x,y)$ of distinct points of $M-\Sigma(f)$, a {\it ($r$ times) broken flow line between $x$ and $y$} is a sequence $\gamma_0,\gamma_1,\ldots,\gamma_r$ ($r\geq 1$) of integral curves of $-\mathrm{grad}\,f$ satisfying the following conditions:
\begin{enumerate}
\item The domain of $\gamma_0$ is $[0,\infty)$, the domain of $\gamma_r$ is $(-\infty,0]$ and the domain of $\gamma_i$, $1\leq i\leq r-1$, is $\R$.
\item $\gamma_0(0)=x$, $\gamma_r(0)=y$.
\item There is a sequence $q_1,q_2,\ldots,q_r$ of distinct critical points of $f$ such that $\lim_{s\to -\infty}\gamma_i(s)=\lim_{s\to \infty}\gamma_{i-1}(s)=q_i$ ($1\leq i\leq r$).
\end{enumerate}
A broken flow line $(\gamma_0,\gamma_1,\ldots,\gamma_r)$ between $x$ and $y$ is determined by the boundary points $x,y$ and intersection points of $\gamma_i$ with level surfaces that lie between $q_i$ and $q_{i+1}$. More precisely, a broken flow line between $x\in W_k$ and $y\in W_j$ is uniquely determined by a point of $W_k\times L_k\times \cdots\times L_{j+1}\times W_j$ up to reparametrizations and conversely a broken flow line between $x\in W_k$ and $y\in W_j$ determines a point of $W_k\times L_k\times \cdots\times L_{j+1}\times W_j$. So we may identify a broken flow line between $x\in W_k$ and $y\in W_j$ with a point of $W_k\times L_k\times \cdots\times L_{j+1}\times W_j$ and call the latter a {\it broken flow sequence}.

Now the main proposition of this subsection can be stated as follows\footnote{We will not give explicit charts on every strata. The article \cite{We} of K.~Wehrheim gives a full description of the smooth structures on the compactification of $\calM_2(f)$ and explicit associative gluing maps in a similar finite dimensional fashion as \cite{BH}. Most of the results on the compactification of $\calM_2(f)$ given below would follow from results in \cite{We}.
}. 
\begin{Prop}\label{prop:bM2_mfd}
Let $(f,\mu)$ be a Morse--Smale pair such that $f$ is ordered and $\mu$ is Euclidean near critical points. Let $\Omega_M=(M\times\infty_M)\cup (\infty_M\times M)$, $\widehat{\Delta}_M=\Delta_M\cup \Omega_M$. Then $\bcalM_2(f)$ in (\ref{eq:def_bcalM}) is compact and satisfies the following conditions.
\begin{enumerate}
\item $\bcalM_2(f)-\bar{b}^{-1}(\wDelta_M)$ is a smooth manifold with corners.
\item $\bar{b}$ induces a diffeomorphism $\mathrm{Int}\,\bcalM_2(f)\to \calM_2(f)$.
\item The codimension $r$ stratum of $\bcalM_2(f)-\bar{b}^{-1}(\wDelta_M)$ consists of $r$ times broken flow sequences. The codimension $r$ stratum of $\bcalM_2(f)-\bar{b}^{-1}(\wDelta_M)$ for $r\geq 1$ is canonically diffeomorphic to
\[ \coprod_{{{q_1,\ldots,q_r\in\Sigma(f)}\atop{q_1,\ldots,q_r\,\mathrm{distinct}}}}\calA_{q_1}(f)\times \bbcalM{f}{q_1}{q_2}\times\cdots\times\bbcalM{f}{q_{r-1}}{q_r}\times\calD_{q_r}(f). \]
\end{enumerate}
\end{Prop}

The proof is divided into \S\ref{ss:mod_short}, \S\ref{ss:mod_long} and \S\ref{ss:mod_gen}.

\begin{Rem}\label{rem:M2_nomfd}
\begin{enumerate}
\item The compactification $\bcalM_2(f)$ is not a submanifold of $M\times M$ whereas $\calM_2(f)$ is a submanifold of $M\times M$. Moreover, the image of $\bar{b}$ may not be a submanifold of $M\times M$ with corners since the dimensions of some faces of the boundary decreases. 
\item In fact, $\bcalM_2(f)$ is smooth on $\bcalM_2(f)-\bar{b}^{-1}(\wDelta_{\Sigma(f)})$, where $\wDelta_{\Sigma(f)}=\{(p,p)\in M\times M; p\in\Sigma(f)\cup\{\infty_M\}\}$. The boundary of $\bcalM_2(f)$ has conic singularities on $\bar{b}^{-1}(\wDelta_{\Sigma(f)})$.
\item The definition of $\bcalM_2(f)$ depends on the choice of the level surfaces $L_k$. But its diffeomorphism type (as a manifold with corners) does not depend on the choice and it is enough for our purpose.
\end{enumerate}
\end{Rem}

%%%%%%%%%%%%%%%%%%%%%%
\subsection{Smooth structure of the moduli space of short trajectories}\label{ss:mod_short}

Let $h$ be the standard quadratic form of (\ref{eq:std_quad}). First, we describe the standard model
\[ \calM_2(h)=\{(x,y)\in(\R^d)^2 \,;\, y=\Phi_h^t(x)\mbox{ for some }t\in (0,\infty)\}. \]
The following lemma is a key lemma in the construction of the compactification.

\begin{Lem}\label{lem:M2(h)}
$\calM_2(h)=\{(\rho u,v)\times (u,\rho v);\,u\in\R^i,v\in\R^{d-i},\rho\in(0,1)\}$. Hence its closure $\bcalM_2(h)$ in $\R^d\times \R^d$ is
\[ \bcalM_2(h)=\{(\rho u,v)\times (u,\rho v);\,u\in\R^i,v\in\R^{d-i},\rho\in[0,1]\} \]
and $\bcalM_2(h)-\{0\times 0\}$ is a smooth manifold with boundary, with
\[  \partial \bcalM_2(h)=(\{0\}\times \R^{d-i})\times (\R^i\times\{0\}) \cup_{0\times 0}\Delta_{\R^d}=(\calA_0(h)\times \calD_0(h)) \cup_{0\times 0}\Delta_{\R^d}.\] 
\end{Lem}
\begin{proof}
Let $\calX=\{(\rho u,v)\times (u,\rho v);\,u\in\R^i,v\in\R^{d-i},\rho\in(0,1)\}$. Suppose that $(\rho u,v)\times (u,\rho v)\in \calX$. The solution for the differential equation
\[ \dot{\gamma}(t)=-(\mathrm{grad}\,h)_{\gamma(t)} \] 
is $\gamma(t)=(\gamma_1(0)e^t,\ldots,\gamma_i(0)e^t,\gamma_{i+1}(0)e^{-t},\ldots,\gamma_d(0)e^{-t})$. If $\gamma(0)=(\rho u,v)$, then $\gamma(t)=(\rho u e^t,v e^{-t})$. The system of equations $\rho u e^t=u$, $v e^{-t}=\rho v$ has a unique solution $t\geq 0$ proveded that $(u,v)\neq (0,0)$, in which case $(\rho u,v)\times (u,\rho v)\in \calM_2(h)$. If $(u,v)=(0,0)$, then $(\rho u,v)\times (u,\rho v)=(0,0)\times (0,0)$ and obviously this belongs $\calM_2(h)$. Conversely, for $(u_0,v_0)\times (u_0e^t,v_0e^{-t})\in \calM_2(h)$ ($t\geq 0$), put $u=u_0e^t$ and $v=v_0$. Then $(u_0,v_0)\times (u_0e^t,v_0e^{-t})=(ue^{-t},v)\times (u,ve^{-t})\in \calX$. This completes the proof of $\calX=\calM_2(h)$. 

For the latter assertion, consider the smooth map $\varphi:[0,1]\times \R^i\times \R^{d-i}\to \R^d\times \R^d$ defined by $\varphi(\rho,u,v)=(\rho u,v)\times (u,\rho v)$. Its Jacobian matrix is
\begin{equation}\label{eq:jacobian}
 J\varphi_{(\rho,u,v)}=\left(\begin{array}{ccc}
  u & \rho I & O\\
  \mathbf{0} & O & I\\
  \mathbf{0} & I & O\\
  v & O & \rho I
\end{array}\right) 
\end{equation}
whose rank is $d+1$ unless $(u,v)=(0,0)$. Namely, $\varphi$ is an immersion outside $[0,1]\times 0\times 0$. Note that $\varphi([0,1]\times 0\times 0)=\{0\times 0\}$. Moreover, it is easy to check that $\bcalM_2(h)-\{0\times 0\}$ is a submanifold with boundary. The boundary corresponds to the image from $\rho=0,1$.  
\end{proof}

\begin{Lem}\label{lem:dM(W,W)}
Let $(f,\mu)$ be as in Proposition~\ref{prop:bM2_mfd} and let $1\leq k\leq N-1$. Then
%\begin{enumerate}
%\item [(i)] 

(i) $\bcalM_2(f;W_k,W_k)-\wDelta_{W_k}$ ($\wDelta_{W_k}=W_k^2\cap \wDelta_M$) is a submanifold of $W_k\times W_k$ with corners, with
\[ \begin{split}
\partial\bcalM_2(f;W_k,W_k)=&\Bigl[(\calA_{p_k}(f)\cap W_k)\times (\calD_{p_k}(f)\cap W_k)\Bigr]\cup_{(p_k,p_k)} \Delta_{W_k}\\
&\cup \calM_2(f;W_k,L_k)\cup \calM_2(f;L_{k+1},W_k). 
\end{split}\]

%\item[(ii)] 
(ii) $\bcalM_2(f;W_k,L_k)-\{\infty_M^2\}$ is a submanifold of $W_k\times L_k$ with corners, with
\[ \begin{split}
	\partial \bcalM_2(f;W_k,L_k)=&\Bigl[(\calA_{p_k}(f)\cap W_k)\times (\calD_{p_k}(f)\cap L_k)\Bigr]\\
        &\cup \calM_2(f;L_{k+1},L_k)\cup \Delta_{L_k}.
\end{split}
\]

%\item[(iii)] 
(iii) $\bcalM_2(f;L_{k+1},L_k)-\{\infty_M^2\}$ is a submanifold of $L_{k+1}\times L_k$ with corners, with
\[ \begin{split}
	\partial \bcalM_2(f;L_{k+1},L_k)=&(\calA_{p_k}(f)\cap L_{k+1})\times (\calD_{p_k}(f)\cap L_k)
\end{split}
\]

%\item[(iv)] 
(iv) $\bcalM_2(f;L_{k+1},W_k)-\{\infty_M^2\}$ is a submanifold of $L_{k+1}\times W_k$ with corners, with
\[ \begin{split}
	\partial \bcalM_2(f;L_{k+1},W_k)=&\Bigl[(\calA_{p_k}(f)\cap L_{k+1})\times (\calD_{p_k}(f)\cap W_k)\Bigr]\\
        &\cup \calM_2(f;L_{k+1},L_k)\cup \Delta_{L_{k+1}}.
\end{split}
\]
%\end{enumerate} 
\end{Lem}
\begin{proof}
Here we only prove (i). The other cases are the restrictions of this case. The part $\calM_2(f;W_k,L_k)\cup \calM_2(f;L_{k+1},W_k)$ is the boundary of $\calM_2(f;W_k,W_k)$. To see the other ends, we choose a covering $\calU=\{U_\lambda\}$ of $W_k-\{\infty_M\}$ by small open subsets $U_\lambda$ each of which is the intersection of an open disk in $M$ and $W_k-\{\infty_M\}$. We check the assertion for $\bcalM_2(f;U_\lambda,U_\mu)$ for any $\lambda,\mu$. We choose $U_\lambda$ so small that for each $\lambda$ one of the following holds. 
\begin{enumerate}
\item $U_\lambda$ is disjoint from $\calA_{p_k}(f)\cup \calD_{p_k}(f)$.
\item $U_\lambda$ is included in a neighborhood $M_{p_k}\subset \mathrm{Int}\,W_k$ of $p_k$ of the Morse lemma. 
\item the gradient translation $\Phi_f^T(U_\lambda)$ for some $T\in\R$ is included in $M_{p_k}$. 
\end{enumerate}
The three lemmas \ref{lem:case1}, \ref{lem:case2} and \ref{lem:case3} below complete the proof. 
\end{proof}

\begin{Lem}[Case (1)]\label{lem:case1} Let $(f,\mu)$ be as in Proposition~\ref{prop:bM2_mfd}. Suppose that either $U_\lambda$ or $U_\mu$ is disjoint from $\calA_{p_k}(f)\cup \calD_{p_k}(f)$. Then
\[ \bcalM_2(f;U_\lambda,U_\mu)=\calM_2(f;U_\lambda,U_\mu). \]
\end{Lem}

\begin{proof}
We may assume without loss of generality that $U_\lambda$ is disjoint from $\calA_{p_k}(f)\cup \calD_{p_k}(f)$. Let $(x,y)\in U_\lambda\times U_\mu$ be any point such that $(x,y)\not\in \calM_2(f;U_\lambda,U_\mu)$. Since $U_\lambda$ is disjoint from $\calA_{p_k}(f)\cup \calD_{p_k}(f)$, the integral curve $\gamma_x$ that passes through $x$ intersects both $L_{k+1}$ and $L_k$. For a small number $\ve>0$, let 
\[ \widetilde{U}_\ve=W_k\cap \bigcup_{t\in\R}\Phi_f^t(U_\ve(x)), \]
where $U_\ve(x)$ is the open $\ve$-ball around $x$. If $\ve$ is sufficiently small, $\widetilde{U}_\ve$ is an open tubular neighborhood of $\mathrm{Im}\,\gamma_x$ in $W_k$. Since $\mathrm{Im}\,\gamma_x$ and $\{y\}$ are disjoint, these are separated by $\widetilde{U}_\ve$ and $U_\ve(y)$ for some $\ve$. This shows that the open set $U_\ve(x)\times U_\ve(y)\subset W_k\times W_k$ is disjoint from $\calM_2(f;U_\lambda,U_\mu)$ and that $\calM_2(f;U_\lambda,U_\mu)$ is a relatively closed subset of $U_\lambda\times U_\mu$, whose closure in $U_\lambda\times U_\mu$ is itself.
\end{proof}

\begin{Lem}[Case (2)]\label{lem:case2} Let $(f,\mu)$ be as in Proposition~\ref{prop:bM2_mfd}. Suppose that both $U_\lambda$ and $U_\mu$ are included in $M_{p_k}$. Then $\bcalM_2(f;U_\lambda,U_\mu)-\{p_k\times p_k\}$ is a smooth manifold with boundary, with
\[
 \partial \bcalM_2(f;U_\lambda,U_\mu)
=\partial\bcalM_2(h)\cap (U_\lambda\times U_\mu),
\]
where we identify $f$ with $h$ in a Morse model.
\end{Lem}
\begin{proof}
By definition, $\calM_2(f;U_\lambda,U_\mu)=\calM_2(h)\cap (U_\lambda\times U_\mu)$. Then apply Lemma~\ref{lem:M2(h)} to obtain the closure. 
\end{proof}

\begin{Lem}[Case (3)]\label{lem:case3} Let $(f,\mu)$ be as in Proposition~\ref{prop:bM2_mfd}. Suppose that $U_\lambda\cap U_\mu=\emptyset$ and that there are real numbers $S,T$ such that their gradient translations $U_\lambda'=\Phi_f^{S}(U_\lambda)$ and $U_\mu'=\Phi_f^T(U_\mu)$ are both included in $M_{p_k}$. Then $\bcalM_2(f;U_\lambda,U_\mu)$ is a smooth manifold with boundary, with
\[ \partial \bcalM_2(f;U_\lambda,U_\mu)\cong \partial\bcalM_2(h)\cap (U_\lambda'\times U_\mu'), \]
where the diffeomorphism is given by $\Phi_f^{S}\times \Phi_f^{T}:U_\lambda\times U_\mu\to U_\lambda'\times U_\mu'$. 
\end{Lem}

\begin{proof}
The diffeomorphism $\Phi_f^S\times \Phi_f^T$ induces a diffeomorphism between $\calM_2(f;U_\lambda,U_\mu)$ and $\calM_2(f;U_\lambda',U_\mu')$ and their closures.
\end{proof}

Let $\calD_{\infty}(f)=\{x\in M;\displaystyle\lim_{t\to -\infty}\Phi_f^t(x)=\infty_M\}$, $\calA_{\infty}(f)=\{x\in M;\displaystyle\lim_{t\to \infty}\Phi_f^t(x)=\infty_M\}$. Then $\calM_2(f;\infty_M,W_k)=\infty_M\times (\calD_\infty(f)\cap W_k)$, $\calM_2(f;W_k,\infty_M)=(\calA_\infty(f)\cap W_k)\times\infty_M$. The following lemma is an analogue of Lemma~\ref{lem:dM(W,W)}.

\begin{Lem}\label{lem:dM(infty,W)}
Let $(f,\mu)$ be as in Proposition~\ref{prop:bM2_mfd}. Then
%\begin{enumerate}
%\item [(i)] 

(i) $\bcalM_2(f;W_N,W_N)-\wDelta_{W_N}$ is a submanifold of $W_N\times W_N$ with corners, with
\[ \begin{split}
\partial\bcalM_2(f;W_N,W_N)=&\Bigl[(\calA_{p_N}(f)\cap W_N)\times (\calD_{p_N}(f)\cap W_N)\Bigr]\cup_{(p_N,p_N)} \Delta_{W_N}\\
&\cup \calM_2(f;W_N,L_N)\cup \calM_2(f;\infty_M,W_N). 
\end{split}\]

%\item[(ii)] 
(ii) $\bcalM_2(f;W_N,L_N)-\wDelta_{L_N}$ is a submanifold of $W_N\times L_N$ with corners, with
\[ \begin{split}
	\partial \bcalM_2(f;W_N,L_N)=&\Bigl[(\calA_{p_N}(f)\cap W_N)\times (\calD_{p_N}(f)\cap L_N)\Bigr]\\
        &\cup \calM_2(f;\infty_M,L_N)\cup \Delta_{L_N}.
\end{split}
\]

%\item [(i)] 
(iii) $\bcalM_2(f;W_0,W_0)-\wDelta_{W_0}$ is a submanifold of $W_0\times W_0$ with corners, with
\[ \begin{split}
\partial\bcalM_2(f;W_0,W_0)=\calM_2(f;W_0,\infty_M)\cup \calM_2(f;L_1,W_0). 
\end{split}\]

%\item[(ii)] 
(iv) $\bcalM_2(f;L_1,W_0)-\wDelta_{L_1}$ is a submanifold of $L_1\times W_0$ with corners, with
\[ \begin{split}
	\partial \bcalM_2(f;L_1,W_0)=\calM_2(f;L_1,\infty_M)\cup \Delta_{L_1}.
\end{split}
\]
%\end{enumerate} 
\end{Lem}

%%%%%%%%%%%%%%%%%%%%%%%%%%%
\subsection{Smooth structure of the moduli space of long trajectories}\label{ss:mod_long}

Next, we shall prove the following lemma.

\begin{Lem}\label{lem:CM_2(W,W)}
Let $(f,\mu)$ be as in Proposition~\ref{prop:bM2_mfd} and suppose that $f$ has $N$ critical points whose critical values are all distinct. Then $\bcalM_2(f;W_k,W_j)-\bar{b}^{-1}(\Omega_M)$ ($0\leq j< k\leq N$, definition in (\ref{eq:M2(k,j)})) is a submanifold of $W_k\times L_k\times L_{k-1}\times\cdots\times L_{j+1}\times W_j$ with corners, whose codimension $r$ stratum for $r\geq 1$ consists of $r-s$ times broken flow sequences $\xi$ with $s$ events in the following list happening.
\begin{itemize}
\item The initial endpoint of $\xi$ lies in $\partial W_k$.
\item The terminal endpoint of $\xi$ lies in $\partial W_j$.
\item The initial endpoint of $\xi$ agrees with $\infty_M$ (only if $k=N$).
\item The terminal endpoint of $\xi$ agrees with $\infty_M$ (only if $j=0$).
\end{itemize}
\end{Lem}

In the following, we follow convention in Appendix~\ref{s:mfd_corners} about smooth manifolds with corners. To prove Lemma~\ref{lem:CM_2(W,W)}, we shall prove the following lemma by induction on $k-j-1$.

\begin{Lem}\label{lem:induction}
Under the assumption of Lemma~\ref{lem:CM_2(W,W)}, for $k-j-1\geq 0$, the moduli space $\bcalM_2(f;W_k,L_{j+1})-\bar{b}^{-1}(\Omega_M)$ is a submanifold of $W_k\times L_k\times L_{k-1}\times\cdots\times L_{j+1}$ with corners, whose codimension $r$ stratum for $r\geq 1$ consists of $r-s$ times broken flow sequences  $\xi$ with $s$ events in the following list happening.
\begin{itemize}
\item The initial endpoint of $\xi$ lies in $\partial W_k$.
\item The initial endpoint of $\xi$ agrees with $\infty_M$ (if $k=N$).
\end{itemize}
\end{Lem}

 For $k-j-1=0$, Lemma~\ref{lem:induction} has been proved in Lemma~\ref{lem:dM(W,W)}. Let us consider the case $k-j-1=1$, i.e. $\bcalM_2(f;W_k,L_{k-1})$. The moduli space $\calM_2(f;W_k,L_{k-1})$ is identified with the fiber product $\calM_2(f;W_k,L_k)\times_{L_k}\calM_2(f;L_k,L_{k-1})$ that is the limit (pullback) of the diagram:
\[ \calM_2(f;L_k,L_{k-1})\stackrel{i_1}{\longrightarrow}L_k \stackrel{i_2}{\longleftarrow}\calM_2(f;W_k,L_k), \]
where $i_2:\calM_2(f;W_k,L_k)\to L_k$ and $i_1:\calM_2(f;L_k,L_{k-1})\to L_k$ are maps induced from projections $\mathrm{pr}_2:W_k\times L_k\to L_k$ and $\mathrm{pr}_1:L_k\times L_{k-1}\to L_k$ respectively. It is easy to see that $i_2$ and $i_1$ are transversal and hence by Proposition~\ref{prop:BT} the fiber product is a smooth manifold with boundary. 
\begin{Lem}\label{lem:glue_M2}
Let $(f,\mu)$ be as in Proposition~\ref{prop:bM2_mfd}. Then the smooth extensions $\bar{i}_2:\bcalM_2(f;W_k,L_k)\to L_k$ and $\bar{i}_1:\bcalM_2(f;L_k,L_{k-1})\to L_k$ of the projections $i_2$ and $i_1$ respectively are strata transversal on the complement of $\bar{b}^{-1}(\Omega_M)$. Hence the complement of $\infty_M\times L_k\times L_k\times L_{k-1}$ in the fiber product
\[ \bcalM_2(f;W_k,L_k)\times_{L_k}\bcalM_2(f;L_k,L_{k-1})\subset W_k\times L_k\times L_k\times L_{k-1}\]
is a smooth manifold with corners, whose strata are as follows.
\begin{enumerate}
\item[(0)] The codimension 0 stratum is $\calM_2(f;\mathrm{Int}\,W_k,L_k)\times_{L_k}\calM_2(f;L_k,L_{k-1})$.
\item[(1)] The codimension 1 stratum is the union of $\partial_1 \bcalM_2(f;W_k,L_k)\times_{L_k}\calM_2(f;L_k,L_{k-1})$ and $\calM_2(f;\mathrm{Int}\,W_k,L_k)\times_{L_k}\partial_1\bcalM_2(f;L_k,L_{k-1})$, where $\partial_r$ denotes the codimension $r$ stratum of the complement of $\bar{b}^{-1}(\Omega_M)$.
\item[(2)] The codimension 2 stratum is $\partial_1\bcalM_2(f;W_k,L_k)\times_{L_k}\partial_1\bcalM_2(f;L_k,L_{k-1})$.
\end{enumerate}
\end{Lem}
\begin{proof}
If either $z_k\in i_2(\calM_2(f;W_k,L_k))$ or $z_k\in i_1(\calM_2(f;L_k,L_{k-1}))$, then $z_k$ is a regular value of one of $i_2$ and $i_1$. Indeed, if for example $z_k=i_2(x,z_k)\in i_2(\calM_2(f;W_k,L_k))$, then there is a small open neighborhood $O$ of $z_k$ in $L_k$ such that $T_{z_k}O$ and the tangent space of the gradient line at $x$ parametrizes $T_{(x,z_k)}\calM_2(f;W_k,L_k)$. Obviously, $di_2:T_{(x,z_k)}\calM_2(f;W_k,L_k)\to T_{z_k}L_k=T_{z_k}O$ is surjective. This shows that $\bar{i}_2$ and $\bar{i}_1$ are transversal between a codimension 0 stratum and any strata. 

If $z_k\in \bar{i}_2(\partial \bcalM_2(f;W_k,L_k)-\calM_2(f;W_k,L_k))$ and $z_k\in \bar{i}_1(\partial \bcalM_2(f;L_k,L_{k-1}))$, then the images of the normal bundles of $\bar{i}_2^{-1}(z_k)$ in $\partial\bcalM_2(f;W_k,L_k)$ and of $\bar{i}_1^{-1}(z_k)$ in $\partial\bcalM_2(f;L_k,L_{k-1})$ under the differentials $d\bar{i}_2$ and $d\bar{i}_1$ agree with $T_{z_k}(\calD_{p_k}(f)\cap L_k)$ and $T_{z_k}(\calA_{p_{k-1}}(f)\cap L_k)$ respectively. Then by the Morse--Smale condition for $(f,\mu)$, these images span $T_{z_k}L_k$. This shows that $\bar{i}_2$ and $\bar{i}_1$ are transversal between codimension 1 strata. Now the lemma follows by applying Proposition~\ref{prop:BT}.
\end{proof}

The following lemma proves Lemma~\ref{lem:induction} for $k-j-1=1$.
\begin{Lem}\label{lem:hyp1}
Let $(f,\mu)$ be as in Proposition~\ref{prop:bM2_mfd}. Then 
\begin{equation}\label{eq:prMM}
 \bcalM_2(f;W_k,L_{k-1})=\mathrm{pr}\Bigl[
	\bcalM_2(f;W_k,L_k)\times_{L_k} \bcalM_2(f;L_k,L_{k-1})\Bigr],
\end{equation}
where $\mathrm{pr}:\bcalM_2(f;W_k,L_k)\times_{L_k} \bcalM_2(f;L_k,L_{k-1})\to W_k\times L_k\times L_{k-1}$ is the embedding $(x,z_k,z_k,z_{k-1})\mapsto (x,z_k,z_{k-1})$. Hence $\bcalM_2(f;W_k,L_{k-1})-\bar{b}^{-1}(\Omega_M)$ is a smooth manifold with corners, whose strata are as follows.
\begin{enumerate}
\item[(0)] The codimension 0 stratum is $\mathrm{pr}\Bigl[\calM_2(f;\mathrm{Int}\,W_k,L_k)\times_{L_k}\calM_2(f;L_k,L_{k-1})\Bigr]$.
\item[(1)] The codimension 1 stratum is the union of 
\[ \begin{split}
&\mathrm{pr}\Bigl[\partial_1 \bcalM_2(f;W_k,L_k)\times_{L_k}\calM_2(f;L_k,L_{k-1})\Bigr]\mbox{ and }\\
&\mathrm{pr}\Bigl[\calM_2(f;\mathrm{Int}\,W_k,L_k)\times_{L_k}\partial_1\bcalM_2(f;L_k,L_{k-1})\Bigr].
\end{split}\]
\item[(2)] The codimension 2 stratum is $\mathrm{pr}\Bigl[\partial_1\bcalM_2(f;W_k,L_k)\times_{L_k}\partial_1\bcalM_2(f;L_k,L_{k-1})\Bigr]$.
\end{enumerate}
\end{Lem}
\begin{proof}
We first show that $\mathrm{pr}$ takes $\bcalM_2(f;W_k,L_k)\times_{L_k} \bcalM_2(f;L_k,L_{k-1})$ diffeomorphically onto its image. But it is analogous to the proof of Lemma~\ref{lem:pi_rho_diffeo}, namely, the map $\mathrm{pr}$ is smooth and there is a smooth section $\gamma:W_k\times L_k\times L_{k-1}\to W_k\times L_k\times L_k\times L_{k-1}$ of $\mathrm{pr}$. 

Since $\bcalM_2(f;W_k,L_{k-1})$ is the closure of $\psi_{k,k-1}(\calM_2(f;W_k,L_{k-1}))$ in $W_k\times L_k\times L_{k-1}$ and since $\gamma$ gives a homeomorphism $W_k\times L_k\times L_{k-1}\approx W_k\times\Delta_{L_j}\times L_{k-1}$, it suffices to show that the closure of $\gamma(\psi_{k,k-1}(\calM_2(f;W_k,L_{k-1})))=\calM_2(f;W_k,L_k)\times_{L_k} \calM_2(f;L_k,L_{k-1})$ agrees with $\bcalM_2(f;W_k,L_k)\times_{L_k} \bcalM_2(f;L_k,L_{k-1})$ to see (\ref{eq:prMM}). This follows from Proposition~\ref{prop:closure_corners}. 
\end{proof}

The following lemma completes the induction and proves Lemma~\ref{lem:induction}.

\begin{Lem}\label{lem:hyp_p}
Under the assumption of Lemma~\ref{lem:CM_2(W,W)}, suppose that Lemma~\ref{lem:induction} holds true for $k-j-1=p\leq N-3$. Then Lemma~\ref{lem:induction} holds true for $k-j-1=p+1$.
\end{Lem}
\begin{proof}
By assumption, the moduli space $\bcalM_2(f;W_k,L_{j+1})-\bar{b}^{-1}(\Omega_M)$ is a smooth manifold with corners, whose strata are as described in Lemma~\ref{lem:induction}. Then by exactly the same argument as in Lemmas~\ref{lem:glue_M2} and \ref{lem:hyp1}, one may see the following.

(1) By Proposition~\ref{prop:BT}, the complement of $\bar{b}^{-1}(\Omega_M)\times L_{j+1}\times L_j$ in the fiber product $\bcalM_2(f;W_k,L_{j+1})\times_{L_{j+1}}\bcalM_2(f;L_{j+1},L_j)\subset(W_k\times L_k\times L_{k-1}\times\cdots\times L_{j+1})\times (L_{j+1}\times L_j)$ has the structure of a smooth manifold with corners, whose codimension $r$ stratum is the union of $\partial_r\bcalM_2(f;W_k,L_{j+1})\times_{L_{j+1}}\calM_2(f;L_{j+1},L_j)$ and $\partial_{r-1}\bcalM_2(f;W_k,L_{j+1})\times_{L_{j+1}}\partial_1\bcalM_2(f;L_{j+1},L_j)$.

(2) By Proposition~\ref{prop:closure_corners}, 
\[ \bcalM_2(f;W_k,L_j)=\mathrm{pr}\Bigl[\bcalM_2(f;W_k,L_{j+1})\times_{L_{j+1}}\bcalM_2(f;L_{j+1},L_j)\Bigr] \]
where $\mathrm{pr}:(W_k\times L_k\times L_{k-1}\times\cdots\times L_{j+1})\times (L_{j+1}\times L_j)\to W_k\times L_k\times L_{k-1}\times\cdots\times L_{j+1}\times L_j$ is the projection $(x,z_k,z_{k-1},\ldots,z_{j+1},z_{j+1},z_j)\mapsto (x,z_k,z_{k-1},\ldots,z_{j+1},z_j)$, which embeds $\bcalM_2(f;W_k,L_{j+1})\times_{L_{j+1}}\bcalM_2(f;L_{j+1},L_j)$.

These observations complete the proof.
\end{proof}

\begin{proof}[Proof of Lemma~\ref{lem:CM_2(W,W)}]
By replacing $\bcalM_2(f;L_{j+1},L_j)$ in the proof of Lemma~\ref{lem:hyp_p} with $\bcalM_2(f;L_{j+1},W_j)$, one may see by Proposition~\ref{prop:closure_corners} that $\bcalM_2(f;W_k,W_j)$ agrees with the projection of the fiber product $\bcalM_2(f;W_k,L_{j+1})\times_{L_{j+1}}\bcalM_2(f;L_{j+1},W_j)$ whose complement of $\bar{b}^{-1}(\Omega_M)$ is a smooth manifold with corners as desired.
\end{proof}

%%%%%%%%%%%%%%%%%
\subsection{Moduli space of general trajectories}\label{ss:mod_gen}

\begin{proof}[Proof of Proposition~\ref{prop:bM2_mfd}]
  Now we know from Lemma~\ref{lem:CM_2(W,W)} that $\bcalM_2(f)$ is the union of moduli spaces $\bcalM_2(f;W_k,W_j)$ ($0\leq j\leq k\leq N$) that are smooth manifolds with corners, glued together by diffeomorphisms of Lemma~\ref{lem:pi_rho_diffeo}. The result is, away from $\bar{b}^{-1}(\widehat{\Omega}_M)$, a smooth manifold with corners (see Lemma~\ref{lem:dM(W,W)} for the reason of exclusion of the diagonal). This proves the property (1). The property (2) is immediate from the definition of $\bcalM_2(f)$.

Since the diffeomorphisms of Lemma~\ref{lem:pi_rho_diffeo} are strata preserving (Appendix~\ref{s:mfd_corners}) in both directions, no new corners will appear under the gluing. The diffeomorphisms induce gluings between strata of the same codimensions and of the same type. For example, the component of $r$ times broken flow sequences in $\bcalM_2(f,W_k,W_j)$ is glued together along $\bcalM_2(f;L_{k+1},W_j)$ with the component of $r$ times broken flow sequences in $\bcalM_2(f;W_{k+1},W_j)$. This proves the property (3). 
\end{proof}

%%%%%
\subsection{Compactifications of descending and ascending manifolds}

Let $(f,\mu)$ be a Morse pair as in Proposition~\ref{prop:bM2_mfd}. For a critical point $p$ of $f$, let
\[ \bcalD_p(f)=\bar{b}^{-1}(p\times M),\quad \bcalA_p(f)=\bar{b}^{-1}(M\times p). \]
We obtain the following well-known result (e.g., \cite[Theorem~1]{BH}).

\begin{Prop}\label{prop:compactification_D}
Let $(f,\mu)$ be as in Proposition~\ref{prop:bM2_mfd} and let $p$ be a critical point of $f$. Then $\bcalD_p(f)$ (resp. $\bcalA_p(f)$) is a compactification of the descending manifold $\calD_p(f)$ (resp. ascending manifold $\calA_p(f)$) to a smooth manifold with corners whose codimension $r$ stratum of $\bcalD_p(f)-\bar{b}^{-1}(p\times \infty_M)$ (resp. $\bcalA_p(f)-\bar{b}^{-1}(\infty_M\times p)$) for $r\geq 1$ is canonically diffeomorphic to
\[ \begin{split}
	&\coprod_{{{q_1,\ldots,q_r\in\Sigma(f)}\atop{p,q_1,\ldots,q_r\,\mathrm{distinct}}}}\bbcalM{f}{p}{q_1}\times\bbcalM{f}{q_1}{q_2}\times\cdots\times\bbcalM{f}{q_{r-1}}{q_r}\times\calD_{q_r}(f)\\
	&(\mbox{resp. }\coprod_{{{q_1,\ldots,q_r\in\Sigma(f)}\atop{p,q_1,\ldots,q_r\,\mathrm{distinct}}}}\calA_{q_r}(f)\times \bbcalM{f}{q_r}{q_{r-1}}\times\cdots\times \bbcalM{f}{q_2}{q_1}\times\bbcalM{f}{q_1}{p}).
\end{split}\]
\end{Prop}
\begin{proof}
Suppose that the singular set $\Sigma(f)=\{p_1,\ldots,p_N\}$ is numbered as in \S\ref{sss:covering_M2} and suppose that $p=p_k\in W_k\cap\Sigma(f)$ for some $k$. It follows from the definition of $\bcalM_2(f)$ that $\bcalD_p(f)\cap \bar{b}^{-1}(W_k\times W_k)=\bcalM_2(f;W_k,W_k)\cap (\{p\}\times W_k)$. By Lemma~\ref{lem:dM(W,W)}, the right hand side is equal to $\{p\}\times(\calD_p(f)\cap W_k)$ since $\calM_2(f;W_k,W_k)\cap(\{p\}\times W_k)=\emptyset$. Similarly, we have
\[ \begin{split}
\bcalD_p(f)\cap \bar{b}^{-1}(W_k\times W_j)&=\bcalM_2(f;W_k,W_j)\cap (\{p\}\times L_k\times \cdots \times L_{j+1}\times W_j)\\
&=\mathrm{pr}\Bigl[(\{p\}\times(\calD_p(f)\cap L_k))\times_{L_k}\bcalM_2(f;L_k,W_j)\Bigr].
\end{split} \]
The descriptions of the strata in the statement follow from these identities and from Lemma~\ref{lem:dM(W,W)}, \ref{lem:dM(infty,W)} and \ref{lem:CM_2(W,W)}. The result for $\bcalA_p(f)$ is analogous.
\end{proof}

%\clearpage
%%%%%%%%%%%%%%%%%%%%%%%%%%%%%%
%%%%%%%%%%%%%%%%%%%%%%%%%%%%%%
\mysection{Moduli space of flow graphs}{}

%%%%%%%%%%%%%%%%%%%%%%%%%%%%%%
\subsection{Transversality for $\calM_\Gamma$}\label{ss:transversality}

Let $\Gamma\in\calG_{n,m,\vec{\eta}}^0(\vec{C})$ be a $\vec{C}$-colored graph with $a$ inputs (and $a$ outputs) and without bivalent vertices. For simplicity, we assume that the noncompact edges in $\Se(\Gamma)$ are labeled (via $\beta$) by $\{1,2,\ldots,a\}$. Let $p_i$ (resp. $q_i$) be a basis element attached on the input (resp. output) of the edge labeled $i$. We define a $C^{r-1}$-smooth map $\Phi_{\vec{f}}:\prod_{j=1}^a(\calA_{q_j}(f_j)\times\calD_{p_j}(f_j))
	\times M_0^n\times \R_{>0}^{m-a}\to M_0^{n+m+a}$ by
\[\begin{split}
 &\Phi_{\vec{f}}(u_1,v_1,\ldots,u_a,v_a;x_1,\ldots,x_n;t_{a+1},\ldots,t_m)\\
	&=\prod_{i=1}^n(x_i,y_{k_1^i},\ldots,y_{k_{r_i}^i},
		v_{\bar{k}_1^i},\ldots,v_{\bar{k}_{\bar{r}_i}^i},
		u_{\bar{\ell}_1^i},\ldots,u_{\bar{\ell}_{\bar{s}_i}^i}),
\end{split} \]
where $y_k=\Phi_{f_k}^{t_k}(x_{\source(k)})$ and $\R_{>0}=(0,\infty)$ (see \S\ref{ss:M_G} for the symbols). Let $\Delta\subset M_0^{n+m+a}$ be the subset consisting of all the points of the form
$\prod_{i=1}^n(x_i,\overbrace{x_i,\ldots,x_i}^{r_i+\bar{r}_i+\bar{s}_i})$
for $(x_1,\ldots,x_n)\in M_0^n$. Then $\calM_\Gamma(\vec{f})$ is the image of $\Phi_{\vec{f}}^{-1}(\Delta)$ under the projection onto $M_0^n$ and the projection induces an embedding.

\begin{Exa}\label{ex:Phi}
We consider the graph in (\ref{eq:ex_graph}), for example. The map $\Phi_{\vec{f}}:(\calA_q(f_1)\times\calD_p(f_1))\times M_0^4\times \R_{>0}^5\to M_0^{11}$ is defined by
\[ \begin{split}
	\Phi_{\vec{f}}&(u,v;x_1,x_2,x_3,x_4;t_2,t_3,t_4,t_5,t_6)\\
	&=(x_1,u,x_2,\Phi_{f_2}^{t_2}(x_1),x_3,\Phi_{f_4}^{t_4}(x_1),\Phi_{f_5}^{t_5}(x_2),
		x_4,\Phi_{f_3}^{t_3}(x_2),\Phi_{f_6}^{t_6}(x_3),v).
\end{split}\]
Then $\calM_\Gamma(\vec{f})$ is the image of $\Phi_{\vec{f}}^{-1}(\Delta)$ under the projection onto $M_0^4$, where
\[ \Delta=\{(x_1,x_1,x_2,x_2,x_3,x_3,x_3,x_4,x_4,x_4,x_4)\,;\,(x_1,x_2,x_3,x_4)\in M_0^4\}. \]
\qed
\end{Exa}

Let $\calU_j$ be a $C^r$-small neighborhood of a Morse function in the Banach manifold $C^r_{\varphi_\infty}(M_0)$ such that the cardinality of the set of critical points is constant on $\calU_j$. By considering $\Phi_{\vec{f}}$ for all $\vec{f}\in \prod_{j=1}^m\calU_j$ for a fixed Riemannian metric $\mu$ on $M_0$, we get a smooth map
\[ \Phi:\prod_{j=1}^a\bigcup_{f_j\in\calU_j}(\calA_{q_j}(f_j)\times\calD_{p_j}(f_j))\times M_0^n\times\R_{>0}^{m-a}\times \prod_{j=a+1}^m\calU_j\to M_0^{n+m+a}, \]
where we consider $\bigcup_{f_j\in\calU_j}(\calA_{q_j}(f_j)\times\calD_{p_j}(f_j))$ as a subspace of $\calU_j\times M_0^2$. 

The proof of the following lemma is almost the same as \cite[Proposition~12.5]{FO}.
\begin{Lem}\label{lem:phi_transversal}
The smooth map $\Phi$ is transversal to $\Delta$.
\end{Lem}
\begin{proof}
For points $\bvec{x}$ of $\Delta\cap\mathrm{Im}\,\Phi$, we show that every normal vector $v=(v_1,\ldots,v_{m+n+a})\in T_{\bvec{\scriptstyle x}} M_0^{m+n+a}=\bigoplus_{i=1}^{m+n+a}T_{x_i}M_0$ to $\Delta$ lies in the image of the differential $d\Phi:T\Dom(\Phi)\to TM_0^{m+n+a}$. Now suppose that a point $\bvec{x}=\prod_{i=1}^n(x_i,\ldots,x_i)\in\Delta$ lies in the image of $\Phi$. We first prove the claim in the case where each $x_i$ does not belong to the singular set $\Sigma(f_1)\cup\cdots\cup\Sigma(f_m)$. Suppose that there exist $x\in M_0$ and $t_\ell\in (0,\infty)$ such that $\Phi_{f_\ell}^{t_\ell}(x)=x_i$ for some $\ell$. Then there exists a local coordinate $\psi_U:V\stackrel{\approx}{\to} U$ from an open neighborhood $V$ of $0\in\R^d$ to a neighborhood $U$ of $x_i$ in $M_0$, such that $-\mathrm{grad}_{\psi_U(q)}\,f_\ell=\lambda(q)d\psi_{U}(v_0)$, $q\in V$, where $\lambda:V\to \R$ is a positive smooth function and $v_0\in\R^d$ is a constant vector. Here, we may assume without loss of generality that $(d\psi_{U})^{-1}(v_i)$ agrees with $(0,\ldots,0,b)$ for some $b<0$. 

Now for small constants $\varepsilon>0$, $\kappa>0$ and $\beta>0$, we define a $C^\infty$-function $g_{\varepsilon,\kappa}:\R\to \R$ by
\[ g_{\varepsilon,\kappa}(t)=\left\{\begin{array}{ll}
	\kappa e^{-\frac{\beta t^2}{\varepsilon^2-t^2}} & -\varepsilon<t<\varepsilon\\
	0 & \mbox{otherwise}
	\end{array}\right. \]
This is a cloche function. We assume without loss of generality that on $U$, the tangent vector $-\mathrm{grad}_{\psi_U(q)}f_\ell$ is transversal to the image of $\psi_U$ of the plane $\{(t_1,\ldots,t_{d-1},0)\}$. We define a $C^\infty$-function $h_{\ell,\ve,\kappa}:\R^d\to \R$ by
\[ h_{\ell,\ve,\kappa}(t_1,\ldots,t_d)=
	f_\ell\circ \psi_U(t_1,\ldots,t_d)- s g_{\varepsilon,\kappa}(t_d). 
\]
for a constant $s>0$. By this perturbation, the negative gradient after passing through the part $-\varepsilon<t_d<\varepsilon$ shifts by a multiple of $(0,\ldots,0,1)$. Therefore, for some $\varepsilon, \kappa$ and $\lambda_0\in\R$, we have $\left.\frac{d\Phi_{h_{\ell,\varepsilon,\kappa}}^{t_\ell+s\lambda_0}}{ds}\right|_{s=0}=v_i$. Of course, one must adjust $h$ in $V$ slightly so that the perturbed function extends to a smooth function on $M_0$. This shows that the differential of $\Phi$ is surjective at any point of $\Phi^{-1}(\bvec{x})$, $\bvec{x}\in\Delta\cap\mathrm{Im}\,\Phi$, provided that $x_i$ avoids singular set.

If $x_i$ agrees with $p_\ell\in \Sigma(f_\ell)$ for some $\ell$ and if the $\ell$-th edge of $\Gamma$ has $p_\ell$ as the input or the output, then by a perturbation of $f_\ell$ on a small neighborhood of $p_\ell$ the position of $p_\ell$ shifts in an arbitrary direction. This shows the transversality for the case $x_i=p_\ell$. 
\end{proof}

\begin{proof}[Proof of Proposition~\ref{prop:M_mfd}]
It follows from Lemma~\ref{lem:phi_transversal} that $\Phi^{-1}(\Delta)$ is a (infinite dimensional) submanifold of codimension $(n+m+a-n)d=(m+a)d$. Let $\pi:\Phi^{-1}(\Delta)\to \prod_{j=1}^m\calU_j$ be the restriction of the projection. Since the projection 
\[ \prod_{j=1}^a\bigcup_{f_j\in\calU_j}(\calA_{q_j}(f_j)\times \calD_{p_j}(f_j))\times M_0^n\times \R_{>0}^{m-a}\times \prod_{j=a+1}^m\calU_j\to \prod_{j=1}^m\calU_j \]
is a Fredholm map of index $\sum_{j=1}^a(i(p_j)+(d-i(q_j)))+nd+(m-a)=nd+ad+m-a+\sum_{j=1}^a\eta_j$, the index of the projection $\pi$ is
\[ nd+ad+m-a+\sum_{j=1}^a\eta_j-(m+a)d=(n-m)d+\sum_{j=1}^m\eta_j. \]
Hence for a regular value $\vec{f}\in\prod_{j=1}^m\calU_j$ of $\pi$, the fiber of $\pi$ is a smooth manifold of dimension $(n-m)d+\sum_{j=1}^m\eta_j$. By the Sard--Smale theorem (\cite{Sm}), the set of regular values of $\pi$ is residual. The second statement follows from the fact that there are only finitely many graphs in $\calG_{m,n,\vec{\eta}}(\vec{C})$ for a fixed triple $m,n,\sum_{j=1}^m \eta_j$ and that a finite intersection of residual subsets is residual too. 
\end{proof}

%%%%%%%%%%%%%%%%%%%%%%%%%%%%%%
\subsection{Compactification of $\Conf_n(M)$ of Fulton--MacPherson}\label{ss:FM}

For a closed $d$-manifold $M$, the configuration space $\Conf_n(M)$ is a submanifold of $M^n$, that is the complement of the closed subset
\[ \Sigma=\{(x_1,\ldots,x_n)\in M^n\,;\,\mbox{$x_i=x_j$ for some $i\neq j$ or $x_i=\infty_M$ for some $i$}\}\subset M^n. \]
There is a natural filtration $\Sigma=\Sigma_n\supset\cdots\supset\Sigma_2\supset \Sigma_1$ with
\[ \Sigma_j=\{(x_1,\ldots,x_n)\in M^n\,;\,\#\{x_1,\ldots,x_n,\infty_M\}\leq j\}. \]
The difference $\Sigma_{i+1}-\Sigma_i$ is a disjoint union of submanifolds of $M^n-\Sigma_i$. This property allows one to iterate (real) blow-ups along the filtration from the deepest one: First, one can consider the blow-up $B\ell(M^n,\Sigma_1)$ along the 0-submanifold $\Sigma_1=\{(\infty_M,\ldots,\infty_M)\}$ of $M^n$. Recall that a blow-up replaces a submanifold with its normal sphere bundle. Since the closure of $\Sigma_2-\Sigma_1$ in $B\ell(M^n,\Sigma_1)$ is also a disjoint union of smooth submanifolds (with boundaries), one can apply another blow-up along it, and so on. After the blow-ups along all the strata of $\Sigma$ of codimension $\geq 1$, one obtains a smooth compact manifold with corners $\bConf_n(M)$. 

We will need a precise description of the boundary of $\bConf_n(M)$ in the proof of invariance of $\widehat{Z}_{2k,3k}$, so we shall briefly recall it here. The space $\bConf_n(M)$ has a natural stratification corresponding to bracketings of the $n+1$ letters $1,2,\ldots,n,\infty$, e.g., $((137)(25))46\infty$ (see \cite{FM, Ko1, BT}). Roughly speaking, a pair of brackets corresponds to a face created by one blow-up. For example, the face corresponding to $((137)(25))46\infty$ is obtained by a sequence of blow-ups corresponding to a sequence $1234567\infty\to (12357)46\infty\to ((137)(25))46\infty$. 

The codimension one (boundary) strata of $\bConf_n(M)$ correspond to bracketings of the form $(\cdots)\cdots$, with only one pair of brackets. For example, the stratum $\partial_{\{1,\ldots,j\}}\bConf_n(M)$ of $\partial\bConf_n(M)$ corresponding to the bracketings $(12\cdots j)j+1\cdots n$ is the face created by the blow-up along the closure of the submanifold
\[ \Delta_j=\{(x_1,\ldots,x_n)\in M^n\,;\,x_1=\cdots=x_j,\mbox{otherwise distinct}\}\subset M^n \]
in the result of the previous blow-ups. More precisely, $\partial_{\{1,\ldots,j\}}\bConf_n(M)$ can be naturally identified with the blow-ups of the total space of the normal $S^{(j-1)d-1}$-bundle of $\Delta_j\subset M^n$ along the intersection with the closures of deeper diagonals that correspond to deeper bracketings. The fiber of the normal $S^{(j-1)d-1}$-bundle over a point $(x_j,\ldots,x_n)\in\Delta_j$ is $(\{(0,y_2,\ldots,y_j)\in(\R^d)^j\}-\{0\})/\mbox{(dilation)}\cong S^{(j-1)d-1}$, where the coordinate $y_i$ corresponds to $x_i-x_1$ (where it makes sense) under the geodesic coordinate from a framing of $T_{x_j}M$. The stratum $\partial_{\{1,\ldots,j\}}\bConf_n(M)$ is a fiber bundle over $\Delta_j$. We denote the fiber of $\partial_{\{1,\ldots,j\}}\bConf_n(M)$ over a point of $\Delta_j$ by $\Conf_j^\mathrm{local}(\R^d)$. As done in \S\ref{ss:infinitesimal}, we identify $\Conf_j^\mathrm{local}(\R^d)$ with the subset of $\Conf_j(\R^d)$, as
\[ \Conf_j^\mathrm{local}(\R^d)=\Bigl\{(y_1,\ldots,y_j)\in \Conf_j(\R^d)\,;\, y_1=0, \sum_{\ell=2}^j\|y_\ell\|^2=1\Bigr\}. \]
 We denote by $\bConf_j^\mathrm{local}(\R^d)$ the closure of the image of the inclusion $\Conf_j^\mathrm{local}(\R^d)\hookrightarrow \bConf_j(\R^d)$, which is compact. The base space $\Delta_j$ is naturally diffeomorphic to $\bConf_{n-j+1}(M)$ and we denote by $\mathrm{pr}_j:\Delta_j\to M$ the projection $(x_j,\ldots,x_n)\mapsto x_j$. So $\partial_{\{1,\ldots,j\}}\bConf_n(M)$ has the structure of the pullback of the associated $\bConf_j^\mathrm{local}(\R^d)$-bundle of $TM$ ($\bConf_j^\mathrm{local}(\R^d)$ is an $SO_d$-space) pulled back by $\mathrm{pr}_j$. The definition of $\partial_A\bConf_n(M)$ for general subset $A\subset \{1,\ldots,n,\infty\}$ corresponding to the bracketing $(A)A^c$ is similar. 

It will turn out that the faces of $\partial\bConf_n(M)$ corresponding to coincidence of two points are special among the codimension one strata of $\bConf_n(M)$. We denote by $\partial^\mathrm{pri}\bConf_n(M)$ (`pri' for principal) the union of the faces corresponding to coincidence of two points and we denote by $\partial^\mathrm{hi}\bConf_n(M)$ (`hi' for hidden) the union of all the faces corresponding to coincidence of at least three points. Among the hidden faces, we call the face $\partial_{\{1,2,\ldots,n\}}\bConf_n(M)$ the anomalous face and denote it by $\partial^a\bConf_n(M)$.

%%%%%%%%%%%%%%%%%%%%%%%%%%%%%%
\subsection{Compactification of the moduli space $\calM_\Gamma$}\label{ss:comp_fuk}

\begin{proof}[Proof of Proposition~\ref{prop:M_G-1-mfd}] Let $\vec{f}=(f_1,f_2,\ldots,f_{3k})$ and $\vec{\mu}=(\mu_1,\mu_2,\ldots,\mu_{3k})$ be sequences of Morse functions and metrics on $M_0$ respectively such that for each $i$ the pair $(f_i,\mu_i)$ is Morse--Smale and that $\mu_i$ is Euclidean near $\Sigma(f_i)$ with respect to the coordinates of the Morse lemma. We assume that the gradients of $f_i$ are taken with respect to $\mu_i$. We shall construct a compactification $\bcalM_{\Gamma}(\vec{f})$ of $\calM_\Gamma(\vec{f})$.

For $p,q\in\Sigma(f)$, let 
\[ \calN_{pq}(f)=\calA_q(f)\times \calD_p(f),\quad \bcalN_{pq}(f)=\bcalA_q(f)\times \bcalD_p(f). \]
For $j\in \{1,2,\ldots,3k\}$, let
\[ Q_j=\left\{\begin{array}{ll}
\bcalM_2(f_j) & \mbox{if $\beta(j)\in\Comp(\Gamma)$}\\
\bcalN_{pq}(f_j) & \mbox{if $\beta(j)\in\Se(\Gamma)$ with input $p$, output $q$}
\end{array}\right. \]
For $i\in \{1,2,\ldots,2k\}$, let $j_1,j_2,j_3\in\{1,2,\ldots,3k\}$ be the labels of the edges which are incident to the $i$-th vertex of $\Gamma$. We define a smooth map $G_i:Q_{j_1}\times Q_{j_2}\times Q_{j_3}\to M\times M\times M$ as follows. Let $\bar{b}_{ji}:Q_j\to M$ be defined by 
\[ \bar{b}_{ji}=\left\{\begin{array}{ll}
\mathrm{pr}_1\circ \bar{b} & \mbox{if $i=\sigma(j)$}\\
\mathrm{pr}_2\circ \bar{b} & \mbox{if $i=\tau(j)$}
\end{array}\right. \]
Then we define $G_i=\bar{b}_{j_1i}\times \bar{b}_{j_2i}\times \bar{b}_{j_3i}$. Let $\widehat{G}_i:\prod_{j=1}^{3k}Q_j\to M^3$ be the composition $G_i\circ\mathrm{proj}:\prod_{j=1}^{3k}Q_j\to Q_{j_1}\times Q_{j_2}\times Q_{j_3}\to M\times M\times M$. Let $\Delta_3=\{(x,x,x);x\in M\}$. We define
\[ \bcalM_\Gamma^\times(\vec{f})=\bigcap_{i=1}^{2k}\widehat{G}_i^{-1}(\Delta_3). \]
The restriction of $\bcalM_\Gamma^\times(\vec{f})$ to $\prod_{j=1}^{3k}\mathrm{Int}\,Q_j$ is canonically identified with $\calM_\Gamma(\vec{f})$, which is a smooth manifold. Let $\bar{b}_\Gamma:\bcalM_\Gamma^\times(\vec{f})\to M^{2k}$ be the map that assigns the positions of $2k$ trivalent vertices of a flow graph in $M$. 

Now we consider the case where $\dim\calM_\Gamma(\vec{f})\leq 1$. 

If $\dim\calM_\Gamma(\vec{f})=0$, then $\bcalM_\Gamma^\times(\vec{f})=\calM_\Gamma(\vec{f})\subset\prod_{j=1}^{3k}\mathrm{Int}\,Q_j$ is a finite set. In this case $\bcalM_\Gamma(\vec{f})=\bcalM_\Gamma^\times(\vec{f})$ is as desired. 

If $\dim\calM_\Gamma(\vec{f})=1$, then $\bcalM_\Gamma^\times(\vec{f})$ may have nonempty intersection with $\partial(\prod_{j=1}^{3k}Q_j)$. By Proposition~\ref{prop:bM2_mfd}, the intersection of $\bcalM_\Gamma^\times(\vec{f})$ with $\partial(\prod_{j=1}^{3k}Q_j)$ consists of flow graphs of the following forms.
\begin{enumerate}
\item There is one edge that is a once broken flow line.
\item A set of edges is collapsed to a finite subset of $M$. 
\end{enumerate}
Here, we may assume that the intersection has no flow graphs broken more than once by perturbing the function $f_j$ for the broken edge slightly. On a neighborhood of a point of $\bcalM_\Gamma^\times(\vec{f})$ of type (1), $\bcalM_\Gamma^\times(\vec{f})$ restricts to a smooth 1-manifold with boundary. On a neighborhood of a point of $\bcalM_\Gamma^\times(\vec{f})$ of type (2), $\bcalM_\Gamma^\times(\vec{f})$ may have singularities on the boundary and may not be a smooth manifold. In fact, it is the cone over finitely many points whose cone point lies on $\partial(\prod_{j=1}^{3k}Q_j)$ by the strata transversality near the boundary. 

The conic singularity can be resolved by a sequence of blow-ups of $\bcalM_\Gamma^\times(\vec{f})$ analogous to the compactification of $\Conf_{2k}(M)$ in \S\ref{ss:FM} as follows. Let $N_{\Sigma_1}\subset M^{2k}$ be a small tubular neighborhood of the highest codimension stratum $\Sigma_1$ of $\Sigma$. Its preimage $\widehat{N}_{\Sigma_1}=\bar{b}^{-1}_\Gamma(N_{\Sigma_1})\subset \bcalM_\Gamma^\times(\vec{f})$ is a subspace of small graphs concentrated near $\infty_M$. The restriction of $\bar{b}_\Gamma$ to $\widehat{N}_{\Sigma_1}$ is a topological embedding of a cone. Hence the blow-up of $M^{2k}$ along $\Sigma_1$ replaces $\bar{b}_\Gamma(\widehat{N}_{\Sigma_1})$ with a smooth 1-manifold $B\ell(\bar{b}_\Gamma(\widehat{N}_{\Sigma_1}),\Sigma_1)$ with boundary. By identifying $\widehat{N}_{\Sigma_1}-\bar{b}^{-1}_\Gamma(\Sigma_1)$ with $\mathrm{Int}\,B\ell(\bar{b}_\Gamma(\widehat{N}_{\Sigma_1}),\Sigma_1)$ through $\bar{b}_\Gamma$, we obtain a space
\[ \bcalM_\Gamma^\times(\vec{f})[1]=(\bcalM_\Gamma^\times(\vec{f})-\bar{b}^{-1}_\Gamma(\Sigma_1))\cup_{\bar{b}_\Gamma} B\ell(\bar{b}_\Gamma(\widehat{N}_{\Sigma_1}),\Sigma_1). \]
The singularities on $\bar{b}_\Gamma^{-1}(\Sigma_1)$ have been resolved. Next, we resolve the singularities on $\bar{b}_\Gamma^{-1}(\Sigma_2)$. Let $N_{\Sigma_2}\subset M^{2k}$ be a small tubular neighborhood of $\Sigma_2-\Sigma_1$. We may assume that there is no edge of broken flow line in the flow graphs of $\bcalM_\Gamma^\times(\vec{f})[1]\cap \bar{b}^{-1}(N_{\Sigma_2})$. Its preimage $\widehat{N}_{\Sigma_2}=\bar{b}^{-1}_\Gamma(N_{\Sigma_2})\subset \bcalM_\Gamma^\times(\vec{f})[1]$ is a subspace with a small subgraph with $2k-1$ vertices. The restriction of $\bar{b}_\Gamma$ to $\widehat{N}_{\Sigma_2}$ is a topological embedding. Hence the blow-up of $B\ell(M^{2k},\Sigma_1)$ along the closure of $\Sigma_2-\Sigma_1$ replaces $\bar{b}_\Gamma(\widehat{N}_{\Sigma_2})$ with a smooth manifold $B\ell(\bar{b}_\Gamma(\widehat{N}_{\Sigma_2}),\Sigma_2)$ with corners. By identifying $\widehat{N}_{\Sigma_2}-\bar{b}^{-1}_\Gamma(\Sigma_2)$ with $\mathrm{Int}\,B\ell(\bar{b}_\Gamma(\widehat{N}_{\Sigma_2}),\Sigma_2)$ through $\bar{b}_\Gamma$, we obtain
\[ \bcalM_\Gamma^\times(\vec{f})[2]=(\bcalM_\Gamma^\times(\vec{f})[1]-\bar{b}^{-1}_\Gamma(\Sigma_2))\cup_{\bar{b}_\Gamma} B\ell(\bar{b}_\Gamma(\widehat{N}_{\Sigma_2}),\Sigma_2). \]
Repeating in this way for $\Sigma_3,\ldots,\Sigma_{2k-1}$, we obtain spaces $\bcalM_\Gamma^\times(\vec{f})[3]$, $\ldots$, $\bcalM_\Gamma^\times(\vec{f})[2k-1]$. We set $\bcalM_\Gamma(\vec{f})=\bcalM_\Gamma^\times(\vec{f})[2k-1]$. By definition this is a compactification of $\calM_\Gamma(\vec{f})$ as desired in Proposition~\ref{prop:M_G-1-mfd}.
\end{proof}
\begin{Rem}\label{rem:b_emb}
By abuse of notation, we denote by $\bar{b}_\Gamma:\bcalM_\Gamma(\vec{f})\to \bConf_{2k}(M)$ the natural map that assigns the positions of $2k$ trivalent vertices of a flow graph in $M$. In general, $\bar{b}_\Gamma$ may not be an embedding but only an immersion if $\dim\bcalM_\Gamma(\vec{f})=1$. 
\end{Rem}
%\clearpage

%%%%%%%%%%%%%%%%%%%%%%%%%%%%%%
%%%%%%%%%%%%%%%%%%%%%%%%%%%%%%
\mysection{(Co)orientations of the moduli spaces}{s:coori}

Let $f:M\to \R$ be a Morse function and $\mu$ be a metric on $M$ that is Morse--Smale and that is Euclidean near critical points with respect to the local coordinate of the Morse lemma. We shall fix (co)orientations of the trajectory spaces and describe the induced coorientations at the boundaries. The results in this section will be used in \S\ref{s:indep_g}, \S\ref{s:indep_4-cob} and \S\ref{ss:ddd}.

%%%%%%%%%%%%%%%%%%%%%%%%%%%%%%
\subsection{Convention for (co)orientations of trajectory spaces}\label{ss:convention_coori}

In the following we follow the orientation convention of Appendix~\ref{s:ori}. Let $o(M)\in\Gamma(\bigwedge^d T^*M)$ denote a $d$-form representing the orientation of $M$. The trajectory space $\calM_2(f)$ is the image of the embedding $\varphi:(M-\Sigma(f))\times (0,\infty)\to (M-\Sigma(f))^2$ given by $\varphi(x,t)=(x,\Phi_f^t(x))$. The Jacobian matrix of $\varphi$ at $(x,t)\in (M-\Sigma(f))\times (0,\infty)$ is 
\begin{equation}\label{eq:Jphi}
 J\varphi_{(x,t)}=\left(\begin{array}{cc}
  I & \mathbf{0}\\
  (d\Phi_f^t)_x & -(\mathrm{grad}\,f)_y
  \end{array}\right). 
\end{equation}
If $e_1,e_2,\ldots,e_d$ is an orthonormal basis of $T_xM$, then $T_{(x,y)}\calM_2(f)$, $y=\Phi_f^t(x)$, is spanned by $e_1+Ae_1,e_2+Ae_2,\ldots,e_d+Ae_d,-(\mathrm{grad}\,f)_y$, where $A=(d\Phi_f^t)_x$. With this in mind, we orient $\calM_2(f)$ as
\[ o(\calM_2(f))_{(x,y)}=(-df)_y\wedge (dx_1+A_*dx_1)\wedge (dx_2+A_*dx_2)\wedge \cdots \wedge (dx_d+A_*dx_d), \]
where we assume that $o(M)_x=dx_1\wedge \cdots \wedge dx_d$ for an orthonormal basis $\{dx_1,\ldots,dx_d\}$ of $T_x^*M$ and $A_*=(d\Phi_f^t)_{x*}:T_x^*M\to T_y^*M$. 

We orient $M\times M$ by $o(M\times M)_{(x,y)}=o(M)_y\wedge o(M)_x$. Then the coorientation $o^*_{M\times M}(\calM_2(f))$ of $\calM_2(f)$ in $M\times M$ is determined by 
\[ o^*_{M\times M}(\calM_2(f))_{(x,y)}=* o(\calM_2(f))_{(x,y)}. \]
We orient $\calN_{pq}(f)=\calA_{q}(f)\times\calD_{p}(f) \subset M\times M$ by giving the coorientation
\[ o^*_{M\times M}(\calN_{pq}(f))_{(x,x')}=o^*_M(\calA_q(f))_{x}\wedge o^*_M(\calD_p(f))_{x'}, \]
where $o^*_M(\calA_q(f))_{x}$ and $o^*_M(\calD_p(f))_{x'}$ are the ones determined by $o(\calD_p(f))$ and $o(\calA_q(f))$ respectively fixed in \S\ref{ss:M_G}.

%%%%%%%%%%%%%%%%%%%%%%%%%%%%%%
\subsection{(Co)orientations induced on the boundaries of descending and ascending manifolds}\label{ss:oriDA}

For a Morse--Smale pair $(f,\mu)$ and its critical points $p,q$, we shall describe the induced (co)orientations of the faces $\calF_r\bcalD_p(f)$ (resp. $\calF_r\bcalA_q(f)$) of $\partial_1\bcalD_p(f)$ (resp. $\partial_1\bcalA_q(f)$) of flow lines broken at a critical point $r$, which are induced from the (co)orientation of $\bcalD_p(f)$ (resp. $\bcalA_p(f)$).

Let $\bar{b}:\bcalD_p(f)\to M$ be the map that assigns to each (possibly broken) flow sequence the terminal endpoint. If $i(p)-i(r)=1$ and if $a$ is a point of $M$ that is the image of $\bar{b}$ from a once broken flow sequence $\hat{a}$ in $\partial_1\bcalD_p(f)$ broken at a critical point $r\in\Sigma(f)$, then by Proposition~\ref{prop:compactification_D} there is an open neighborhood $N_a$ of $a$ in $M$ such that $\bar{b}^{-1}(N_a)$ is a disjoint union of finitely many half-disks whose set of components naturally corresponds to the finite set $\bbcalM{f}{p}{r}$. Let $\widehat{N}_{\hat{a}}$ be the component of $\bar{b}^{-1}(N_a)$ on which $\hat{a}$ lies. The restriction of $\bar{b}$ to $\widehat{N}_{\hat{a}}$ is an embedding and hence the coorientation $o^*_M(\partial_1\bcalD_p(f))_a$ makes sense by identifying $\widehat{N}_{\hat{a}}$ with $\bar{b}(\widehat{N}_{\hat{a}})$. The same is also true for $\partial_1\bcalA_q(f)$ at a once broken flow sequence broken at $r\in\Sigma(f)$ such that $i(r)-i(q)=1$. 

Note that $\mathrm{Int}\,\bar{b}(\widehat{N}_{\hat{a}})$ is an open subset of $\calD_p(f)$ and its closure in $N_a$ is $\bar{b}(\widehat{N}_{\hat{a}})$. Hence the (co)orientation of $\calD_p(f)$ induces a (co)orientation of the boundary $\partial \bar{b}(\widehat{N}_{\hat{a}})$ at $a$. We define $o^*_M(\partial_1\bcalD_p(f))_a$ to be the one induced in this way. We also define $o^*_M(\partial_1\bcalA_q(f))_a$ similarly. 

\begin{Lem}\label{lem:ind_ori_D}
Under the assumption above, let $p,r$ be critical points of $f$ such that $f(p)>f(r)$ and $i(p)-i(r)=1$. Let $N_a$ and $a\in \bar{b}(\widehat{N}_{\hat{a}})$ be as above. Let $b$ be a point of $\bbcalM{f}{p}{r}$ such that $\widehat{N}_{\hat{a}}$ corresponds to $b$. Then the following identity in $\bigwedge^\bullet T_a^*M$ holds.
\[ o^*_M(\partial_1\bcalD_p(f))_a=(-1)^{i(r)+1}\ve_f(p,r)_b\, o^*_M(\calD_r(f))_{a}.\]
\end{Lem}
\begin{proof}
Let $i=i(r)$. By assumptions $f(p)>f(r)$ and $i(p)-i(r)=1$, the index of $r$ is in $0\leq i(r)\leq d-1$. It suffices to check the assertion for one broken flow line. By Morse Lemma there is a local coordinate $(x_1,\ldots,x_d)$ around $r$ on which $f$ agrees with $\displaystyle f(r)-\frac{x_1^2}{2}-\cdots-\frac{x_i^2}{2}+\frac{x_{i+1}^2}{2}+\cdots+\frac{x_d^2}{2}$. In this coordinate, $\calD_r(f)$ agrees with $\{(x_1,\ldots,x_d)\in\R^d;x_{i+1}=\cdots=x_d=0\}$ and $\calA_r(f)$ agrees with $\{(x_1,\ldots,x_d)\in\R^d;x_1=\cdots=x_i=0\}$. We may put 
\[ o(\calD_r(f))=\beta\,dx_1\cdots dx_i\quad (\beta=\pm 1). \]
We may assume without loss of generality that $\calD_p(f)$ agrees with $\{(x_1,\ldots,x_d)\in\R^d; x_{i+1}=\cdots=x_{d-1}=0,\,x_d>0\}$ in a neighborhood of $r$ and we may put
\[ o(\calD_p(f))=\alpha\, dx_1\cdots dx_i dx_d \quad (\alpha=\pm 1). \]
Moreover we may assume that $a=(a_1,0,\ldots,0)$ for some $a_1\geq 0$. Then $\bar{b}(\widehat{N}_{\hat{a}})$ agrees with $\{(x_1,\ldots,x_d)\in\R^d;x_{i+1}=\cdots=x_{d-1}=0,\,x_d\geq 0\}\cap N_a$ on $N_a$ and
\[ o(\partial \bar{b}(\widehat{N}_{\hat{a}}))_a = \iota\left(\frac{\partial}{\partial x_d}\right)\alpha\,dx_1\cdots dx_idx_d=(-1)^i\alpha\,dx_1\cdots dx_i=(-1)^i\alpha\beta\,o(\calD_r(f))_a. \]

On the other hand, by assumption we have
\[ \begin{split}
	o^*_M(\calD_p(f))_b&=(-1)^{d-1-i}\alpha\,dx_{i+1}\cdots dx_{d-1},\quad\\
	o^*_M(\calA_r(f))_b&=*\beta\,dx_{i+1}\cdots dx_d=(-1)^{i(d-i)}\beta\,dx_1\cdots dx_i
\end{split} \]
for $b=(0,\ldots,0,b_d)$, $b_d>0$. Hence
\[ \begin{split}
	&o^*_M(\calD_p(f))_b\wedge o^*_M(\calA_r(f))_b=(-1)^{di+d-1}\alpha\beta\,dx_{i+1}\cdots dx_{d-1}dx_1\cdots dx_i\\
	&=(-1)^{d-1}\alpha\beta\,dx_1\cdots dx_{d-1}=-\alpha\beta\,\iota\Bigl(-\frac{\partial}{\partial x_d}\Bigr)dx_1\cdots dx_d
\end{split} \]
and we have $\ve_f(p,r)_b=-\alpha\beta$. This together with the equality above, we obtain the desired identity $o(\partial \bar{b}(\widehat{N}_{\hat{a}}))_a = (-1)^{i+1}\ve_f(p,r)_b\,o(\calD_r(f))_a$.
\end{proof}

\begin{Lem}\label{lem:ind_ori_A}
Under the assumption above, let $q,r$ be critical points of $f$ such that $f(q)<f(r)$ and $i(r)-i(q)=1$. Let $N_a$ and $a\in \bar{b}(\widehat{N}_{\hat{a}})$ be as above. Let $b$ be a point of $\bbcalM{f}{r}{q}$ such that $\widehat{N}_{\hat{a}}$ corresponds to $b$. Then the following identity in $\bigwedge^\bullet T_a^*M$ holds. 
\[ o^*_M(\partial_1\bcalA_q(f))_a=\ve_f(r,q)_b\,o^*_M(\calA_r(f))_a. \] 
\end{Lem}
\begin{proof}
Let $i=i(r)$. By assumptions $f(r)>f(q)$ and $i(r)-i(q)=1$, the index of $r$ is in $1\leq i(r)\leq d$. It suffices to check the assertion for one broken flow line. By Morse Lemma there is a local coordinate $(x_1,\ldots,x_d)$ around $r$ on which $f$ agrees with $\displaystyle f(r)-\frac{x_1^2}{2}-\cdots-\frac{x_i^2}{2}+\frac{x_{i+1}^2}{2}+\cdots+\frac{x_d^2}{2}$. In this coordinate, $\calD_r(f)$ agrees with $\{(x_1,\ldots,x_d)\in\R^d;x_{i+1}=\cdots=x_d=0\}$ and $\calA_r(f)$ agrees with $\{(x_1,\ldots,x_d)\in\R^d;x_1=\cdots=x_i=0\}$. We may put 
\[ o(\calA_r(f))=\beta\,dx_{i+1}\cdots dx_d\quad (\beta=\pm 1). \]
We may assume without loss of generality that $\calA_q(f)$ agrees with $\{(x_1,\ldots,x_d)\in\R^d; x_2=\cdots=x_i=0,\,x_1>0\}$ in a neighborhood of $r$ and we may put
\[ o(\calA_q(f))=\alpha\, dx_1dx_{i+1}\cdots dx_d \quad (\alpha=\pm 1). \]
Moreover we may assume that $a=(0,\ldots,0,a_d)$ for some $a_d\geq 0$. Then $\bar{b}(\widehat{N}_{\hat{a}})$ agrees with $\{(x_1,\ldots,x_d)\in\R^d;x_2=\cdots=x_i=0,\,x_1\geq 0\}\cap N_a$ on $N_a$ and 
\[ o(\partial \bar{b}(\widehat{N}_{\hat{a}}))_a = \iota\left(\frac{\partial}{\partial x_1}\right)\alpha\,dx_1dx_{i+1}\cdots dx_d=\alpha\,dx_{i+1}\cdots dx_d=\alpha\beta\,o(\calA_r(f))_a. \]

On the other hand, by assumption we have
\[ \begin{split}
	o^*_M(\calA_q(f))_b&=(-1)^{(i-1)(d-i)}\alpha\,dx_2\cdots dx_i,\quad \\
	o^*_M(\calD_r(f))_b&=*\beta\,dx_1\cdots dx_i=\beta\,dx_{i+1}\cdots dx_d
\end{split} \]
for $b=(b_1,0,\ldots,0)$, $b_1>0$. Hence
\[ \begin{split}
	o^*_M(\calD_r(f))_b\wedge o^*_M(\calA_q(f))_b&=(-1)^{(i-1)(d-i)}\alpha\beta\,dx_{i+1}\cdots dx_d dx_2\cdots dx_i\\
	&=\alpha\beta\,dx_2\cdots dx_d=\alpha\beta \,\iota\Bigl(\frac{\partial}{\partial x_1}\Bigr)dx_1\cdots dx_d
\end{split} \]
and we have $\ve_f(r,q)_b=\alpha\beta$. This together with the equality above, we obtain the desired identity $o(\partial \bar{b}(\widehat{N}_{\hat{a}}))_a = \ve_f(r,q)_b\,o(\calA_r(f))_a$.
\end{proof}

The following corollary shows that the boundary operator $\partial$ of the Morse complex satisfies $\partial\circ \partial(p)=\sum_q\sum_r\#\bbcalM{f}{p}{r}\cdot \#\bbcalM{f}{r}{q}\,q=0$. 
\begin{Cor}\label{cor:dd=0}
Let $p,r,q$ be critical points of $f$ such that $f(q)<f(r)<f(p)$ and $i(p)-i(r)=i(r)-i(q)=1$. Let $N_a$ and $a\in \bar{b}(\widehat{N}_{\hat{a}})$ be as above for $\hat{a}\in \partial_1\bcalD_p(f)$. Let $b$ be a point of $\bbcalM{f}{p}{r}$ such that $\widehat{N}_{\hat{a}}$ corresponds to $b$. Then the following identity in $\bigwedge^\bullet T_a^*M$ holds. 
\[ o^*_M(\partial \bcalD_p(f)\pitchfork \calA_q(f))_a=\ve_f(p,r)_b\, \ve_f(r,q)_a\,\iota(-\mathrm{grad}\,f)\,o(M)_a. \] 
\end{Cor}
\begin{proof}
By Lemma~\ref{lem:ind_ori_D} and (\ref{eq:ind_ori_bd_int}), $o^*_M(\partial \bcalD_p(f)\pitchfork \calA_q(f))_a$ is given as follows.
\[ \begin{split}
  &(-1)^{\mathrm{deg}\,o^*_M(\calA_q(f))_a}(-1)^{i(r)+1}\ve_f(p,r)_b
  \, o^*_M(\calD_r(f))_a\wedge o^*_M(\calA_q(f))_a\\
  &=(-1)^{i(q)+i(r)+1}\ve_f(p,r)_b\, \ve_f(r,q)_a\,\iota(-\mathrm{grad}\,f)\,o(M)_a.
\end{split}\]
\end{proof}

%%%%%%%%%%%%%%%%%%%%%%%%%%%%%%%%%%%%%%%%

%%%%%%%%%%%%%%%%%%%%%%%%%%%%%%
\subsection{(Co)orientation induced on the boundary of $\bcalM_2(f)$}

Let $f:M_0\to \R$ be a Morse function and $\mu$ be a metric on $M_0$ that is Morse--Smale and that is Euclidean near critical points. For a critical point $r$ of $f$, we shall describe the induced orientations of the face $\calF_r\bcalM_2(f)$ of $\partial_1\bcalM_2(f)$ of flow lines broken at a critical point $r$, that are induced from the orientation of $\bcalM_2(f)$. In the following we again follow the orientation convention of Appendix~\ref{s:ori}. 

Let $\bar{b}:\bcalM_2(f)\to M\times M$ be the map that assigns to each (possibly broken) flow sequence the pair of initial and terminal endpoints. If $a\in \calA_r(f)$ and $a'\in\calD_r(f)$, then there are open neighborhoods $N_a$ and $N_{a'}$ of $a$ and $a'$ in $M_0$ respectively such that $\bar{b}:\bar{b}^{-1}(N_a\times N_{a'})\to N_a\times N_{a'}$ is an embedding. Let $\widehat{N}_{(a,a')}=\bar{b}^{-1}(N_a\times N_{a'})$. Then the coorientation $o^*_{M\times M}(\partial_1\bcalM_2(f))_{(a,a')}$ makes sense by identifying $\widehat{N}_{(a,a')}$ with $\bar{b}(\widehat{N}_{(a,a')})$. 

Note that $\mathrm{Int}\,\bar{b}(\widehat{N}_{(a,a')})$ is an open subset of $\calM_2(f)$ and its closure in $N_a\times N_{a'}$ is $\bar{b}(\widehat{N}_{(a,a')})$. Hence the coorientation of $\calM_2(f)$ induces a coorientation of the boundary $\partial \bar{b}(\widehat{N}_{(a,a')})$ at $(a,a')$. We define $o^*_{M\times M}(\partial_1\bcalM_2(f))_{(a,a')}$ to be the one induced in this way.

\begin{Lem}\label{lem:ind_ori_M2}
Under the assumption above, let $a\in \calA_r(f)$ and $a'\in \calD_r(f)$. Then
\[ o^*_{M\times M}(\partial_1 \bcalM_2(f))_{(a,a')}=(-1)^{(i(r)+1)d+1}o^*_M(\calA_r(f))_a\wedge o^*_M(\calD_r(f))_{a'}. \]
\end{Lem}
\begin{proof}
Let $i=i(r)$. By Morse lemma, it suffices to check the assertion for the standard form $h(x_1,\ldots,x_d)=\displaystyle-\frac{x_1^2}{2}-\cdots-\frac{x_i^2}{2}+\frac{x_{i+1}^2}{2}+\cdots+\frac{x_d^2}{2}$ in place of $f$ and for $r=(0,\cdots,0)$. By convention, 
\[ o^*_{\R^d}(\calD_r(h))_x=\beta dx_{i+1}\cdots dx_d,\quad o^*_{\R^d}(\calA_r(h))_x=(-1)^{i(d-i)}\beta dx_1\cdots dx_i \]
for some $\beta=\pm 1$. 

First, assume $i\geq 1$. We assume without loss of generality that $a=(0,\ldots,0,a_d)$, $a'=(a_1',0,\ldots,0)$ for some $a_d\geq 0$ and $a_1'>0$. Recall that $\bcalM_2(h)$ is the set of points $(\rho u, v)\times (u,\rho v)$ for $u\in\R^i$, $v\in \R^{d-i}$, $\rho\in [0,1]$ (Lemma~\ref{lem:M2(h)}). Since the Jacobian matrix of $\varphi(\rho,u,v)=(\rho u, v)\times (u,\rho v)$ at $u=a'\in \R^i$, $v=a\in \R^{d-i}$ is (\ref{eq:jacobian}), $T_{\varphi(\rho,u,v)}^*\bcalM_2(h)$ is spanned by $a_1'dx_1+a_ddy_d$, $\rho dx_1+dy_1$, $\ldots$ , $\rho dx_i+dy_i$, $dx_{i+1}+\rho dy_{i+1}$, $\ldots$ , $dx_d+\rho dy_d$ if $dx_1,\ldots,dx_d$ (resp. $dy_1,\ldots,dy_d$) is the standard basis of $T_a^*\R^d$ (resp. $T_{a'}^*\R^d$). In fact, if $\rho>0$ and small, 
\begin{equation}\label{eq:goodori}
 o(\calM_2(h))_{\varphi(\rho,u,v)}\sim -(a_1'dx_1+a_ddy_d)\wedge \bigwedge_{k=1}^i(\rho dx_k+dy_k)\wedge \bigwedge_{k=i+1}^d(dx_k+\rho dy_k). 
\end{equation}
Indeed, by convention 
\[ o(\calM_2(h))_{(u,v)\times (e^t u,e^{-t}v)}=dy_1\wedge \bigwedge_{k=1}^i(dx_k+e^t dy_k)\wedge \bigwedge_{k=i+1}^d(dx_k+e^{-t} dy_k) \]
for $u\neq 0$, $v\neq 0$, $t> 0$. See (\ref{eq:Jphi}) in \S\ref{ss:convention_coori}. Then
\[ o(\calM_2(h))_{(u,v)\times (e^t u,e^{-t}v)}\wedge dy_2\cdots dy_d=dy_1\cdots dy_d\wedge dx_1\cdots dx_d. \]
On the other hand, 
\[ \begin{split}
	&-(a_1'dx_1+a_ddy_d)\wedge \bigwedge_{k=1}^i(\rho dx_k+dy_k)\wedge \bigwedge_{k=i+1}^d(dx_k+\rho dy_k)\wedge dy_2\cdots dy_d\\
	&=(a_1'-\rho a_d)\rho^{i-1}\,dy_1dx_1dx_2\cdots dx_d dy_2\cdots dy_d\\
	&=(a_1'-\rho a_d)\rho^{i-1}\,dy_1\cdots dy_d\wedge dx_1\cdots dx_d.
\end{split} \]
The coefficient $(a_1'-\rho a_d)\rho^{i-1}$ is positive if $\rho$ is small. 

The expression (\ref{eq:goodori}) is convenient because it extends smoothly to an orientation of $\bcalM_2(h)$ except the point from $(u,v)=(0,0)$. At the boundary point 
\[ (a,a')=\varphi(0,(a_1',0,\ldots,0),(0,\ldots,0,a_d))\in \partial_1 \bcalM_2(h), \quad a_1'>0,\quad a_d\geq 0,\]
the dual of the inward normal vector at $(a,a')$ is given by $a_1'dx_1+a_d dy_d$. Hence
\[ \begin{split}
	o(\partial_1\bcalM_2(h))_{(a,a')}&=-dy_1\cdots dy_idx_{i+1}\cdots dx_d,\\
	o^*_{\R^d\times \R^d}(\partial_1\bcalM_2(h))_{(a,a')}&= (-1)^{d-i+1} dx_1\cdots dx_idy_{i+1}\cdots dy_d\\
		&=(-1)^{(i+1)d+1}o^*_{\R^d}(\calA_r(h))_a\wedge o^*_{\R^d}(\calD_r(h))_{a'}.
\end{split} \]
\end{proof}
%%%%%%%%%%%%%%%%%%%%

%%%%%%%%%%%%%%%%%%%%%%%%%%%%%%
\subsection{Orientations of some faces of $\partial\bConf_{n}(M)$}\label{ss:ori_face}

Now assume that $d=3$. We shall describe the orientation of the face $\partial_{ij}\bConf_n(M):=\partial_{\{i,j\}}\bConf_n(M)$ induced from the standard orientation $o(M)_{x_1}\wedge o(M)_{x_2}\wedge\cdots\wedge o(M)_{x_{n}}$ of $\Conf_n(M)$. Let 
\[
  \Delta_{ij}=\{(x_1,x_2,\ldots,x_n)\in M^n; x_i=x_j\}.
 \]
The interior of the face $\partial_{ij}\bConf_n(M)$ is an open subset of $\partial B\ell_{\Delta_{ij}}(M^n)$. By definition of blow-up, the boundary of $B\ell_{\Delta_{ij}}(M^n)$ is the normal sphere bundle of the submanifold $\Delta_{ij}$. More precisely, let $N_{\Delta_{ij}}$ be the total space of the normal bundle of $\Delta_{ij}$. By identifying a small tubular neighborhood of $\Delta_{ij}$ with that of the zero section of $N_{\Delta_{ij}}$, we may identify a small collar neighborhood of $\partial B\ell_{\Delta_{ij}}(M^n)$ with that of $\partial B\ell_0(N_{\Delta_{ij}})$. 

A framing $\tau:TM\to \R^3\times M$ induces a trivialization $\phi_{ij}:N_{\Delta_{ij}}\to \R^3\times \Delta_{ij}$. This is smoothly extended to a trivialization $\overline{\phi}_{ij}:B\ell_0(N_{\Delta_{ij}})\to B\ell_0(\R^3)\times \Delta_{ij}$. Let $\omega_{p-1}$ denote the closed $(p-1)$-form on $B\ell_0(\R^p)$ that is the pullback of the $SO_p$-invariant volume form $\sum_{i=1}^p (-1)^{i-1}x_i\,dx_1\wedge\cdots\wedge \widehat{dx_i}\wedge\cdots\wedge dx_p$ on $S^{p-1}$ by the natural map $B\ell_0(\R^p)\to S^{p-1}$ (see Appendix~\ref{s:blow-up}). If $i<j$, let 
\[ o(\Delta_{ij})_{\vec{x}}=(-1)^{j-1}o(M)_{x_1}\wedge \cdots \wedge o(\Delta_M)_{(x_i,x_j)}\wedge \cdots \wedge 
\widehat{o(M)}_{x_j}\wedge\cdots\wedge o(M)_{x_n}.\]
Now we orient $\partial B\ell_{\Delta_{ij}}(M^n)=\partial B\ell_0(N_{\Delta_{ij}})$ as follows.
\[ 
  o(\partial B\ell_0(N_{\Delta_{ij}})) = \overline{\phi}^*_{ij}(\omega_2 \wedge o(\Delta_{ij})).
\]
This is the one induced from the standard orientation $o(M)_{x_1}\wedge o(M)_{x_2}\wedge\cdots\wedge o(M)_{x_{n}}$ of $\Conf_n(M)$. Indeed, if $x_i=x_j$, we have
\[ 
  (du_j^{(1)}-du_i^{(1)})\wedge(du_j^{(2)}-du_i^{(2)})\wedge(du_j^{(3)}-du_i^{(3)})\wedge o(\Delta_{ij})_{\vec{x}}= 2^3 o(M^n)_{\vec{x}},
\]
where $(u_k^{(1)},u_k^{(2)},u_k^{(3)})$ is a Euclidean local coordinate around $x_k$ and $o(M)_{x_k}=du_k^{(1)}du_k^{(2)}du_k^{(3)}$ and $o(\Delta_M)_{(x_i,x_j)}=(du_i^{(1)}+du_j^{(1)})\wedge(du_i^{(2)}+du_j^{(2)})\wedge(du_i^{(3)}+du_j^{(3)})$. The left hand side of the above expression also gives an orientation of $N_{\Delta_{ij}}$ and the orientation induced from it on the unit sphere bundle of $N_{\Delta_{ij}}$ is $\omega_2\wedge o(\Delta_{ij})$. 

%%%%%%%%%%%%%%%%%%%%%%%%%%%%%%
\subsection{Standard co-orientation of $\calM_\Gamma$ from graph orientation}\label{sss:std_ori}

We shall first give another definition of $\calM_\Gamma(\vec{f})$ using $\calM_2(f)$ and $\calN_{pq}(f)$. For a graph $\Gamma$ without bivalent vertices, the space $\calM_\Gamma(\vec{f})$ can be defined as the intersection of submanifolds of $\Conf_n(M)$, as follows. Suppose for simplicity that the separated edges of $\Gamma$ are labeled $1,2,\ldots,a$. Let $\pi_{ij}:\Conf_n(M)\to \Conf_2(M)$ be the projection $(x_1,\ldots,x_n)\mapsto (x_i,x_j)$ and let $\Theta_\ell$ and $H_\ell$ be the submanifolds of $\Conf_n(M)$ defined by
\[	\Theta_\ell=\pi_{ij}^{-1}(\calM_2(f_\ell)),\quad
	H_\ell=\pi_{ij}^{-1}(\calN_{p_\ell q_\ell}(f_\ell)),
\]
where $i=\sigma(\ell),j=\tau(\ell)$. Their codimensions are $\codim\Theta_\ell=2$, $\codim H_\ell=3-i(p_\ell)+i(q_\ell)=3-\eta_\ell$. Then we have
\[ \calM_\Gamma(\vec{f})=\bigcap_{j=1}^a H_{j}\cap\bigcap_{j=a+1}^m\Theta_j, \]
where the intersection is transversal. 

Let $o^*_{\Conf_2(M)}(\Theta_\ell)\in\Omega^2_{\mathrm{dR}}(B)$ and $o^*_{\Conf_2(M)}(H_\ell)\in\Omega^{3-\eta_\ell}_\mathrm{dR}(B)$ be differential forms on a neighborhood $B$ of a point on the crossing $\calM_\Gamma(\vec{f})$ in $\Conf_n(M)$, defined by
\[ o^*_{\Conf_n(M)}(\Theta_\ell)=\pi_{ij}^* o^*_{\Conf_2(M)}(\calM_2(f_\ell)),\quad o^*_{\Conf_n(M)}(H_\ell)=\pi_{ij}^* o^*_{\Conf_2(M)}(\calN_{p_\ell q_\ell}(f_\ell)).\]
We represent co-orientation of $\calM_\Gamma(\vec{f})$ by a wedge product of $o^*_{\Conf_n(M)}(H_j)$'s and $o^*_{\Conf_n(M)}(\Theta_j)$'s. We now define the coorientations for the graphs that are relevant.

\subsubsection{Graphs in $\calG_{2k,3k}^0(\vec{C})$, $\calG_{2k,3k,(2,1,\ldots,1)}^0(\vec{C})$, $\calG_{2k,3k,(0,1,\ldots,1)}^0(\vec{C})$}
Now we assume that $\Gamma\in\calG_{2k,3k}^0(\vec{C})$, so that $\codim\Theta_j=\codim{H_j}=2$. We shall define a standard co-orientation of $\calM_\Gamma(\vec{f})$ in a product of $M$ from the labels and the edge orientations of $\Gamma$, as follows. The labels of trivalent vertices determine the correspondence between $V(\Gamma)$ and the coordinate $(x_1,x_2,\ldots,x_{2k})$. The edge orientation determine which of $\pi_{ij}$ and $\pi_{ji}$ is used to define $\Theta_\ell$ or $H_\ell$. Then we define the standard co-orientation of $\calM_\Gamma(\vec{f})$ by the formula
\[ o^*_{\Conf_{2k}(M)}(\calM_\Gamma(\vec{f}))=\bigwedge_{j=1}^ao^*_{\Conf_{2k}(M)}(H_j)\wedge\bigwedge_{j=a+1}^{3k} o^*_{\Conf_{2k}(M)}(\Theta_j)\in\Omega^{6k}_\mathrm{dR}(B). \]
Since $\codim\Theta_j$ and $\codim H_j$ are even, the order of wedge product does not matter. This depends only on the orientation $o(\Gamma)$ of $\Gamma$. This gives $\#\calM_\Gamma(\vec{f})$ in \S\ref{ss:count}.

The same rule equally works for graphs in $\calG_{2k,3k,(2,1,\ldots,1)}^0(\vec{C})$, $\calG_{2k,3k,(0,1,\ldots,1)}^0(\vec{C})$ etc. without bivalent vertices, since in that case only one $H_j$ is odd codimensional, so again the coorientation of $\calM_\Gamma(\vec{f})$ is canonically determined from the graph orientation by the same formula. 

\subsubsection{Graphs in $\calG_{2k,3k,\vec{\eta}}^0(\vec{C})$, $\eta_{j_2}=2$, $\eta_{j_0}=0$}
For a graph in $\calG_{2k,3k,\vec{\eta}}^0(\vec{C})$ without bivalent vertices such that there is exactly one $j$ with $\eta_j=2$, exactly one $j$ with $\eta_j=0$ and otherwise $\eta_j=1$, let $j_2$ and $j_0$ be such that $\eta_{j_2}=2$, $\eta_{j_0}=0$. Then we define a standard co-orientation $o^*_{\Conf_{2k}(M)}(\calM_\Gamma(\vec{f}))$ of $\calM_\Gamma(\vec{f})$ by
\[ o^*_{\Conf_{2k}(M)}(H_{j_0})\wedge o^*_{\Conf_{2k}(M)}(H_{j_2})\wedge\bigwedge_{{{1\leq j\leq a}\atop{j\neq j_0,j_2}}}o^*_{\Conf_{2k}(M)}(H_j)\wedge\bigwedge_{j=a+1}^{3k} o^*_{\Conf_{2k}(M)}(\Theta_j).
 \]

\subsubsection{Graphs in $\calG_{2k,3k,(1,1,\ldots,1)}^0(\vec{C})$ with one bivalent vertex}
We also consider co-orientations of (not yet defined) $\calM_\Gamma(\vec{f})$ for graphs $\Gamma\in\calG_{2k,3k,(1,1,\ldots,1)}^0(\vec{C})$ with only one bivalent vertex. For three possibilities for the position of the bivalent vertex in $\Gamma$, we define $\calM_\Gamma(\vec{f})$ and its standard co-orientation as follows.

%%%
{(1)} $\Gamma=\fig{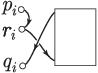}$, $i(p_i)=i(r_i)+1=i(q_i)+2$. Let $\Gamma'$ be the graph obtained from $\Gamma$ by removing the segment $\partial^{(i)}(p_i,r_i)$. In this case, we define 
\[ \calM_\Gamma(\vec{f})=\bbcalM{f_i}{p_i}{r_i}\times \calM_{\Gamma'}(\vec{f}). \]
If a point $b\in \bbcalM{f_i}{p_i}{r_i}$ is specified, we may consider a coorientation of $\{b\}\times \calM_{\Gamma'}(\vec{f})$ in $\Conf_{2k}(M)$ by identifying it with $\calM_{\Gamma'}(\vec{f})$. Then we define
\[ o^*_{\Conf_{2k}(M)}(\calM_\Gamma(\vec{f}))_{b\times (x_1,\ldots,x_{2k})}
	=\ve_f(p_i,r_i)_b\, o^*_{\Conf_{2k}(M)}(\calM_{\Gamma'}(\vec{f}))_{(x_1,\ldots,x_{2k})}. \]

%%%
{(2)}  $\Gamma=\fig{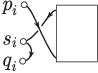}$, $i(q_i)=i(s_i)-1=i(p_i)-2$. Let $\Gamma'$ be the graph obtained from $\Gamma$ by removing the segment $\partial^{(i)}(s_i,q_i)$. In this case, we define
\[ \calM_\Gamma(\vec{f})=\bbcalM{f_i}{s_i}{q_i}\times \calM_{\Gamma'}(\vec{f}). \]
If a point $b\in \bbcalM{f_i}{s_i}{q_i}$ is specified, we may consider a coorientation of $\{b\}\times \calM_{\Gamma'}(\vec{f})$ in $\Conf_{2k}(M)$ by identifying it with $\calM_{\Gamma'}(\vec{f})$. Then we define
\[ o^*_{\Conf_{2k}(M)}(\calM_\Gamma(\vec{f}))_{b\times (x_1,\ldots,x_{2k})}
	=\ve_f(s_i,q_i)_b\, o^*_{\Conf_{2k}(M)}(\calM_{\Gamma'}(\vec{f}))_{(x_1,\ldots,x_{2k})}. \]

%%%
{(3)}  $\Gamma=\fig{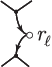}$. In this case, $o^*_{\Conf_{2k}(M)}(\calM_\Gamma(\vec{f}))$ is determined by the intersection of $H'_\ell=\pi_{ij}^{-1}(\calA_{r_\ell}(f_\ell)\times \calD_{r_\ell}(f_\ell))$ (codimension $i(r_\ell)+(3-i(r_\ell))=3$) with the intersection of $\Theta_j$'s and $H_j$'s of codimension 2. We define $o^*_{\Conf_{2k}(M)}(\calM_\Gamma(\vec{f}))$ by
\[o^*_{\Conf_{2k}(M)}(H_{j_2})\wedge \bigwedge_{{{1\leq j\leq a}\atop{j\neq j_2}}} o^*_{\Conf_{2k}(M)}(H_j)\wedge o^*_{\Conf_{2k}(M)}(H_\ell')\wedge \bigwedge_{{{a+1\leq j\leq 3k}\atop{j\neq \ell}}} o^*_{\Conf_{2k}(M)}(\Theta_j).\]

%%%%%%%%%%%%%%%%%%%%%%%
\subsection{Induced coorientation on $\partial\bcalM_\Gamma$} 

Now we define boundary operators, which formally describe the boundary of moduli space of flow graphs. 
We define a linear map $d:\calG_{n,m,\vec{\eta}}(\vec{C})\to 
	\bigoplus_{j=1}^{\ell}\calG_{n-1,m-1,(\eta_1,\ldots,\widehat{\eta_j},\ldots,\eta_\ell)}
		(C_*^{(1)},\ldots,\widehat{C}_*^{(j)},\ldots,C_*^{(\ell)})$ by
\[ d(\Gamma,o)=\sum_{e\in\Comp(\Gamma)}(\Gamma/e,\mathrm{induced\ ori}), \]
where for $e_j=(u,v)$ ($u,v$: vertices) the induced orientation of $\Gamma/e_j$ is formally given by $\iota(v^*)(v_1\wedge\cdots\wedge v_n)
	\wedge (e_1^+\wedge e_1^-)\wedge\stackrel{j}{\hat{\cdots}}\wedge (e_m^+\wedge e_m^-)$, where $v^*$ is the dual of $v$ with respect to the standard inner product and $\iota$ is the interior product. Also, let $d'\Gamma=\sum_{e\in\Comp(\Gamma)\cup\Se(\Gamma)}d'_e\Gamma$, where
\begin{align*}
d_e'\fig{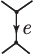}&=\sum_{r_i\in P_*^{(i)}}(-1)^{i(r_i)+1}\fig{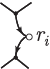}\qquad (\beta(i)=e),\\
d_e'\fig{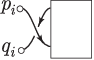}&=\sum_{{{r_i\in P_*^{(i)}}\atop{i(r_i)=i(p_i)-1}}}(-1)^{i(q_i)+1}\,\fig{d_4.eps}
	+\sum_{{{s_i\in P_*^{(i)}}\atop{i(s_i)=i(q_i)+1}}}(-1)^{i(s_i)+1}\fig{d_5.eps},
\end{align*}
$(\beta(i)=e)$, where the orientation of $d_e'\Gamma$ is the naturally induced one. 

Let $\calG_{n,m,\Sigma(1,1,\ldots,1)}^0(\vec{C})=\calG_{n,m,(2,1,\ldots,1)}^0(\vec{C})\cup \calG_{n,m,(1,2,1,\ldots,1)}^0(\vec{C})\cup \cdots\cup \calG_{n,m,(1,1,\ldots,2)}^0(\vec{C})$. Let $\bar{b}_\Gamma:\bcalM_\Gamma(\vec{f})\to \bConf_{2k}(M)$ be the map which gives the positions of the $2k$ trivalent vertices in a flow-graph (Remark~\ref{rem:b_emb}). Let $\bcalM_\Gamma^\mathrm{hi}(\vec{f})=\bar{b}_\Gamma^{-1}\partial^\mathrm{hi}\bConf_{2k}(M)$. 

\begin{Prop}\label{prop:dbM}
Suppose $d=3$ and $(\vec{f},\vec{\mu})$ is generic as in Proposition~\ref{prop:M_G-1-mfd}. For $\Gamma\in\calG_{2k,3k,\Sigma(1,\ldots,1)}^0(\vec{C})$ without bivalent vertices, there is a natural diffeomorphism
\[ \partial\bcalM_\Gamma(\vec{f})\cong \bcalM_{(d+d')\Gamma}(\vec{f})
	\tcoprod \bcalM_\Gamma^\mathrm{hi}(\vec{f}) \]
of oriented 0-manifolds (for some orientation of $\bcalM_\Gamma^\mathrm{hi}(\vec{f})$), where $\bcalM$ of a sum of graphs means the sum of $\bcalM$ for graphs in the sum.
\end{Prop}

\begin{proof}
We shall compare the co-orientation of a face of $\partial\bcalM_\Gamma(\vec{f})$ induced from $o^*_{\Conf_{2k}(M)}(\calM_\Gamma(\vec{f}))$ and the standard one of the same face of $\partial\bcalM_\Gamma(\vec{f})$ fixed in \S\ref{sss:std_ori}. 

Suppose that $\Gamma\in\calG_{2k,3k,\vec{\eta}}^0(\vec{C})$ has no bivalent vertex and that there is only one number $j$ with $\eta_j=2$ and $\eta_\ell=1$ for $\ell \neq j$. Let $j_2$ be such that $\eta_{j_2}=2$. As in \S\ref{sss:std_ori}, we consider $\calM_\Gamma(\vec{f})$ as the intersection of the chains $H_j$'s and $\Theta_j$'s in $\Conf_{2k}(M)$. 

Choose a number $\ell$ such that $1\leq \ell\leq 3k$ and let 
\[ \Sigma_\ell=\left\{\begin{array}{ll}
	\bigcap_{{{1\leq j\leq a}\atop{j\neq \ell}}}H_j\cap\bigcap_{j=a+1}^{3k}\Theta_j
		& \mbox{if $1\leq \ell\leq a$}\\
	\bigcap_{j=1}^a H_j\cap \bigcap_{{{a+1\leq j\leq 3k}\atop{j\neq \ell}}}\Theta_j
		& \mbox{if $a+1\leq \ell\leq 3k$}
	\end{array}\right. \]
Then by Proposition~\ref{prop:M_mfd}, $\codim\Sigma_\ell=\codim\calM_\Gamma(\vec{f})-\codim H_\ell=(2kd-(2k-3k)d-3k-1)-(d+\eta_\ell)=6k-4-\eta_\ell$.

First, we consider the contribution of $\partial^\mathrm{pri}\bConf_{2k}(M)$. We check that the contribution of the principal face is $\#\calM_{d\Gamma}(\vec{f})$. We consider the principal face corresponding to the collapse of the $\ell$-th edge of $\Gamma$. By convention, 
\[ o(\bcalM_2(f_\ell))_{(x,y)}=(-df_\ell)_y\wedge o(\Delta_M)_{(x,x)}+O(d(x,y)), \]
where $d(x,y)$ is the geodesic distance. Let $\xi=-\mathrm{grad}\,f_\ell$. The orientation induced on the face $\Delta_M$ of $\partial\bcalM_2(f_\ell)$ is
\[ \iota(-\xi_x\oplus \xi_y)(-df_\ell)_y\wedge o(\Delta_M)_{(x,x)}=o(\Delta_M)_{(x,x)}\quad (\mbox{if }x=y). \]
This implies that the coorientation of the boundary of $\bcalM_2(f_\ell)$ on $\partial B\ell_{\Delta_M}(M^2)$ is given by $\omega_2$, showing that the principal face contribution is $\#\calM_{d\Gamma}(\vec{f})$.

Let $X$ be one of the graphs that appear in the sum $d_{\beta(k)}'\Gamma$. We shall describe the co-orientation of the face $\calS_X$ of $\partial\bcalM_\Gamma(\vec{f})$ corresponding to $X$ induced from the standard co-orientation of $\calM_\Gamma(\vec{f})$ in $\Conf_{2k}(M)$.

(1) $X=\fig{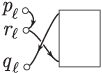}$. By Lemma~\ref{lem:ind_ori_D}, (\ref{eq:ind_ori_boundary}) and (\ref{eq:ind_ori_bd_int}), the co-orientation of $\calS_X$ induced from the standard one
\begin{equation}\label{eq:std_ori_1}
 o^*_{\Conf_{2k}(M)}(\calM_\Gamma(\vec{f}))
	=o^*_M(\calA_{q_\ell}(f_\ell))\wedge o^*_M(\calD_{p_\ell}(f_\ell))\wedge o^*_{\Conf_{2k}(M)}(\Sigma_\ell) 
\end{equation}
is given by
\[ \begin{split}
	&(-1)^{d-1}(-1)^{2kd-1}(-1)^{i(r_\ell)+1+\eta_\ell}\ve_f(p_\ell,r_\ell)_b\, o^*_M(\calA_{q_\ell}(f_\ell))\, o^*_M(\calD_{r_\ell}(f_\ell))\, o^*_{\Conf_{2k}(M)}(\Sigma_\ell)\\
	&=(-1)^{i(q_\ell)+1}\ve_f(p_\ell,r_\ell)_b\, o^*_{\Conf_{2k}(M)}(\calM_{\Gamma'}(\vec{f}))=(-1)^{i(q_\ell)+1}\,o^*_{\Conf_{2k}(M)}(\calM_X(\vec{f})).
	\end{split}\]

(2) $X=\fig{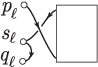}$. By Lemma~\ref{lem:ind_ori_A}, (\ref{eq:ind_ori_boundary}) and (\ref{eq:ind_ori_bd_int}), the co-orientation of $\calS_X$ induced from the standard one (\ref{eq:std_ori_1}) is given by
\[ \begin{split}
	&(-1)^{d-1}(-1)^{2kd-1}(-1)^{3-i(p_\ell)+\eta_\ell}\ve_f(s_\ell,q_\ell)_b\, o^*_M(\calA_{s_\ell}(f_\ell))\, o^*_M(\calD_{p_\ell}(f_\ell))\, o^*_{\Conf_{2k}(M)}(\Sigma_\ell)\\
	&=(-1)^{i(s_\ell)+1}\ve_f(s_\ell,q_\ell)_b\, o^*_{\Conf_{2k}(M)}(\calM_{\Gamma'}(\vec{f}))=(-1)^{i(s_\ell)+1}o^*_{\Conf_{2k}(M)}(\calM_X(\vec{f})).
\end{split}\]

(3) $X=\fig{d_2_2.eps}$. The induced co-orientation on the boundary is as in Lemma~\ref{lem:ind_ori_M2}, which differs from the standard coorientation by $(-1)^{i(r_\ell)+\eta_\ell}=(-1)^{i(r_\ell)+1}$.

Now we have seen that the signs in the formula of the definition of $d'$ are consistent with the induced co-orientations on the boundary of $\bcalM_\Gamma(\vec{f})$.\end{proof}

%\clearpage

%%%%%%%%%%%%%%%%%%%%%%%%%%%%%%
%%%%%%%%%%%%%%%%%%%%%%%%%%%%%%
\mysection{Independence of combinatorial propagator}{s:indep_g}

In the definition of $Z_{2k,3k}(\vec{f})$, a sequence $\vec{g}=(g_1,\ldots,g_{3k})$ of combinatorial propagators for $\vec{C}$ is chosen. In this section, we shall prove that $Z_{2k,3k}(\vec{f})$ does not depend on the choice of $\vec{g}$. Recall that $Z_{2k,3k}(\vec{f})=\Tr_{\vec{g}}(\widetilde{\gamma}_{2k,3k})$ for $\widetilde{\gamma}_{2k,3k}=\sum_{\Gamma\in\calG_{2k,3k}^0(\vec{C})}\#\calM_\Gamma(\vec{f})\,\Gamma\in\calG_{2k,3k}(\vec{C})$. We shall prove the following lemma.
\begin{Lem}\label{lem:indep_g}
$Z_{2k,3k}(\vec{f})=\Tr_{\vec{g}}(\widetilde{\gamma}_{2k,3k})$ does not depend on the choice of $\vec{g}$.
\end{Lem}

%%%%%%%%%%%%%%%%%%%%%%%%%%%%%%
\subsection{Boundary strata of $\bcalM_\Gamma$}

\begin{Prop}\label{prop:dM=Md}
Let $\Gamma$ be a graph in $\calG_{2k,3k,\Sigma(1,\ldots,1)}^0(\vec{C})$ without bivalent (i.e. white) vertices. For a permutation $\sigma\in\mathfrak{S}_{3k}$ and for a subset $\tau\subset E(\Gamma)$, let $\Gamma_\sigma^\tau$ denote the labeled graph obtained from $\Gamma$ by permuting the labels of edges by $\sigma$ and reversing the orientations of all the edges in $\tau$. For $\vec{f}$ generic, we have
\[ \sum_{{\sigma\in\mathfrak{S}_{3k}}}\sum_{{\tau\subset E(\Gamma)}}\#\bcalM_{(d+d')\Gamma_\sigma^\tau}(\vec{f})
	=\sum_{{\sigma\in\mathfrak{S}_{3k}}}\sum_{{\tau\subset E(\Gamma)}}\#\partial\bcalM_{\Gamma_\sigma^\tau}(\vec{f})
	=0. \]
\end{Prop}

For the proof of Proposition~\ref{prop:dM=Md}, we need two lemmas, which are analogues of Kontsevich's lemma \cite[Lemma~2.2]{Ko1}. In the following $H$ is a subgraph $\Gamma\in \calG_{2k,3k,\Sigma(1,\ldots,1)}^0(\vec{C})$ with only compact edges.

\begin{Lem}\label{lem:symmetry}
Suppose that $H$ has a bivalent black vertex $a$ and that the edges of $H$ including the vertex $a$ are $(b,a)$ and $(a,c)$ where $b$ and $c$ are both black vertices. Let $H'$ be the labeled graph obtained from $H$ by exchanging labels for edges $(b,a)$ and $(a,c)$. Suppose that $\dim\calM_H^\loc(\vec{\gamma})=\dim\calM_{H'}^\loc(\vec{\gamma})=0$. Then 
\[ \#\calM_H^\loc(\vec{\gamma})+\#\calM_{H'}^\loc(\vec{\gamma})=0. \]
\end{Lem}
\begin{proof}
Let $n=|V(H)|$ and let $(x_1,\ldots,x_n)\in \bConf_n^\loc(\R^3)$ be a point on $\calM_H^\loc(\vec{\gamma})$. Suppose that the vertices $a$, $b$ and $c$ of $H$ correspond to $x_\alpha$, $x_\beta$ and $x_\gamma$ respectively. Then consider the automorphism $s:\bConf_n^\loc(\R^3)\to \bConf_n^\loc(\R^3)$ which sends $x_\alpha$ to $x_\beta+x_\gamma-x_\alpha$ and fixes other points. Then $s$ exchanges $\calM_H^\loc(\vec{\gamma})$ and $\calM_{H'}^\loc(\vec{\gamma})$. Put $y_\alpha=x_\beta+x_\gamma-x_\alpha$ and $\theta(\ell)_{(x,y)}=o^*_{\R^3\times \R^3}(\Theta_\ell(\gamma_\ell))_{(x,y)}$ (see Definition~\ref{def:Mloc}). Let $V_\alpha\in \bigwedge^3 T_{x_\alpha}\R^3$ and $V_{\alpha}'\in \bigwedge^3 T_{y_\alpha}\R^3$ be nontrivial elements that give the orientation of $\R^3$. Then the evaluation with $V_\alpha$ (resp. with $V_{\alpha}'$) gives a map $\bigwedge^k (T_{x_\alpha}\R^3\oplus T_{x_\beta}\R^3\oplus T_{x_\gamma}\R^3)\to \bigwedge^{k-3}(T_{x_\beta}\R^3\oplus T_{x_\gamma}\R^3)$ (resp. $\bigwedge^k (T_{y_\alpha}\R^3\oplus T_{x_\beta}\R^3\oplus T_{x_\gamma}\R^3)\to \bigwedge^{k-3}(T_{x_\beta}\R^3\oplus T_{x_\gamma}\R^3)$) and we have
\[ \begin{split}
  &\langle \theta(\ell)_{(x_\beta,x_\alpha)}\wedge \theta(\ell')_{(x_\alpha,x_\gamma)}, V_\alpha\rangle =\langle s^*(\theta(\ell)_{(y_\alpha,x_\gamma)}\wedge \theta(\ell')_{(x_\beta,y_\alpha)}), V_{\alpha}\rangle\\
  =&\,\langle \theta(\ell)_{(y_\alpha,x_\gamma)}\wedge \theta(\ell')_{(x_\beta,y_\alpha)}, ds_* V_{\alpha}\rangle=-\langle \theta(\ell')_{(x_\beta,y_\alpha)}\wedge \theta(\ell)_{(y_\alpha,x_\gamma)}, V_{\alpha}'\rangle.
\end{split} \]
This induces the desired identity.
\end{proof}

The following lemma can be proved in the same way as Lemma~\ref{lem:symmetry}. 
\begin{Lem}\label{lem:symmetry2}
Suppose that $H$ has a bivalent black vertex $a$ and that the edges of $H$ including the vertex $a$ are $(a,b)$ and $(a,c)$ where $b$ and $c$ are both black vertices. Let $H'$ be the labeled graph obtained from $H$ by exchanging labels for edges $(a,b)$ and $(a,c)$ and reversing the orientations of both edges. Suppose that $\dim\calM_H^\loc(\vec{\gamma})=\dim\calM_{H'}^\loc(\vec{\gamma})=0$. Then 
\[ \#\calM_H^\loc(\vec{\gamma})+\#\calM_{H'}^\loc(\vec{\gamma})=0. \]
\end{Lem}

\begin{Lem}\label{lem:dilation}
Suppose that $H$ has a univalent black vertex $a$ and $|V(H)|\geq 3$. Then $\calM_H^\loc(\vec{\gamma})$ admits a smooth free $\R$-action. 
\end{Lem}
\begin{proof}
Suppose that the edge of $H$ including the vertex $a$ is $(a,b)$ where $b$ is a black vertex. There is a dilation of the linear trajectory corresponding to $(a,b)$ in $\calM_H^\loc(\vec{\gamma})$, which fixes points in $V(H)-\{a\}$. Since $|V(H)|\geq 3$, this $\R$-action is free in $\bConf_n^\loc(\R^3)$ and gives a desired $\R$-action.
\end{proof}

\begin{proof}[Proof of Proposition~\ref{prop:dM=Md}]
The assumption $\Gamma\in\calG_{2k,3k,\Sigma(1,\ldots,1)}^0(\vec{C})$ implies that $\bcalM_\Gamma(\vec{f})$ is 1-dimensional by Proposition~\ref{prop:M_mfd}. For a subset $A$ of $V(\Gamma)$, let $\Gamma_A$ denote the subgraph of $\Gamma$ such that $V(\Gamma_A)=A$ and $E(\Gamma_A)$ consists of all edges of $E(\Gamma)$ between points in $A$. Suppose $\Gamma$ and $A$ are such that $E(\Gamma_A)=\Comp(\Gamma_A)$. Let $\vec{f}_A\subset \vec{f}$ be the subsequence corresponding to the subset $E(\Gamma_A)\subset E(\Gamma)$ and let $-\mathrm{grad}\,\vec{f}_A=(-\mathrm{grad}\,f_{i_1},\ldots,-\mathrm{grad}\,f_{i_{|A|}})$, where $i_1,\ldots,i_{|A|}$ are the labels of the edges in $\Gamma_A$. According to the description of $\partial\bConf_n(M)$ in \S\ref{ss:FM}, the face of $\partial\bcalM_\Gamma(\vec{f})$ coming from the face $\partial_A\bConf_n(M)$ is diffeomorphic, through the map induced from a trivialization $TM\stackrel{\approx}{\to} M\times\R^3$, to
\[ \bcalM_{\Gamma/\Gamma_A}(\vec{f}\setminus\vec{f}_A)\times \calM_{\Gamma_A}^\loc(-\mathrm{grad}\,\vec{f}_A), \]
which is at most an oriented 0-manifold. Here, we consider $\calM_{\Gamma_A}^\loc(-\mathrm{grad}\,\vec{f}_A)$ as a subspace of the $\bConf_{|A|}^\loc(\R^3)$-bundle over $M-\bigcup_{i\in A}\Sigma(f_i)$. Since $\Gamma$ is a graph in $\calG_{2k,3k,\Sigma(1,\ldots,1)}^0(\vec{C})$, the subgraph $\Gamma_A$ must have bivalent or univalent black vertices or none of them. If $|A|\geq 3$ and $\Gamma_A$ has a bivalent black vertex and if $\dim\calM_{\Gamma_A}^\loc(-\mathrm{grad}\,\vec{f}_A)=0$, then we have $\#\calM_{\Gamma_A}^\loc(-\mathrm{grad}\,\vec{f}_A)+\#\calM_{\Gamma_A'}^\loc(-\mathrm{grad}\,\vec{f}_A)=0$ by Lemma~\ref{lem:symmetry}. Hence by taking the sum over $\sigma$ and $\tau$, the contributions of $\Gamma_A$ with bivalent vertex cancel with each other. 
If $|A|\geq 3$ and if $\Gamma_A$ has a univalent black vertex, then $\calM_{\Gamma_A}^\loc(-\mathrm{grad}\,\vec{f}_A)\cong \emptyset\times\R=\emptyset$ by Lemma~\ref{lem:dilation}. 

If $E(\Gamma_A)=\emptyset$, then $\Gamma/\Gamma_A$ has a vertex of valence $\geq 6$. In this case, one can see by Proposition~\ref{prop:M_mfd} that $\bcalM_{\Gamma/\Gamma_A}(\vec{f}\setminus\vec{f}_A)=\bcalM_{\Gamma/\Gamma_A}(\vec{f})=\emptyset$ if $\vec{f}$ is generic. 

Finally, we must check that there are no contribution of $\partial_{A\cup\{\infty\}}\bConf_{2k}(M)$ for generic $\vec{f}$. This has been checked in \cite[Lemma~6.6, 6.7]{Sh}. We outline the proof with our notations. We may identify the interior of $\partial_{A\cup\{\infty\}}\bConf_{2k}(M)$ with $\Conf_{2k-j}(M)\times \Conf_j^\infty(\R^3)$, where
\[ \Conf_j^\infty(\R^3)=\Bigl\{(y_1,\ldots,y_j)\in\Conf_j(\R^3);\sum_{\ell=1}^j\|y_\ell\|^2=1\Bigr\}. \]
Let $f_j^\infty:\R^3\to \R$ ($j=1,2,\ldots,3k$) be the linear map such that $\varphi_\infty^*f_j^\infty$ agrees with $f_j$ near $\infty_M$. Let $\vec{f}^\infty=(f_1^\infty,\ldots,f_{3k}^\infty)$ and let $B=V(\Gamma)\setminus A$. Suppose that $E(\Gamma/\Gamma_B)=\Comp(\Gamma/\Gamma_B)$. Let $\calM_{\Gamma/\Gamma_B}^\infty(\vec{f}^\infty\setminus\vec{f}_B^\infty)$ be the space of linear graphs $(\Gamma,\Gamma_B)\to (\R^3,\{0\})$ modulo the dilation of $\R^3$ whose edge not in $E(\Gamma_B)$ labeled $\ell$ follows the negative gradient of $f_\ell^\infty$. Then (the interior of) the face of $\partial\bcalM_\Gamma(\vec{f})$ coming from $\partial_{A\cup\{\infty\}}\bConf_{2k}(M)$ is diffeomorphic to $\calM_{\Gamma_B}(\vec{f}_B)\times \calM_{\Gamma/\Gamma_B}^\infty(\vec{f}^\infty\setminus\vec{f}_B^\infty)$.
The configuration space $\Conf_j^\infty(\R^3)$ is $(3j-1)$-dimensional. If the number of edges in $E(\Gamma)$ that intersect both $V(\Gamma_A)$ and $V(\Gamma_B)$ is $m$, then the codimension of $\calM_{\Gamma/\Gamma_B}^\infty(\vec{f}^\infty\setminus\vec{f}_B^\infty)$ is $2\times \frac{3j+m}{2}=3j+m$. Since $m$ is non-negative, the codimension exceeds $\dim \Conf_j^\infty(\R^3)=3j-1$ and the moduli space $\calM_{\Gamma/\Gamma_B}^\infty(\vec{f}^\infty\setminus\vec{f}_B^\infty)$ must be empty. Hence there are no face of $\partial \bcalM_\Gamma(\vec{f})$ in $\partial_{A\cup\{\infty\}}\bConf_{2k}(M)$. 

These together with Proposition~\ref{prop:dbM} imply the proposition.
\end{proof}

%%%%%%%%%%%%%%%%%%%%%%%%%%%%%%
\subsection{Independence of combinatorial propagator}

Let $\partial^{(i)}(x,y)$ denote the graph \fig{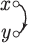} where $x,y\in P_*^{(i)}$. Let $\Gamma$ be a $\vec{C}$-colored graph having a separated edge $\beta(i)\in \Se(\Gamma)$ such that basis elements $x$ and $y\in P_*^{(i)}$ attached on the input and output white vertex, respectively. To specify that $\Gamma$ has such an edge, we will write $\Gamma=\Gamma(x,y)_i$. This notation allows us to express the graph obtained from $\Gamma$ by replacing $x$ and $y$ with $x'$ and $y'$ respectively, as $\Gamma(x',y')_i$. We will write $\Gamma(\emptyset,\emptyset)_i$ the graph obtained from $\Gamma$ by replacing the separated edge $\beta(i)$ with a compact edge. We denote by $\partial^{(i)}(p_i,r_i)*\Gamma(r_i,q_i)_i$ or $\Gamma(p_i,r_i)*\partial^{(i)}(r_i,q_i)$ the composition (one point union) of two graphs at the univalent vertices on which $r_i$ is attached, and $\partial_{p_ir_i}^{(i)}\in \Z$ is the coefficient in $\partial^{(i)}p_i=\sum_{r_i}\partial_{p_ir_i}^{(i)}r_i$.

\begin{proof}[Proof of Lemma~\ref{lem:indep_g}]
We prove the assertion for $\vec{g}=(g,g^{(2)},\ldots,g^{(m)})$ and $\vec{g}'=(g',g^{(2)},\ldots,g^{(m)})$, where $g$ and $g'$ are two combinatorial propagators for $(C_*^{(1)},\partial^{(1)})$. As mentioned in \S\ref{s:acyclic_cpx}, there exists an endomorphism $h\in\End_2(C_*^{(1)})$ such that $\partial^{(1)}h-h\partial^{(1)}=g-g'$. Then the difference $\Tr_{g,\ldots}(\widetilde{\gamma}_{2k,3k})-\Tr_{g',\ldots}(\widetilde{\gamma}_{2k,3k})$ equals
\begin{equation}\label{eq:Tr-Tr}
\begin{split}
	&\Tr_{g-g',\ldots}
		\Bigl[\sum_{{{\Gamma(p_1,q_1)_1}\atop{i(p_1)=i(q_1)+1}}}\#\calM_{\Gamma(p_1,q_1)_1}(\vec{f})\,\Gamma(p_1,q_1)_1\Bigr]\\
	&=\Tr_{\partial^{(1)}h-h\partial^{(1)},\ldots}
		\Bigl[\sum_{{{\Gamma(p_1,q_1)_1}}}\#\calM_{\Gamma(p_1,q_1)_1}(\vec{f})\,\Gamma(p_1,q_1)_1\Bigr]\\
	&=\Tr_{h,\ldots}
		\Bigl[\sum_{{{\Gamma(p_1,q_1)_1}}}\sum_{{{x_1\in P^{(1)}_*}\atop{i(x_1)=i(p_1)+1}}}
		\partial_{x_1p_1}^{(1)}\cdot\#\calM_{\Gamma(p_1,q_1)_1}(\vec{f})\,\Gamma(x_1,q_1)_1\\
	&\hspace{11mm}-\sum_{{{\Gamma(p_1,q_1)_1}}}\sum_{{{y_1\in P^{(1)}_*}\atop{i(y_1)=i(q_1)-1}}}
		\partial_{q_1y_1}^{(1)}\cdot\#\calM_{\Gamma(p_1,q_1)_1}(\vec{f})\,\Gamma(p_1,y_1)_1\Bigr]\\
	&=\Tr_{h,\ldots}
		\Bigl[\sum_{{{\Gamma(p_1,q_1)_1}}}\sum_{{{x_1}}}
		\#\calM_{\partial^{(1)}(x_1,p_1)*\Gamma(p_1,q_1)_1}(\vec{f})\,\Gamma(x_1,q_1)_1\\
	&\hspace{11mm}-\sum_{{{\Gamma(p_1,q_1)_1}}}\sum_{{{y_1}}}
		\#\calM_{\Gamma(p_1,q_1)_1*\partial^{(1)}(q_1,y_1)}(\vec{f})\,\Gamma(p_1,y_1)_1
		\Bigr].
\end{split}
\end{equation}
We show that (\ref{eq:Tr-Tr}) vanishes. We write $\vec{C}[j]=(C_*^{(1)},\ldots,\widehat{C}_*^{(j)},\ldots,C_*^{(m)})$ and $\vec{\eta}[j]=(\eta_1,\ldots,\widehat{\eta}_j,\ldots,\eta_m)$ for simplicity. If we define $d^*:\bigoplus_{j=1}^m\calG_{n-1,m-1,\vec{\eta}[j]}(\vec{C}[j])\to \calG_{n,m,\vec{\eta}}(\vec{C})$ by the coefficient of $1\otimes\Gamma$ in $\sum_{\Gamma'\in\calG_{n,m,\vec{\eta}}^0(\vec{C})}\Gamma'\otimes d\Gamma'$, then $\mathrm{Im}\,d^*$ is the span of the IHX-relation. 

Let $p_1',q_1'\in P^{(1)}_*$ be such that $i(p_1')=\ell$ and $i(q_1')=\ell-2$. By Proposition~\ref{prop:dM=Md}, the following expression vanishes:
\[ \begin{split}
	&\Tr_{h,\ldots}\Bigl[
		\kern-3mm\asum{\Gamma(p_1',q_1')_1\in}{\calG_{n,m,(2,1,\ldots,1)}^0(\vec{C})}
		\kern-6mm\#\calM_{(d+d')\Gamma(p_1',q_1')_1}(\vec{f})\,\Gamma(p_1',q_1')_1 \Bigr]\\
	&=\Tr_{h,\ldots}\Bigl[
		\kern-10mm\asum{\Gamma'(p_1',q_1')_1\in}{\prod_j\calG_{n-1,m-1,(2,1,\ldots,1)[j]}^0(\vec{C}[j])}
		\kern-12mm\#\calM_{\Gamma'(p_1',q_1')_1}(\vec{f})\,d^*\Gamma'(p_1',q_1')_1+
		\kern-5mm\asum{\Gamma(p_1',q_1')_1\in}{\calG_{n,m,(2,1,\ldots,1)}^0(\vec{C})}
		\kern-6mm\#\calM_{d'\Gamma(p_1',q_1')_1}(\vec{f})\,\Gamma(p_1',q_1')_1\Bigr]\\
	&=\Tr_{h,\ldots}\Bigl[
		\sum_{\Gamma(p_1',q_1')_1}\Bigl(\sum_{r_1'\in P^{(1)}_{\ell-1}}
                (-1)^{\ell-1}\#\calM_{\partial^{(1)}(p_1',r_1')*\Gamma(r_1',q_1')_1}(\vec{f})\\
        &\hspace{37mm}+(-1)^{\ell}\#\calM_{\Gamma(p_1',r_1')_1*\partial^{(1)}(r_1',q_1')}(\vec{f})
                \Bigr)\,\Gamma(p_1',q_1')_1\\
	&\hspace{11mm}+
		\sum_{\Gamma(p_1',q_1')_1}\asum{i\neq 1}{i\notin \Comp(\Gamma(p_1',q_1')_1)}\sum_{j=1}^3\asum{p_i\in P^{(i)}_j}{q_i\in P^{(i)}_{j-1}}\Bigl(\sum_{r_i\in P^{(i)}_{j-1}}(-1)^{j}\#\calM_{\partial^{(i)}(p_i,r_i)*\Gamma(p_1',q_1')_1(r_i,q_i)_i}(\vec{f})\\
        &\hspace{25mm}+\sum_{r_i\in P^{(i)}_j}(-1)^{j+1}\#\calM_{\Gamma(p_1',q_1')_1(p_i,r_i)_i*\partial^{(i)}(r_i,q_i)}(\vec{f})
		\Bigr)\,\Gamma(p_1',q_1')_1(p_i,q_i)_{i}\\
	&\hspace{11mm}+
		\sum_{\Gamma(p_1',q_1')_1}
		\asum{e\in\Comp(\Gamma(p_1',q_1')_1)}{e\neq 1}
                \#\calM_{d_e'\Gamma(p_1',q_1')_1}(\vec{f})\,\Gamma(p_1',q_1')_1\Bigr].
\end{split}\]
In this expression, the first line agrees with $(-1)^{\ell-1}$ times the part of (\ref{eq:Tr-Tr}) of $i(x_1)=i(y_1)=\ell-1$. The vanishing of the last two lines can be shown as follows: for each $p_1',q_1'\in P_*^{(1)}$ with $i(p_1')=i(q_1')+2$ and for $i\neq 1$, we have
\[\begin{split}
	&\Tr_{h,\ldots,g^{(i)},\ldots}\Bigl[\,\,\sum_{j=1}^3
		\asum{p_i\in P^{(i)}_j}{q_i\in P^{(i)}_{j-1}}
                \Bigl(\sum_{r_i\in P^{(i)}_{j-1}}(-1)^{j}\#\calM_{\partial^{(i)}(p_i,r_i)*\Gamma(p_1',q_1')_1(r_i,q_i)_i}(\vec{f})\\
        &\hspace{25mm}+\sum_{r_i\in P^{(i)}_j}(-1)^{j+1}\#\calM_{\Gamma(p_1',q_1')_1(p_i,r_i)_i*\partial^{(i)}(r_i,q_i)}(\vec{f})
		\Bigr)\,\Gamma(p_1',q_1')_1(p_i,q_i)_{i}\Bigr]\\
%%%%%%%%
	&=\Tr_{h,\ldots,\partial^{(i)}g^{(i)}+g^{(i)}\partial^{(i)},\ldots}\Bigl[\,\,\sum_{j=0}^3
		(-1)^{j+1}\kern-3mm\sum_{p_i',q_i'\in P^{(i)}_j}
		\#\calM_{\Gamma(p_1',q_1')_1(p_i',q_i')_i}(\vec{f})\,\Gamma(p_1',q_1')_1(p_i',q_i')_{i}\Bigr]\\
%%%%%%%%
	&=\Tr_{h,\ldots,\mathrm{id},\ldots}\Bigl[\,\,\sum_{j=0}^3
                (-1)^{j+1}\sum_{r_i'\in P^{(i)}_j}
                \#\calM_{\Gamma(p_1',q_1')_1(r_i',r_i')_i}(\vec{f})\,\Gamma(p_1',q_1')_1(r_i',r_i')_{i}\Bigr].
\end{split}\]
This cancels with the corresponding term in 
\[\Tr_{h,\ldots}\Bigl[
		\asum{e\in\Comp(\Gamma(p_1',q_1')_1)}{e\neq 1}
                \#\calM_{d_e'\Gamma(p_1',q_1')_1}(\vec{f})\,\Gamma(p_1',q_1')_1
\Bigr].\]
This completes the proof.
\end{proof}
%\clearpage

%%%%%%%%%%%%%%%%%%%%%%%%%%%%%%
%%%%%%%%%%%%%%%%%%%%%%%%%%%%%%
%%%%%%%%%%%%%%%%%%%%%%%%%%%%%%
\mysection{Independence of 4-cobordism and sections on it}{s:indep_4-cob}

%%%%%%%%%%%%%%%%%%%%%%%%%%%%%%%%%%%%%%%%%%%%
\subsection{Spin cobordism invariance of $Z_{2k,3k}^\mathrm{anomaly}$}\label{ss:spincob}

In the rest of this section we assume that $M$ is a $\Z$-homology 3-sphere. We say that two compact spin 4-manifolds $W$ and $W'$ with $\partial W=\partial W'=\bM$ are {\it relatively spin cobordant} if there is a compact spin 5-manifold $V$ with corners with $\partial V=(-W)\cup_{\partial}([0,1]\times \bM)\cup_\partial W'$ whose spin structure is an extension of those of $-W$ and $W'$. 

\begin{Prop}\label{prop:spin-cob_inv}
Let $W$ and $W'$ be two compact spin 4-manifolds with $\partial W=\partial W'=\bM$ and $\chi(W)=\chi(W')=1$ as in Lemma~\ref{lem:KM} (1). If $W$ and $W'$ are relatively spin cobordant, then the following assertions hold.
\begin{enumerate}
\item There exists a framing $\tau_M$ of $TM_0$ such that $\tau_M'$ can be extended to 4-framings of both $TW$ and $TW'$. Hence one can find sequences of GM sections $\vec{\gamma}_W\in\Gamma(T^vW)^{3k}$ and $\vec{\gamma}_{W'}\in\Gamma(T^vW')^{3k}$ with $\vec{\gamma}_W|_{M-U_\infty'}=\vec{\gamma}_{W'}|_{M-U_\infty'}=-\mathrm{grad}\,\vec{f}$. 
\item For any such extensions $\vec{\gamma}_W$ and $\vec{\gamma}_{W'}$, which are generic in the sense of Proposition~\ref{prop:anomaly_welldefined} (1), we have
\begin{equation}\label{eq:spin-cob_inv}
 Z_{2k,3k}^\mathrm{anomaly}(\vec{\gamma}_W)
	=Z_{2k,3k}^\mathrm{anomaly}(\vec{\gamma}_{W'}). 
\end{equation}
\end{enumerate}
\end{Prop}
\begin{proof}
(1) Put $X=(-W)\cup_{g}([0,1]\times \bM)\cup_{g'}W'$, where the gluing maps $g:-\partial W=-\bM\to \{1\}\times \bM$ and $g':\partial W'=\bM\to \{0\}\times \bM$ are the natural ones. By assumption, we have $[X]=0\in\Omega_4^\mathrm{spin}$. Take a 5-dimensional compact spin manifold $V$ with corners with $\partial V=X$ whose smooth structure near $[0,1]\times \bM$ is isomorphic to that of $[0,1]\times W$. Then $TV$ restricts on the boundary to a vector bundle that is isomorphic to $\ve^1\oplus TX$. 

By the isomorphism $\mathrm{sign}:\Omega_4^\mathrm{spin}\stackrel{\sim}{\to}16\Z$ and by Hirzebruch's signature theorem $\mathrm{sign}\,X=\frac{1}{3}\langle p_1(TX),[X]\rangle$ for $X$ closed, it follows that
\[ -p_1(TW;\tau_M')+p_1(TW';\tau_M')=\langle p_1(\ve^1\oplus TX),[X]\rangle=3\,\mathrm{sign}\,X=0 \]
for any choice of $\tau_M$. By choosing $\tau_M$ suitably, we may assume that $p_1(TW;\tau_M')=p_1(TW';\tau_M')=0$ by Lemma~\ref{lem:KM} (2). By Lemma~\ref{lem:KM} (2), such a 4-framing $\tau_M'$ extends to 4-framings on both $W$ and $W'$. 

(2) Since the 4-framings $\tau_W$ and $\tau_W'$ obtained in (1) above are extensions of $\tau_M'$, they can be trivially extended to a sub 4-framing $\tau_X$ of $\ve^1\oplus TX$ by the product structure of $[0,1]\times \bM$. The sub 3-framing of $\tau_X$ whose restriction to $\{1\}\times(M- U_\infty')$ agrees with $\tau_M$ spans a rank 3 subbundle $T^vX$ of $\ve^1\oplus TX$. Then there is a piecewise smooth GM sections $\vec{\gamma}_X$ of $T^vX$, which is a gluing of $\vec{\gamma}_W$, $\vec{\gamma}_{W'}$ and $\mathrm{pr}^{-1}\vec{\gamma}_W|_{M-U_\infty'}\in\Gamma(T^v([0,1]\times (M-U_\infty')))^{3k}$ together at the boundary. By definition of $Z_{2k,3k}^\mathrm{anomaly}$, 
\begin{equation}\label{eq:Z(rho)}
 Z_{2k,3k}^\mathrm{anomaly}(\vec{\gamma}_X)=
	-Z_{2k,3k}^\mathrm{anomaly}(\vec{\gamma}_W)
	+Z_{2k,3k}^\mathrm{anomaly}(\vec{\gamma}_{W'}). 
\end{equation}
Then Lemma~\ref{lem:Z(rho)=0} below completes the proof.
\end{proof}

\begin{Lem}\label{lem:Z(rho)=0}
Let $X$ and $\vec{\gamma}_X$ be as in the proof of Proposition~\ref{prop:spin-cob_inv} (2). Then we have $Z_{2k,3k}^\mathrm{anomaly}(\vec{\gamma}_X)=0$.
\end{Lem}

We use the following lemma in the proof of Lemma~\ref{lem:Z(rho)=0}.
\begin{Lem}\label{lem:Vpara}
Let $X$ be as in the proof of Proposition~\ref{prop:spin-cob_inv} (2) and $\tau_X$ be as above. Then $X\tcoprod X$ bounds a compact connected parallelizable 5-manifold $V$ on which the stabilization of the 4-framing $\tau_X\tcoprod \tau_X$ extends as a 5-framing.
\end{Lem}
\begin{proof}
Since $X$ is spin null-cobordant, there exists a compact connected spin 5-manifold $V$ with corners with $\partial V=X$. We first consider the obstruction to extending the stable framing $n\oplus\tau_X$ of $\ve^1\oplus TX$ to a 5-framing on $TV$, where $n$ is the unit vector field normal to the span of $\tau_X$ with respect to a metric of $V$. 

Since $V$ is spin and since $\pi_2(SO_5)=0$ and $\pi_3(SO_5)\cong\Z$, the first obstruction $\mathfrak{o}_1(V;n\oplus\tau_X)$ to the extension lies in the group $H^4(V,\partial V;\pi_3(SO_5))\cong H_1(V;\Z)$. We shall see that we may assume that this group is trivial after changing $V$ by surgery. It is easy to see that any class in $H_1(V;\Z)$ can be realized by an embedding $f:S^1\to \mathrm{Int}\,V$. Since $V$ is spin, the normal bundle $N_f$ of the image of $f$ is trivial. By a surgery along a framed embedding $(f,\tau_f)$, i.e., attaching of a 6-dimensional 2-handle along a tubular neighborhood of $\mathrm{Im}\,f$ through the trivialization, the homology class $[f]$ can be eliminated. Moreover, by replacing the 4-framing $\tau_f$ suitably, we may assume that the resulting 5-manifold of the surgery is spin since $\pi_1(SO_4)\to \pi_1(SO_5)$ and $\pi_1(SO_4)\to \pi_1 (SO_6)$ are isomorphisms. Namely, choose a 5-framing $\tau_2$ on an open neighborhood $U$ of the 2-skeleton of a CW structure on $V$. We may assume after an isotopy that the image of $f$ is included in $U$. Since $\pi_1(SO_5,SO_4)=0$, $\tau_2$ can be deformed to a 5-framing $\tau_2'$ whose restriction to $\mathrm{Im}\,f$ consists of tangent vectors of $f$ and a normal 4-framing of $\mathrm{Im}\,f$. The obstruction to extending a stabilization of $\tau_2'$ to a 6-framing on the 2-handle $D^2\times D^4$ lies in $H^2(D^2,\partial D^2;\pi_1(SO_6))\cong \Z_2$, which can be removed by a $\pi_1(SO_4)$-twist of the attaching map. Since $\pi_2(SO_6,SO_5)=0$, the 6-framing on the 2-handle can be modified so that the restriction to a 2-skeleton of the boundary is a stabilization of a 5-framing. Now the 2-skeleton of the result of the surgery is framed. Hence the result of the surgery is spin again.
 
Now we assume $H_1(V;\Z)=0$ by doing surgeries as above if necessary. Then the next obstruction $\mathfrak{o}_2(V;n\oplus\tau_X)$ for the extension lies in the group $H^5(V,\partial V;\pi_4(SO_5))\cong \Z_2$ since $\pi_4(SO_5)=\Z_2$. To eliminate $\mathfrak{o}_2(V;n\oplus\tau_X)$, we consider the connected sum $V'=V\# V$ taken between the interiors. Then one can check that the obstruction $\mathfrak{o}_2(V';n\oplus\tau_X\tcoprod n\oplus\tau_X)\in H^5(V',\partial V';\pi_4(SO_5))$ vanishes in any case. This completes the proof.
\end{proof}

\begin{proof}[Proof of Lemma~\ref{lem:Z(rho)=0}]
We prove $Z_{2k,3k}^\mathrm{anomaly}(\vec{\gamma}_X)=0$ by constructing cobordisms of moduli spaces. Let $V$ be a compact parallelizable 5-manifold with $\partial V=X\tcoprod X$ as in Lemma~\ref{lem:Vpara}. Roughly, we will construct 1-dimensional moduli spaces $\calM_\Gamma^\loc(\vec{\gamma})$ in a fiber bundle over $V$ for each 3-valent graph $\Gamma$ and we will see that
\[ 2Z_{2k,3k}^\mathrm{anomaly}(\vec{\gamma}_X)=\sum_{\Gamma\in\calG_{2k,3k}}\#\partial \bcalM_\Gamma^\loc(\vec{\gamma})\,[\Gamma]=0. \]
Since the replacement of $X$ with $X\tcoprod X$ and $V$ with $V\# V$ changes $Z_{2k,3k}^\mathrm{anomaly}(\vec{\gamma}_X)$ just by a multiple of $2$, it is enough for our purpose to assume that the obstruction $\mathfrak{o}_2(V;n\oplus \tau_X)$ vanishes in advance. Because of this we assume for simplicity that we have a framed 5-manifold $(V,\tau_V)$ that extends $(X,n\oplus \tau_X)$. 

We shall now define the moduli space $\calM_\Gamma^\loc$ extended over $V$. Let $\Gamma$ be a labeled graph in $\calG_{2k,3k}^0$. Since we assume that the 5-framing $\tau_V$ extends $n\oplus\tau_X$, we have a sub 3-framing of $\tau_V$ that is an extension of the 3-framing of $T^vX$ and it spans a rank 3 subbundle $T^vV$ of $TV$. Moreover by Lemma~\ref{lem:GM_section} there is a GM extension $\vec{\gamma}=(\gamma_1,\ldots,\gamma_{3k})\in\Gamma(T^vV)^{3k}$ of $\vec{\gamma}_X$. 
Since for each $j$, $\Sigma(\gamma_j)$ is a compact 2-submanifold of $V$, we may arrange that $\Sigma(\gamma_j)$'s are disjoint from each other by a general position argument. Then we consider the blow-up $q:\overline{V}\to V$, where
\[ \overline{V}=B\ell(V,\tcoprod_{\ell=1}^{3k}\Sigma(\gamma_\ell)). \]
We identify $\mathrm{Int}\,\overline{V}$ with $V- \tcoprod_\ell \Sigma(\gamma_\ell)$ by $q$. Consider the pullback bundle $q^*TV$ over $\overline{V}$ and we set $T^v\overline{V}=q^*T^vV$. Note that $T^v\overline{V}$ is not a subbundle of $T\overline{V}$. We identify the total space of the associated $\bConf_{2k}^\loc(\R^3)$-bundle to $T^v\overline{V}$ with $\overline{V}\times\bConf_{2k}^\loc(\R^3)$ via the trivialization $\tau_V$. The nowhere zero sections $\gamma_1,\ldots,\gamma_{3k}$ of $T^v(V-\tcoprod_\ell\Sigma(\gamma_\ell))$ extends smoothly to nowhere zero sections of $T^v\overline{V}$. We denote by $\overline{\Theta}_\ell(\gamma_\ell)$ the closure of $\Theta_\ell(\gamma_\ell)$ in $\overline{V}\times\bConf_{2k}^\loc(\R^3)$, which is a compact oriented submanifold with boundary. Then we may define the compact moduli space
\begin{equation}\label{eq:bMlocal}
 \bcalM_\Gamma^\loc(\vec{\gamma})=\bigcap_{\ell=1}^{3k}\overline{\Theta}_{\ell}(\gamma_\ell)\subset \overline{V}\times \bConf_{2k}^\loc(\R^3).
\end{equation}
\par\medskip
%%%%%%%%%%%%%%%%%%%
\noindent{\bf Claim~1.}\quad {\it After a $C^0$-small perturbation of $\vec{\gamma}$ in $\Gamma(T^vV)^{3k}$ without affecting the general positions for $\Sigma(\gamma_j)$'s, we may arrange that $\bcalM_\Gamma^\loc(\vec{\gamma})$ is a compact smooth 1-submanifold of $\overline{V}\times\bConf_{2k}^\loc(\R^3)$ and that the 1-manifold $\bcalM_\Gamma^\loc(\vec{\gamma})$ is transversal to $\partial\overline{V}\times\bConf_{2k}^\loc(\R^3)$. \par\medskip

\noindent{Proof.}} The restriction for the singularities of GM sections $\gamma_j$ given in \S\ref{ss:H-bundle} is used here. Let $\Gamma'$ be the graph obtained from $\Gamma$ by replacing $E(\Gamma)$ with $E(\Gamma)-\{\beta(j)\}$, let $V_j=V- \tcoprod_{\ell\neq j}\Sigma(\gamma_\ell)$ and $\overline{V}_j=B\ell(V,\tcoprod_{\ell\neq j}\Sigma(\gamma_\ell))$. Let $\pi':\overline{V}\times \bConf_{2k}^\loc(\R^3)\to \overline{V}$ and $\pi_j':\overline{V}_j\times\bConf_{2k}^\loc(\R^3)\to \overline{V}_j$ be the projections. Then as mentioned in the proof of Proposition~\ref{prop:anomaly_welldefined} (1), $\calM_{\Gamma'}^\loc(\vec{\gamma}\setminus\{\gamma_j\})$ is a submanifold of $V_j\times\bConf_{2k}^\loc(\R^3)$ of codimension $6k-2$, i.e., 3-dimensional, and we may define its compactification $\bcalM_{\Gamma'}^\loc(\vec{\gamma}\setminus\{\gamma_j\})$ as the closure of $\calM_{\Gamma'}^\loc(\vec{\gamma}\setminus\{\gamma_j\})$ in $\overline{V}_j\times\bConf_{2k}^\loc(\R^3)$. We denote by $\bcalM_{\Gamma'}^\loc(\vec{\gamma}\setminus\{\gamma_j\};\overline{V})$ the closure of $\calM_{\Gamma'}^\loc(\vec{\gamma}\setminus\{\gamma_j\})\cap\pi'^{-1}(V-\tcoprod_{\ell=1}^{3k}\Sigma(\gamma_\ell))$ in $\overline{V}\times \bConf_{2k}^\loc(\R^3)$. Then we have 
\[ \bcalM_\Gamma^\loc(\vec{\gamma})=\overline{\Theta}_j(\gamma_j)\cap \bcalM_{\Gamma'}^\loc(\vec{\gamma}\setminus\{\gamma_j\};\overline{V})\subset \overline{V}\times \bConf_{2k}^\loc(\R^3).\]
$\bcalM_\Gamma^\loc(\vec{\gamma})$ may have boundary points on $\pi'^{-1}q^{-1}(\Sigma(\gamma_j))$. Such boundary points can not be avoided since the 3-manifold $\bcalM_{\Gamma'}^\loc(\vec{\gamma}\setminus\{\gamma_j\})$ may intersect the codimension 3 submanifold $\pi_j'^{-1}(\Sigma(\gamma_j))$ of $\overline{V}_j\times\bConf_{2k}^\loc(\R^3)$. After a small perturbation of $\gamma_j$ in a small neighborhood of $\Sigma(\gamma_j)$, we may arrange that the intersection of the two submanifolds is transversal and that $\pi_j'(\bcalM_{\Gamma'}^\loc(\vec{\gamma}\setminus\{\gamma_j\}))$ and $\Sigma(\gamma_j)$ are transversal.

We shall give a local description of $\bcalM_{\Gamma'}^\loc(\vec{\gamma}\setminus\{\gamma_j\};\overline{V})$ near the transversal intersection. Take a point $x\in \bcalM_{\Gamma'}^\loc(\vec{\gamma}\setminus\{\gamma_j\})\pitchfork \pi_j'^{-1}(\Sigma(\gamma_j))$ and a small open neighborhood $U_x'$ of $x$ in $\overline{V}_j\times\bConf_{2k}^\loc(\R^3)$ so that $U_x'$ contains exactly one intersection point. Let $U_x=\pi_j'(U_x')$. After a suitable $C^0$-small perturbation of $\vec{\gamma}\setminus\{\gamma_j\}$ in a small neighborhood of $\Sigma(\gamma_j)$, we may arrange that
\begin{enumerate}
\item[(i)] $\pi_j'(x)$ is a Morse singularity of $\gamma_j$ and $U_x\cap \Sigma^2(\gamma_j)=\emptyset$,
\item[(ii)] $\pi_j'(\bcalM_{\Gamma'}^\loc(\vec{\gamma}\setminus\{\gamma_j\}))$ is tangent to $T^v\overline{V}_j$ at $\pi_j'(x)$. (This is possible since $\Sigma(\gamma_j)$ is transversal to both $T^v\overline{V}_j$ and $\pi_j'(\bcalM_{\Gamma'}^\loc(\vec{\gamma}\setminus\{\gamma_j\}))$.)
\end{enumerate}
\begin{figure}
\fig{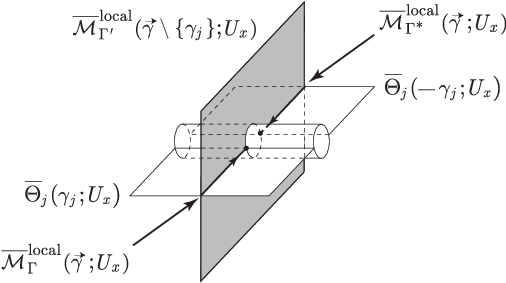}
\caption{The intersection in $\pi_j'^{-1}(B\ell(U_x,U_x\cap\Sigma(\gamma_i)))$}\label{fig:blow_up}
\end{figure}
We consider the blow-up $B\ell(U_x,U_x\cap\Sigma(\gamma_j))$ and let 
\[ \bcalM_{\Gamma'}^\loc(\vec{\gamma}\setminus\{\gamma_j\};U_x)=\overline{\calM_{\Gamma'}^\loc(\vec{\gamma}\setminus\{\gamma_j\})\cap \pi_j'^{-1}(U_x-\Sigma(\gamma_j))}\quad \mbox{(the closure)} \]
in $B\ell(U_x,U_x\cap\Sigma(\gamma_j))\times\bConf_{2k}^\loc(\R^3)$. Since $\bcalM_{\Gamma'}^\loc(\vec{\gamma}\setminus\{\gamma_j\})$ is transversal to $\pi_j'^{-1}(\Sigma(\gamma_j))\cap U_x'$, $\bcalM_{\Gamma'}^\loc(\vec{\gamma}\setminus\{\gamma_j\};U_x)$ is a submanifold of $B\ell(U_x,U_x\cap\Sigma(\gamma_j))\times\bConf_{2k}^\loc(\R^3)$ with boundary that meets $\partial B\ell(U_x,U_x\cap\Sigma(\gamma_j))\times\bConf_{2k}^\loc(\R^3)$ transversally. 

On the other hand, $\Theta_{j}(\gamma_j)\subset (V-\coprod_{\ell=1}^{3k}\Sigma(\gamma_\ell))\times\bConf_{2k}^\loc(\R^3)$ has the closure $\overline{\Theta}_{j}(\gamma_j;U_x)$ in $B\ell(U_x,U_x\cap \Sigma(\gamma_j))\times\bConf_{2k}^\loc(\R^3)$ that is a submanifold with boundary that meets $\partial B\ell(U_x,U_x\cap\Sigma(\gamma_j))\times\bConf_{2k}^\loc(\R^3)$ transversally since $U_x\cap\Sigma(\gamma_j)$ consists only of Morse singularities. By the assumption (ii), the intersection of $\overline{\Theta}_{j}(\gamma_j;U_x)$ and $\bcalM_{\Gamma'}^\loc(\vec{\gamma}\setminus\{\gamma_j\};U_x)$ is transversal even on the boundary and forms a 1-submanifold of $B\ell(U_x,U_x\cap\Sigma(\gamma_j))\times\bConf_{2k}^\loc(\R^3)$ with boundary. Let
\[ \bcalM_{\Gamma}^\loc(\vec{\gamma};U_x)=
	\overline{\Theta}_{j}(\gamma_j;U_x)\pitchfork\bcalM_{\Gamma'}^\loc(\vec{\gamma}\setminus\{\gamma_j\};U_x).\]
See Figure~\ref{fig:blow_up} for a schematic illustration. $\bcalM_\Gamma^\loc(\vec{\gamma};U_x)$ is a local model of $\bcalM_\Gamma^\loc(\vec{\gamma})$. Clearly $\bcalM_\Gamma^\loc(\vec{\gamma};U_x)$ is transversal to $\partial B\ell(U_x,U_x\cap \Sigma(\gamma_j))\times\bConf_{2k}^\loc(\R^3)$. By similar arguments for other intersection points $x$ and for other $j$, we may arrange that $\bcalM_\Gamma^\loc(\vec{\gamma})$ is transversal to the boundary.\qed
\par\medskip
%%%%%%%%%%%%%%%%%%
\noindent{\bf Claim~2.}\quad {\it If $\vec{\gamma}$ is as in Claim~1, then the boundary contribution of $\bcalM_\Gamma^\loc(\vec{\gamma})$ at the `inner' boundary $(\partial\overline{V}-\partial V)\times\bConf_{2k}^\loc(\R^3)$ is cancelled with that of some other graph $\Gamma^*$ by symmetry.\par\medskip
\noindent{Proof.}} By the assumption (ii) in the proof of Claim~1, the boundary of $\pi_j'(\bcalM_{\Gamma}^\loc(\vec{\gamma};U_x))$ lies in the fiber $S^2_x$ of the unit sphere bundle $S(T^vV)$ at $\pi_j'(x)$. Let $\Gamma^*$ denote the graph obtained from $\Gamma$ by reversing the orientation of the edge labeled $j$. Notice that there are individual terms for $\Gamma$ and $\Gamma^*$ in the formula of $Z_{2k,3k}^\mathrm{anomaly}(\vec{\gamma}_X)$ in Definition~\ref{def:Za}. Since $\bcalM_{\Gamma}^\loc(\vec{\gamma};U_x)\tcoprod \bcalM_{\Gamma^*}^\loc(\vec{\gamma};U_x)$ is transversal to $\partial B\ell(U_x,U_x\cap\Sigma(\gamma_j))\times \bConf_{2k}^\loc(\R^3)$ by Claim 1 and since on a neighborhood of $\Sigma(\gamma_j)$ there is a symmetry between the moduli spaces $\bcalM_\Gamma^\loc(\vec{\gamma};U_x)$ and $\bcalM_{\Gamma^*}^\loc(\vec{\gamma};U_x)$ by the assumption (i) and by the symmetry of the standard model around a Morse point, the intersection of $\pi_j'(\bcalM_{\Gamma}^\loc(\vec{\gamma};U_x)\tcoprod\bcalM_{\Gamma^*}^\loc(\vec{\gamma};U_x))$ with $\partial B\ell(U_x,U_x\cap\Sigma(\gamma_j))$ consists of two points in $S_x^2$ that are precisely in an antipodal position. Hence one may see that
\[ \begin{split}
&\#\partial\bcalM_{\Gamma}^\loc(\vec{\gamma};U_x)\,[\Gamma]+\#\partial\bcalM_{\Gamma^*}^\loc(\vec{\gamma};U_x)\,[\Gamma^*]\\
=&\Bigl(\#\partial\bcalM_{\Gamma}^\loc(\vec{\gamma};U_x)-\#\partial\bcalM_{\Gamma^*}^\loc(\vec{\gamma};U_x)\Bigr)[\Gamma]\\
=&\Bigl(\#\partial\bcalM_{\Gamma}^\loc(\vec{\gamma};U_x)-\#\partial\bcalM_{\Gamma}^\loc(\vec{\gamma};U_x)\Bigr)[\Gamma]=0.
\end{split} \]
Here, the second equality follows by the facts that the symmetry reverses the orientation of $\overline{\Theta}_j$, and that the inward normal vectors at $\partial\bcalM_{\Gamma}^\loc(\vec{\gamma};U_x)$ and $\partial\bcalM_{\Gamma^*}^\loc(\vec{\gamma};U_x)$ are opposite. See Figure~\ref{fig:blow_up} and \ref{fig:involution_ij}.\qed
\begin{figure}
\fig{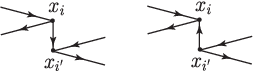}
\caption{}\label{fig:involution_ij}
\end{figure}
\par\bigskip
We continue the proof of Lemma~\ref{lem:Z(rho)=0}. Now by Claims~1 and 2, 
\[ 0=\sum_{\Gamma\in\calG_{2k,3k}}\#\partial\bcalM_\Gamma^\loc(\vec{\gamma})\,[\Gamma]=Z_{2k,3k}^\mathrm{anomaly}(\vec{\gamma}_X)-\sum_{\Gamma\in\calG_{2k,3k}}\#\bcalM_{d\Gamma}^\loc(\vec{\gamma})\,[\Gamma].
\]
The second term in the RHS vanishes by the IHX relation of $\calA_{2k,3k}$.
\end{proof}

\begin{Rem}
Proposition~\ref{prop:spin-cob_inv} shows that $Z_{2k,3k}^\mathrm{anomaly}(\vec{\gamma}_W)$ does not depend on the GM extension $\vec{\gamma}_W$ of $-\mathrm{grad}\,\vec{f}$. However, we fixed a diffeomorphism $\varphi_\infty:U'_\infty\to U_\infty$ and extension $\tau_M'$ of $n\oplus \tau_M$, so we must check that $Z_{2k,3k}^\mathrm{anomaly}(\vec{\gamma}_W)$ does not depend on these choices. It will be checked in Lemma~\ref{lem:inv_ordered}.
\end{Rem}

%%%%%%%%%%%%%%%%%%%%%%%%%%%%%%%%%%%%%
\subsection{Well-definedness of the correction term}\label{ss:correctionterm}

To prove Proposition~\ref{prop:anomaly_welldefined} (2), we consider general pairs of spin 4-manifolds $W$ and $W'$ with $\partial W=\partial W'=M$, $\chi(W)=\chi(W')=1$ which may not be relatively spin cobordant. We choose 3-framings $\sigma_M$ and $\tau_M$ on $TM_0$ so that 
\[ p_1(TW;\tau_M')=p_1(TW';\sigma_M')=0, \]
which are canonical up to homotopy. Then by Lemma~\ref{lem:KM}, $\tau_M'$ extends to a 4-framing of $W$ and $\sigma_M'$ extends to a 4-framing of $W'$. But $\tau_M$ may not be homotopic to $\sigma_M$, so we may not have a stable framing of $\ve^1\oplus TX$, $X=(-W)\cup_{g}([0,1]\times M)\cup_{g'}W'$, namely, $X$ may be just almost parallelizable. Although we do not have a stable framing of $\ve^1\oplus TX$, we have a rank 3 (possibly nontrivial) subbundle $T^vX$ of $\ve^1\oplus TX$ that agrees with $\mathrm{pr}_1^*TM_0$ on $[0,1]\times M_0$, which extends those spanned by $\sigma_M$ and $\tau_M$. By choosing a generic GM sections $\vec{\gamma}_X\in\Gamma(T^vX)^{3k}$ extending $-\mathrm{grad}\,\vec{f}$, one can define $Z_{2k,3k}^\mathrm{anomaly}(\vec{\gamma}_X)\in\calA_{2k,3k}$. 

More generally, one can also define $Z_{2k,3k}^\mathrm{anomaly}(\vec{\gamma}_X)$ for any almost parallelizable, closed, connected, spin 4-manifold\footnote{Note that any compact connected spin 4-manifold is almost parallelizable. Thus the assumption of almost parallelizability is unnecessary.} with $\chi(X)=2$. Namely, by a straightforward analogue of \cite[Theorem~2.2]{KM}, the restriction of a framing on $X-\mathrm{Int}\,([0,1]\times D^3)$ to $\partial ([0,1]\times D^3)$ can be deformed to a framing of the form $\mathrm{pr}_1^{-1}\tau_{D^3}\oplus\mathrm{pr}_2^{-1}\tau_{[0,1]}$ if and only if $\chi(X)=2$. 

Let $\Omega_4^\mathrm{spin}(2)$ denote the set of spin cobordism classes of closed, connected, spin 4-manifolds $X$ with $\chi(X)=2$. By the same argument as in the proof of Lemma~\ref{lem:Z(rho)=0}, one may see that the assignment $X\mapsto Z_{2k,3k}^\mathrm{anomaly}(\vec{\gamma}_X)$ for generic $\vec{\gamma}_X$ defines a well-defined map
\[ Z_{2k,3k}^\mathrm{anomaly}:\Omega_4^\mathrm{spin}(2)\to \calA_{2k,3k}. \]
The set $\Omega_4^\mathrm{spin}(2)$ has a group structure given by connected sum. More precisely, if $X$ is a closed, connected, spin 4-manifold with $\chi(X)=2$, then there is a framing on $X-[0,1]\times D^3$. If $X'$ is another closed, connected, spin 4-manifold with $\chi(X')=2$, then by forming the boundary connected sum $\overline{X-[0,1]\times D^3}\,\natural\, \overline{X'-[0,1]\times D^3}$ and capping by $[0,1]\times D^3$ along the boundary in a natural way, we will obtain an almost parallelizable, closed, connected, spin 4-manifold $X''$ with $\chi(X'')=2$ that is diffeomorphic to $X\#X'$. This defines an abelian group structure on $\Omega_4^\mathrm{spin}(2)$ on which the inverse of $X$ is given by $-X$. 

\begin{Lem}\label{lem:Z_hom}
The map $Z_{2k,3k}^\mathrm{anomaly}:\Omega_4^\mathrm{spin}(2)\to \calA_{2k,3k}$ is a group homomorphism. 
\end{Lem}
\begin{proof}
If $[X]=0\in\Omega_4^\mathrm{spin}(2)$, then we have $0=\mathrm{sign}\,X=\frac{1}{3}\langle p_1(\ve_X^1\oplus TX),[X]\rangle$ and $\chi(X)=2$, thus by Lemma~\ref{lem:KM}, the stabilization of the 4-framing on $\partial\overline{X-[0,1]\times D^3}$ induced from that of $X-[0,1]\times D^3$ extends over $X$. Namely, $X$ is stably parallelizable. Then by the same argument as in the proof of Lemma~\ref{lem:Z(rho)=0}, we have $Z_{2k,3k}^\mathrm{anomaly}(\vec{\gamma}_{X})=0$ for any generic GM sections $\vec{\gamma}_{X}\in(\Gamma(T^vX))^{3k}$. The additivity of $Z_{2k,3k}^\mathrm{anomaly}$ follows from the fact that $Z_{2k,3k}^\mathrm{anomaly}$ is invariant under spin cobordism as shown in Proposition~\ref{prop:spin-cob_inv}, and that $X\tcoprod X'$ and $X\# X'$ are spin cobordant. Hence $Z_{2k,3k}^\mathrm{anomaly}$ is a homomorphism.
\end{proof}

\begin{proof}[Proof of Proposition~\ref{prop:anomaly_welldefined} (2)]
By Lemma~\ref{lem:Z_hom}, $Z_{2k,3k}^\mathrm{anomaly}$ is a restriction of a group homomorphism $\Omega_4^\mathrm{spin}\to \calA_{2k,3k}$. So there exists a constant $\mu_k\in\calA_{2k,3k}$ such that
\[ Z_{2k,3k}^\mathrm{anomaly}(\vec{\gamma}_X)=\mu_k\,\mathrm{sign}\,X, \]
for $X=(-W)\cup_g([0,1]\times M)\cup_{g'}W'$. By (\ref{eq:Z(rho)}) and by the additivity of the signature, 
\[ \begin{split}
-Z_{2k,3k}^\mathrm{anomaly}(\vec{\gamma}_W)+Z_{2k,3k}^\mathrm{anomaly}(\vec{\gamma}_{W'})
	&=\mu_k\,\mathrm{sign}\,X=-\mu_k\,\mathrm{sign}\,W+\mu_k\,\mathrm{sign}\,W'.
\end{split} \]
This completes the proof.
\end{proof}

%\clearpage

%%%%%%%%%%%%%%%%%%%%%%%%%%%%%%
%%%%%%%%%%%%%%%%%%%%%%%%%%%%%%
%%%%%%%%%%%%%%%%%%%%%%%%%%%%%%
%%%%%%%%%%%%%%%%%%%%%%%%%%%%%%
\mysection{Moduli space of gradient flow graphs in 1-parameter family}{}

The next two sections contain preliminaries for the proof of Theorem~\ref{thm:Z_invariant}, which are 1-parameter analogues of the results in \S\ref{s:M2} to \S\ref{s:coori}. We consider generic 1-parameter families of smooth functions $f_s:M_0\to \R$ and metrics $\mu_s$ on $M_0$ parametrized by $s\in [0,1]$, and see what happens to the moduli spaces of flow graphs during the homotopy $\{(f_s,\mu_s)\}_{s\in [0,1]}$. We shall extend the definition of the moduli spaces $\calM_2(f)$ and $\calM_\Gamma(\vec{f})$ to those for 1-parameter families (\S\ref{ss:M_Gamma_1-para}) and give their compactifications to smooth manifolds with corners. 
%%%%%%%%%%%%%%%%%%%%%%%%%%%%%%
\subsection{Bifurcations in 1-parameter family of smooth functions and metrics}\label{ss:bifurcations}

Let $f,f':M_0\to \R$ be two Morse functions. Then there exists a smooth 1-parameter family $\{f_s:M_0\to \R\}_{s\in[0,1]}$ of functions on $M_0$ such that $f_0=f$ and $f_1=f'$ and $f_s$ is standard near $\infty_M$ with respect to a chart $\varphi_{\infty s}:U_{\infty s}'\to U_\infty$ ($\infty_M\in U_{\infty s}'$), where we say that a 1-parameter family $\{f_s\}_{s\in[0,1]}$ is smooth if the map $F:[0,1]\times M_0\to \R$, $F(s,x)=f_s(x)$ is smooth. It is known that $F$ can be chosen so that for all $s\in[0,1]$, $f_s$ does not have higher singularities. 

\begin{Lem}[\cite{Ce}]\label{lem:cerf}
Two Morse functions on a manifold can be connected by a smooth 1-parameter family of smooth functions with only Morse or birth-death ($A_2$) singularities.
\end{Lem}
The proof of the lemma can be found in \cite[\S{4.3}]{Lau}. 

In the following, we will often identify a smooth 1-parameter family $\{f_s\}_{s\in [a,b]}$ of functions on $M_0$ with the smooth map $F:[0,1]\times M_0\to \R$, $F(x,s)=f_s(x)$. Under this identification, we consider $f_s$ as both a map $M_0\to\R$ and a map $\{s\}\times M_0\to \R$. We will consider $\calD_p(f_s)$ etc as subsets of $M_0$ or $M_0\times\{s\}$, depending on the context.

Let $\{(f_s,\mu_s)\}_{s\in[0,1]}$ be a smooth 1-parameter family of smooth functions and metrics such that $(f_0,\mu_0)$ and $(f_1,\mu_1)$ are Morse--Smale. Here we say that the family $\{\mu_s\}_{s\in[0,1]}$ of metrics is smooth if it is the restriction of a smooth metric on $[0,1]\times M_0$ that is standard near $[0,1]\times\infty_M$. We will sometimes call $s\in[0,1]$ a {\it time} and we say that a time $s_0\in[0,1]$ is a {\it bifurcation} if $(f_{s_0},\mu_{s_0})$ is not Morse--Smale or not ordered. 

\begin{Lem}[\cite{HW}(p.~42), Lemma~2.11 of \cite{Hu}]\label{lem:generic_1-para}
After a perturbation of $\{(f_s,\mu_s)\}_{s\in[0,1]}$ fixing endpoints, we may arrange that there are finitely many bifurcation times in $[0,1]$ each of which is one of the following. 
\begin{enumerate}
\item Level exchange, i.e., a time where the order of the critical values changes.
\item Birth-death bifurcation, i.e., a time $s$ where $\Sigma(f_s)$ consists of Morse singularities and one birth-death singularity.
\item $i/i$-intersection (\cite{HW}), i.e., a time where a family of descending manifolds and a family of ascending manifolds of the same index $i$ intersect transversally in $[0,1]\times M$.
\item A time where the intersection of a descending manifold and an ascending manifold is not transversal.
\end{enumerate}
We may assume that no two different bifurcations overlap on a single time. (We will call such a 1-parameter family a \emph{generic 1-parameter family}.) 
\end{Lem}
At a bifurcation, the topologies of the moduli spaces $\calM'(f_s;p_s,q_s)$, $p_s,q_s\in\Sigma(f_s)$, may change. 

Lemma~\ref{lem:generic_1-para} can be proved as follows. By Lemma~\ref{lem:cerf} and by definition of bifurcations, it is enough to prove (3) and (4) of the lemma in the case where $f_s$ is ordered Morse for all $s\in [0,1]$ (see Lemma~\ref{lem:ordered} for the definition of ordered Morse function). Put $J=[0,1]$. It suffices to prove that for a pair of critical loci $p=\{p_s\}_{s\in J}$ and $q=\{q_s\}_{s\in J}$, the submanifolds $\widetilde{\calA_p}(f_{J})=\bigcup_{s\in J}\calA_{p_s}(f_s)$ and $\widetilde{\calD_q}(f_{J})=\bigcup_{s\in J}\calD_{q_s}(f_s)$ of $J\times M_0$ can be made transversal. They are indeed submanifolds of $J\times M_0$ for a similar reason as the descending and ascending manifolds are submanifolds of $M_0$. Namely, by the parametrized Morse lemma (\cite[Appendix]{Ig2}) one may see that they are submanifolds on a neighborhood of the critical locus and then extended by the gradient flow without changing its diffeomorphism type. By modifying the 1-parameter family $\{\mu_s\}_{s\in J}$ of metrics on $M_0$ suitably, one can show, by a similar argument as the proof of the genericity of the Morse--Smale condition (see e.g. \cite{Pe}), that $\widetilde{\calA_p}(f_{J})$'s and $\widetilde{\calD_q}(f_{J})$'s intersect mutually transversal in the trivial $M_0$-bundle over $J$ after a fiber preserving small perturbation of the metrics. Note that even if so, it may not be true that $\calA_{p_s}(f_s)$ and $\calD_{q_{s}}(f_s)$ are transversal for every $s$. If the transversality of $\calA_{p_s}(f_s)$ and $\calD_{q_{s}}(f_s)$ for $i(p_s)=i(q_s)$ fails, then $s$ is of type (3). For other indices, the intersection $\widetilde{\calA}_p(f_J)\cap \widetilde{\calD}_q(f_J)$ is a submanifold of $J\times M_0$. We may assume that the map $\mathrm{pr}:\widetilde{\calA}_p(f_J)\cap \widetilde{\calD}_q(f_J)\to J$ induced from the projection $J\times M_0\to J$ is Morse for every pair $(p,q)$ of distinct critical loci\footnote{If $\widetilde{\calA}_p(f_J)$ and $\widetilde{\calD}_q(f_J)$ are transversal, then that $\calA_{p_s}(f_s)$ and $\calD_{q_s}(f_s)$ are transversal is equivalent to that $s$ is a regular value of $\mathrm{pr}:\widetilde{\calA}_p(f_J)\cap \widetilde{\calD}_q(f_J)\to J$. This can be checked by applying the formula $\dim\,V+W=\dim\,V+\dim\,W-\dim\,V\cap W$ for vector spaces twice.}. There are finitely many\footnote{The finiteness is proved by using compactifications of $\widetilde{\calA}_p(f_J)$ and $\widetilde{\calD}_q(f_J)$ given later. Although we use Lemma~\ref{lem:generic_1-para} in the construction of the compactification, there is no problem in this because we do not use the finiteness of the bifurcations for the compactifications. } critical values of $\mathrm{pr}$, which are bifurcations of type (4).

We say that a 1-parameter family $(f_J,\mu_J)=\{(f_s,\mu_s)\}_{s\in J}$ of Morse pairs satisfies the {\it parametrized Morse--Smale condition} if for every pair $(p,q)$ of critical loci of $f_J$ the intersection of $\widetilde{\calA_p}(f_{J})$ and $\widetilde{\calD_q}(f_{J})$ is transversal. 

It is convenient to represent bifurcations in a 1-parameter family by the graph of critical values, equipped with the information of $i/i$-intersections. See Figure~\ref{fig:graphic} for an example. Such a diagram is called Cerf's {\it graphic} (\cite{Ce}). In a graphic, a level exchange corresponds to a crossing of two curves, an $i/i$-intersection between a pair of critical points is represented by a dotted arrow, and a birth-death bifurcation corresponds to beaks.

\begin{figure}
\fig{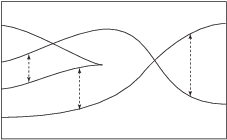}
\caption{Cerf's graphic}\label{fig:graphic}
\end{figure}
%%%%%%%%%%%%%%%%%%%%%%%%%%%%%%
\subsection{Moduli space $\calM_\Gamma$ in 1-parameter family and transversality}\label{ss:M_Gamma_1-para}

Let $\{(f_s,\mu_s)\}_{s\in [0,1]}$ be a generic 1-parameter family. Let $J=[s_0,s_1]$ be a closed interval in $[0,1]$ on which $\{(f_s,\mu_s)\}_{s\in J}$ does not have birth-death bifurcation. We consider a 1-parameter family $\vec{f}_s=(f_s,f_2,\ldots,f_m)$, $s\in J$, and extend the definition of the moduli space $\calM_\Gamma(\vec{f})$ to the family $\vec{f}_J=\{\vec{f}_s\}_{s\in J}$. 

The moduli space $\calM_\Gamma(\vec{f}_s)$ for a generic parameter $s\in J$ is defined similarly as $\calM_\Gamma(\vec{f})$ by replacing $f_1$ in the definition of $\calM_\Gamma(\vec{f})$ (\S\ref{ss:M_G}) with $f_s$, $\mu_1$ with $\mu_s$ and critical points with critical loci. For graphs $\Gamma$ with $\mathrm{dim}\,\calM_\Gamma<0$ with respect to the formula of Proposition~\ref{prop:M_mfd}, the moduli space $\calM_\Gamma(\vec{f}_s)$ is empty at a generic parameter $s$, but we will see that $\calM_\Gamma(\vec{f}_s)$ may be non-empty at finitely many non-generic parameters in $J$ if the formula of Proposition~\ref{prop:M_mfd} gives $\dim\calM_\Gamma(\vec{f}_s)=-1$. 

\begin{Prop}\label{prop:Ms_mfd}
Let $\{(f_s,\mu_s)\}_{s\in J}$ be as above and let $\vec{C}$ be the sequence\\ $(C^{(s_0)},C^{(2)},C^{(3)},\ldots,C^{(m)})$ of acyclic complexes, where $C^{(s_0)}$ is the Morse complex for $(f_{s_0},\mu_{s_0})$. Suppose that $\Gamma\in\calG_{n,m,\vec{\eta}}^0(\vec{C})$ has no bivalent vertex. For a generic choice of $\{(f_s,\mu_s)\}_{s\in J}$, the space $\calM_\Gamma(\vec{f}_J)=\bigcup_{s\in J}\calM_\Gamma(\vec{f}_s)$, $\vec{f}_J=\{(f_s,f_2,\ldots,f_m)\}_{s\in J}$, is a smooth submanifold of $J\times \Conf_n(M)$ of dimension $(n-m)d+\sum_{i=1}^m\eta_i+1$.
\end{Prop}

For simplicity, we only check the transversality on the moduli space $\calM_\Gamma(\vec{f}_J)$ for the special graph $\Gamma$ of (\ref{eq:ex_graph}) since other cases are similar. Suppose that $\vec{f}_{s_0}=(f_{s_0},f_2,\ldots,f_6)\in (C^r_{\varphi_\infty}(M_0))^6$ is generic in the sense of Proposition~\ref{prop:M_mfd}. We decompose $\Gamma$ into two parts:
\[ \Gamma'=\fig{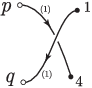},\mbox{ and } \Gamma''=\fig{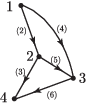}. \]
The moduli space $\calM_{\Gamma''}(\vec{f''})\subset \Conf_4(M)$, $\vec{f''}=(f_2,\ldots,f_6)$, is given by 
$\calM_{\Gamma''}(\vec{f''})=\mathrm{pr}_1(\Phi_{\vec{f''}}^{-1}(\Delta''))$, 
where $\Phi_{\vec{f''}}:\Conf_4(M)\times\R_+^5\to M_0^9$ is defined by 
\[ \begin{split}
	\Phi_{\vec{f''}}&(x_1,\ldots,x_4,t_2,\ldots,t_6)\\
		&=(x_1,x_2,\Phi_{f_2}^{t_2}(x_1),x_3,\Phi_{f_4}^{t_4}(x_1),\Phi_{f_5}^{t_5}(x_2),x_4,\Phi_{f_3}^{t_3}(x_2),\Phi_{f_6}^{t_6}(x_3)),
		\end{split}\]
and $\Delta''=\{(x_1,x_2,x_2,x_3,x_3,x_3,x_4,x_4,x_4)\,;\,x_1,x_2,x_3,x_4\in M_0\}$. The genericity of $\vec{f}_{s_0}$ implies that $\calM_{\Gamma''}(\vec{f''})$ is a submanifold of $\Conf_4(M)$ of dimension $(4d+5)+4d-9d=5-d$. On the other hand, the moduli space $\calM_{\Gamma'}(f_J)=\bigcup_{s\in J}\calM_{\Gamma'}(f_s)$ is given by the $(d+\eta_1+1)$-dimensional manifold
$\calN_{pq}(f_J)=\bigcup_{s\in J}\calN_{pq}(f_s)\subset J\times \Conf_2(M)$. Then we have $\calM_{\Gamma}(\vec{f}_J)=\widetilde{\pi}_{14}^{-1}(\calM_{\Gamma'}(f_J))\cap (J\times\calM_{\Gamma''}(\vec{f''}))$, where $\widetilde{\pi}_{14}:J\times\Conf_4(M) \to J\times\Conf_2(M)$ is the projection $(s,x_1,x_2,x_3,x_4)\mapsto (s,x_1,x_4)$. By the transversality theorem, we may assume after a small perturbation of the family $\{(f_s,\mu_s)\}_{s\in J}$ that the intersection is transversal, and hence $\calM_\Gamma(\vec{f}_J)=\bigcup_{s\in J}\calM_{\Gamma}(\vec{f}_s)$ is a submanifold of dimension $(d+\eta_1+1+2d)+(5-d+1)-(4d+1)=-2d+(\eta_1+5)+1$. If the first edge of $\Gamma$ were a compact edge, then $\calM_{\Gamma'}(f_J)$ would be replaced with $\calM_2(f_J)=\bigcup_{s\in J}\calM_2(f_s)$.

%%%%%%%%%%%%%%%%%%%%%%%%%%%%%%
\subsection{Compactification of the moduli space $\calM_2$ of trajectories in 1-parameter family of Morse pairs}\label{ss:compact_M2}

We construct compactifications of the spaces 
\[ 
  \calM_2(f_J)=\bigcup_{s\in J}\calM_2(f_s),\quad \calN_{pq}(f_J)=\bigcup_{s\in J} \calN_{pq}(f_s)\quad\subset J\times M_0^2,
 \]  
where $J\subset[0,1]$ is a compact interval that does not have birth-death bifurcations for the family $(f_J,\mu_J)=\{(f_s,\mu_s)\}_{s\in J}$, i.e., $f_J$ is a 1-parameter family of Morse functions. The goal of this subsection is to prove the following proposition.

\begin{Prop}\label{prop:bM_2(f_J)}
Let $(f_J,\mu_J)=\{(f_s,\mu_s)\}_{s\in J}$ be a generic 1-parameter family of Morse pairs that satisfies the parametrized Morse--Smale condition. There is a natural compactification $\bcalM_2(f_J)$ of $\calM_2(f_J)=\bigcup_{s\in J}\calM_2(f_s)$ such that the complement of $\bar{b}\,{}^{-1}(\widehat{\Delta}_M)$ in $\bcalM_2(f_J)$, where $\bar{b}:\bcalM_2(f_J)\to M\times M$ is the smooth extension of the evaluation map $\calM_2(f_J)\to M_0\times M_0$, is a smooth manifold with corners whose codimension $k$ stratum for $k\geq 1$ consists of families of $k$ times broken trajectories and $\partial_{k-1}\bcalM_2(f_{\partial J})$, the codimension $k-1$ stratum of $\bcalM_2(f_{\partial J})-\bar{b}^{-1}(\widehat{\Delta}_M)$ in $\partial J\times M$.
\end{Prop}

Let $\bar{b}_J:\bcalM_2(f_J)\to (J\times M)\times (J\times M)$ be the evaluation map with time, which is defined for a possibly broken trajectory $\gamma$ in $\{s\}\times M$ with $\bar{b}(\gamma)=(x,y)$ to be $\bar{b}_J(\gamma)=(s,x)\times(s,y)$. For a critical locus $p=\{p_s\}_{s\in J}$ of $f_J$, we write
\[ 
  \calC\widetilde{\calD}_p(f_J)=\bar{b}_J^{-1}(p\times (J\times M)),\quad  \calC\widetilde{\calA}_p(f_J)=\bar{b}_J^{-1}((J\times M)\times p).
\]
Let $\bar{b}_{\calA}:\calC\wcalA_q(f_J)\to J\times M$ (resp. $\bar{b}_{\calD}:\calC\wcalD_p(f_J)\to J\times M$) be the map that assigns the initial endpoint (resp. terminal endpoint) of a possibly broken flow line. Let $\Delta_J\times M^2$ be the subset of $(J\times M)^2$ consisting of points of the form $(s,x)\times (s,y)$ and let
\[ \bcalN_{pq}(f_J)=(\bar{b}_\calA\times \bar{b}_\calD)^{-1}(\Delta_J\times M^2)\subset \calC\wcalA_q(f_J)\times \calC\wcalD_p(f_J). \]
Let $\overline{bb}:\calC\wcalA_q(f_J)\times \calC\wcalD_p(f_J)\to M\times M$ be the composition of $\bar{b}_\calA\times \bar{b}_\calD$ and the projection $(J\times M)^2\to M\times M$. For subsets $\widetilde{A},\widetilde{B}\subset J\times M$, let $\widetilde{A}\times_J\widetilde{B}=(\widetilde{A}\times\widetilde{B})\cap (\Delta_J\times M^2)$ and let 
\[ \calM_2(f_J;\widetilde{A},\widetilde{B})=\calM_2(f_J)\cap (\widetilde{A}\times_J\widetilde{B}). \]
The following corollaries are immediate consequences (analogue of Proposition~\ref{prop:compactification_D}) of Proposition~\ref{prop:bM_2(f_J)}.
\begin{Cor}\label{cor:compactification_tD}
Let $(f_J,\mu_J)=\{(f_s,\mu_s)\}_{s\in J}$ be a generic 1-parameter family of Morse pairs as in Proposition~\ref{prop:bM_2(f_J)} and let $p$ be a critical locus of $f_J$. Then $\calC\widetilde{\calD}_p(f_J)$ (resp. $\calC\widetilde{\calA}_p(f_J)$) is a compactification of $\widetilde{\calD}_p(f_J)$ (resp. $\widetilde{\calA}_p(f_J)$) such that the complement of $\bar{b}\,{}^{-1}(\widehat{\Delta}_M)$ in $\calC\widetilde{\calD}_p(f_J)$ (resp. $\calC\widetilde{\calA}_p(f_J)$) is a smooth manifold with corners whose codimension $k$ stratum for $k\geq 1$ consists of families of $k$ times broken trajectories and $\partial_{k-1}\bcalD_p(f_{\partial J})$ (resp. $\partial_{k-1}\bcalA_p(f_{\partial J})$).
\end{Cor}

\begin{Cor}\label{cor:compactification_tN}
Let $(f_J,\mu_J)=\{(f_s,\mu_s)\}_{s\in J}$ be a generic 1-parameter family of Morse pairs as in Proposition~\ref{prop:bM_2(f_J)} and let $p$, $q$ be critical loci of $f_J$. Then $\bcalN_{pq}(f_J)$ is a compactification of $\calN_{pq}(f_J)$ such that the complement of $\overline{bb}\,{}^{-1}(\widehat{\Delta}_M)$ in $\bcalN_{pq}(f_J)$ is a smooth manifold with corners whose codimension $k$ stratum for $k\geq 1$ consists of families of $k$ times broken trajectories and $\partial_{k-1}\bcalN_{pq}(f_{\partial J})$.
\end{Cor}

%%%%%
\subsubsection{The moduli space $\bcalM_2(f_J)$ around a level exchange bifurcation}\label{ss:mod_level_ex}

We first construct the compactification of $\calM_2(f_J)$ around level exchange bifurcations and then extend to whole of $J$. In the construction of $\bcalM_2(f)$ (in \S\ref{ss:mod_short}), we assumed that the critical values of $f$ are all distinct (Lemma~\ref{lem:ordered}). However, this is not the case for a 1-parameter family due to level exchange bifurcations. We consider the space of `semi-short' trajectories that are close to an exchanging pair of critical loci to construct a compact space of trajectories around the level exchange bifurcation. Let $u\in J$ be a level exchange bifurcation and choose a small compact interval $J_u=[u-\ve,u+\ve]$ so that there are no other bifurcations over $J_u$. We shall prove the following lemma.

\begin{Lem}\label{lem:bM_2_level_exchg}
Let $J_u$ be as above and suppose that $\mu_{J_u}$ is such that $\mu_s$ is Euclidean near $\Sigma(f_s)$ for each $s\in J_u$. If $\ve$ is sufficiently small, then there is a natural compactification $\bcalM_2(f_{J_u})$ of $\calM_2(f_{J_u})$ such that $\bcalM_2(f_{J_u})-\bar{b}^{-1}(\widehat{\Delta}_M)$ is a smooth manifold with corners whose codimension $k$ stratum for $k\geq 1$ consists of families of $k$ times broken trajectories and $\partial_{k-1}\bcalM_2(f_{u\pm \ve})$. 
\end{Lem}

Let $p=\{p_s\}_{s\in J_u},q=\{q_s\}_{s\in J_u}$ be the pair of critical loci of $f_{J_u}=\{f_s\}_{s\in J_u}$ that are in a level exchange position. Then there exist smooth functions $\gamma_a,\gamma_b:J_u\to \R$ such that 
\begin{enumerate}
\item $\gamma_a(s)<\gamma_b(s)$ for all $s\in J_u$,
\item $f_s(p_s), f_s(q_s)\in (\gamma_a(s),\gamma_b(s))$ for all $s\in J_u$,
\item for each $s\in J_u$, there are no critical points of $f_s$ in $f_s^{-1}[\gamma_a(s),\gamma_b(s)]$ except $p_s$ and $q_s$. 
\end{enumerate}
We put $\widetilde{L}_a=\bigcup_{s\in J_u}f^{-1}_s(\gamma_a(s))$, $\widetilde{L}_b=\bigcup_{s\in J_u}f^{-1}_s(\gamma_b(s))$, $W_{pq}(s)=f^{-1}_s[\gamma_a(s),\gamma_b(s)]$, $\widetilde{W}_{pq}=\bigcup_{s\in J_u}W_{pq}(s)$, all considered as subsets of $J_u\times M_0$. We define
\[ \begin{split}
  &\bcalM_2(f_{J_u};\widetilde{L}_b,\widetilde{L}_a)=\mathrm{Closure}(\calM_2(f_{J_u};\widetilde{L}_b,\widetilde{L}_a))\subset \widetilde{L}_b\times_{J_u}\widetilde{L}_a,\\
  &\bcalM_2(f_{J_u};\widetilde{L}_b,\widetilde{W}_{pq})=\mathrm{Closure}(\calM_2(f_{J_u};\widetilde{L}_b,\widetilde{W}_{pq}))\subset \widetilde{L}_b\times_{J_u}\widetilde{W}_{pq},\\
  &\bcalM_2(f_{J_u};\widetilde{W}_{pq},\widetilde{L}_a)=\mathrm{Closure}(\calM_2(f_{J_u};\widetilde{W}_{pq},\widetilde{L}_a))\subset \widetilde{W}_{pq}\times_{J_u} \widetilde{L}_a,\\
  &\bcalM_2(f_{J_u};\widetilde{W}_{pq},\widetilde{W}_{pq})=\mathrm{Closure}(\calM_2(f_{J_u};\widetilde{W}_{pq},\widetilde{W}_{pq}))\subset \widetilde{W}_{pq}\times_{J_u} \widetilde{W}_{pq}.
\end{split}\]
\begin{Lem} Suppose that $\widetilde{\calD}_p(f_{J_u})\cap \widetilde{\calA}_q(f_{J_u})=\emptyset$. Then the following hold.
\begin{enumerate}
\item[(i)] $\bcalM_2(f_{J_u};\widetilde{L}_b,\widetilde{L}_a)-\{\Delta_J\times \infty_M^2\}$ is a submanifold of $\widetilde{L}_b\times_{J_u}\widetilde{L}_a$ with boundary whose boundary consists of once broken flow sequences and the moduli spaces at endpoints.
\item[(ii)] $\bcalM_2(f_{J_u};\widetilde{L}_b,\widetilde{W}_{pq})-\{\Delta_J\times \infty_M^2\}$ is a submanifold of $\widetilde{L}_b\times_{J_u}\widetilde{W}_{pq}$ with corners whose boundary consists of once broken flow sequence and of points in $\widetilde{L}_b\times_{J_u}\partial\widetilde{W}_{pq}$ and the moduli spaces at endpoints.
\item[(iii)] $\bcalM_2(f_{J_u};\widetilde{W}_{pq},\widetilde{L}_a)-\{\Delta_J\times \infty_M^2\}$ is a submanifold of $\widetilde{W}_{pq}\times_{J_u}\widetilde{L}_a$ with corners whose boundary consists of once broken flow sequences and of points in $\partial\widetilde{W}_{pq}\times_{J_u}\widetilde{L}_a$ and the moduli spaces at endpoints.
\item[(iv)] $\bcalM_2(f_{J_u};\widetilde{W}_{pq},\widetilde{W}_{pq})-\widehat{\Delta}_{\widetilde{W}_{pq}}$ is a submanifold of $\widetilde{W}_{pq}\times_{J_u}\widetilde{W}_{pq}$ with corners whose boundary consists of once broken flow sequences and of points in $\partial(\widetilde{W}_{pq}\times_{J_u}\widetilde{W}_{pq})\cup \widehat{\Delta}_{\widetilde{W}_{pq}}$ and the moduli spaces at endpoints.
\end{enumerate}
\end{Lem}
\begin{proof}
Let $\widetilde{K}_p=\bigcup_{s\in J_u}(\calD_{p_s}(f_s)\cup \calA_{p_s}(f_s))\cap W_{pq}(s)$ and $\widetilde{K}_q=\bigcup_{s\in J_u}(\calD_{q_s}(f_s)\cup \calA_{q_s}(f_s))\cap W_{pq}(s)$. 
\begin{figure}%
\fig{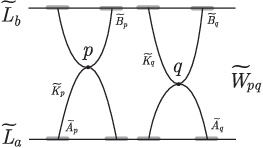}%
\caption{}\label{fig:KpKq}%
\end{figure}%
See Figure~\ref{fig:KpKq}. Take small compact neighborhoods $\widetilde{B}_p$ and $\widetilde{B}_q$ of $\widetilde{K}_p\cap \widetilde{L}_b$ and $\widetilde{K}_q\cap \widetilde{L}_b$ respectively in $\widetilde{L}_b$. Let $\widetilde{A}_p\subset \widetilde{L}_a$ be the union of $\widetilde{K}_p\cap\widetilde{L}_a$ and the subset of $\widetilde{L}_a$ consisting of points $(s,\ell)$ such that $\ell=\Phi_{f_s}^{t}(x)$ for a point $x\in\widetilde{B}_p\cap (\{s\}\times M_0)$ and for some $t>0$. In other words, $\widetilde{A}_p$ be the union of $\widetilde{K}_p\cap \widetilde{L}_a$ and the image of the negative gradient flow from $\widetilde{B}_p\cap \widetilde{L}_b$. $\widetilde{A}_q\subset \widetilde{L}_a$ is defined similarly with respect to $\widetilde{K}_q\cap\widetilde{L}_a$. Let $\widetilde{C}_p$ be the subset of $\widetilde{W}_{pq}$ consisting of points $(s,x)$ such that either $(s,x)\in \widetilde{K}_p$ or such that the integral curve $\gamma_x$ of $\mathrm{grad}_{\mu_s}f_s$ in $\{s\}\times M_0$ with $\gamma_x(0)=(s,x)$ intersects $\widetilde{B}_p$. $\widetilde{C}_q\subset \widetilde{W}_{pq}$ is defined similarly with respect to $\widetilde{K}_q$ and $\widetilde{B}_q$.

Since $\widetilde{K}_p\cap\widetilde{K}_q=\emptyset$, we may assume that any trajectory starting from $\widetilde{B}_p$ (resp. $\widetilde{B}_q$) are disjoint from trajectories starting from the complement of $\widetilde{B}_p$ (resp. $\widetilde{B}_q$). Thus we have
\begin{equation}\label{eq:M2(B,L)} 
	\calM_2(f_{J_u};\widetilde{B}_p,\widetilde{L}_a)=\calM_2(f_{J_u};\widetilde{B}_p,\widetilde{A}_p),\quad
	\calM_2(f_{J_u};\widetilde{B}_q,\widetilde{L}_a)=\calM_2(f_{J_u};\widetilde{B}_q,\widetilde{A}_q),
\end{equation}
and in particular, the two moduli spaces are disjoint in $\calM_2(f_{J_u};\widetilde{L}_b,\widetilde{L}_a)$. Since each of $\widetilde{C}_p$ and $\widetilde{C}_q$ has only one critical locus, the compactifications
$\bcalM_2(f_{J_u};\widetilde{B}_p,\widetilde{A}_p)$, $\bcalM_2(f_{J_u};\widetilde{B}_q,\widetilde{A}_q)$
can be defined in a similar way as Lemma~\ref{lem:dM(W,W)} (using parametrized Morse lemma \cite[Appendix]{Ig2}, assuming $\mu_s$ is Euclidean with respect to the local coordinate). They are smooth manifolds with boundary and are closures of $\calM_2(f_{J_u};\widetilde{B}_p,\widetilde{A}_p)$ and $\calM_2(f_{J_u};\widetilde{B}_q,\widetilde{A}_q)$ in $\widetilde{L}_b\times_{J_u}\widetilde{L}_a$. In accordance with (\ref{eq:M2(B,L)}), we define
\[ \bcalM_2(f_{J_u};\widetilde{B}_p,\widetilde{L}_a)=\bcalM_2(f_{J_u};\widetilde{B}_p,\widetilde{A}_p),\quad
	\bcalM_2(f_{J_u};\widetilde{B}_q,\widetilde{L}_a)=\bcalM_2(f_{J_u};\widetilde{B}_q,\widetilde{A}_q).
\]
We construct an extension of $\bcalM_2(f_{J_u};\widetilde{B}_p,\widetilde{L}_a)\tcoprod\bcalM_2(f_{J_u};\widetilde{B}_q,\widetilde{L}_a)$ to $\bcalM_2(f_{J_u};\widetilde{L}_b,\widetilde{L}_a)$ as follows. Let $\widetilde{X}\subset \widetilde{L}_b$ be the closure of the complement of $\widetilde{B}_p\cup\widetilde{B}_q$. Since there is no critical loci except $p$ and $q$ in $\widetilde{W}_{pq}$, the negative gradient flow carries $\widetilde{X}$ diffeomorphically onto a compact subset $\widetilde{Y}$ of $\widetilde{L}_a$, where $\widetilde{Y}$ is the closure of the complement of $\widetilde{A}_p\cup\widetilde{A}_q$. Hence $\calM_2(f_{J_u};\widetilde{X},\widetilde{L}_a)=\calM_2(f_{J_u};\widetilde{X},\widetilde{Y})\approx\widetilde{X}$, which is compact. The union
$\bcalM_2(f_{J_u};\widetilde{B}_p,\widetilde{L}_a)\cup
	\calM_2(f_{J_u};\widetilde{X},\widetilde{L}_a)\cup
	\bcalM_2(f_{J_u};\widetilde{B}_q,\widetilde{L}_a)$
is a smooth manifold with boundary and is the closure of $\calM_2(f_{J_u};\widetilde{L}_b,\widetilde{L}_a)$ in $\widetilde{L}_b\times_{J_u}\widetilde{L}_a$, namely, agrees with $\bcalM_2(f_{J_u};\widetilde{L}_b,\widetilde{L}_a)$.

For the compactifications $\bcalM_2(f_{J_u};\widetilde{L}_b,\widetilde{W}_{pq})$, $\bcalM_2(f_{J_u};\widetilde{W}_{pq},\widetilde{L}_a)$, $\bcalM_2(f_{J_u};\widetilde{W}_{pq},\widetilde{W}_{pq})$ etc. we consider $\bcalM_2(f_{J_u};\widetilde{L}_b,\widetilde{C}_p)$, $\bcalM_2(f_{J_u};\widetilde{C}_p,\widetilde{L}_a)$, $\bcalM_2(f_{J_u};\widetilde{C}_p,\widetilde{C}_p)$ etc. by a similar way as the unparametrized case and extend them as previous paragraph. 
\end{proof}

\begin{proof}[Proof of Lemma~\ref{lem:bM_2_level_exchg}]
We may assume that all the critical loci except $p$ and $q$ are ordered over the interval $J_u$ and according to Lemma~\ref{lem:generic_1-para}, we may assume that $(f_s,\mu_s)$ is Morse--Smale for all $s\in J_s$ if $\ve$ is sufficiently small. Thus fiber-product construction similar to Lemma~\ref{lem:hyp_p} can be applied and we will finally get a compactification $\bcalM_2(f_{J_u})$ of $\calM_2(f_{J_u})$. Then straightforward analogues of Lemma~\ref{lem:glue_M2}, \ref{lem:hyp1} and \ref{lem:hyp_p} show that $\bcalM_2(f_{J_u})-\bar{b}^{-1}(\widehat{\Delta}_M)$ is a smooth manifold with corners.
\end{proof}

By the same construction at all the level exchange points $u_1,u_2,\ldots,u_r\in J_u$, we will obtain a compactification $\bcalM_2$ on $\coprod_{j=1}^r J_{u_j}$.

\begin{Rem}
We assumed in Lemma~\ref{lem:bM_2_level_exchg} that $\mu_s$ is Euclidean near critical loci with respect to the local coordinate of parametrized Morse lemma. However, this assumption is not essential because if $\mu_s$ is not Euclidean near critical loci, then the flow lines near a critical locus are the images of flow lines in $J_u\times \R^3$ for the standard quadratic form with respect to the Euclidean metric of $\R^3$ under a fiber-preserving diffeomorphism defined on a neighborhood of $J_u\times\{0\}$. This remark will be taken into account to make sure that the compactification $\overline{\calM}_2(f_J)$ in Proposition~\ref{prop:bM_2(f_J)} is consistent with that at a birth-death bifurcation. 
\end{Rem}

%%%%%
\subsubsection{The moduli space $\bcalM_2(f_J)$ on ordered 1-parameter family of Morse pairs}
Next, we extend the compactifications of moduli spaces on $\coprod_{j=1}^r J_{u_j}$, given in \S\ref{ss:mod_level_ex}, over the whole of $J$. We assume $u_1<u_2<\cdots<u_r$. Let $I_j\subset J$, $j=0,1,2,\ldots,r$ be a sequence of mutually disjoint compact intervals such that 
\begin{enumerate}
\item $\bigcup_{j=0}^r I_j \cup \bigcup_{j=1}^r J_{u_j}=J$,
\item $\mathrm{Int}\,I_j\cap \mathrm{Int}\,J_{u_j}\neq \emptyset$ if $j>0$, and $\mathrm{Int}\, I_j\cap \mathrm{Int}\,J_{u_{j+1}}\neq\emptyset$ if $j<r$,
\item $(\coprod_{j=0}^r I_j)\cap \{u_1,\ldots,u_r\} =\emptyset$,
\item $I_j\subset (u_j,u_{j+1})$ if $1\leq j<r$.
\end{enumerate}
See Figure~\ref{fig:graphic_2}. We shall construct a compactification $\bcalM_2(f_{I_j})$ of $\calM_2(f_{I_j})$, which connects $\bcalM_2(f_{J_{u_j}})$ and $\bcalM_2(f_{J_{u_{j+1}}})$. 

\begin{Lem}\label{lem:bM_2_ordered}
Let $I_j$ be as above. Then there is a natural compactification $\bcalM_2(f_{I_j})$ of $\calM_2(f_{I_j})$ such that $\bcalM_2(f_{I_j})-\bar{b}^{-1}(\widehat{\Delta}_M)$ is a smooth manifold with corners whose codimension $k$ stratum for $k\geq 1$ consists of families of $k$ times broken trajectories and $\partial_{k-1}\bcalM_2(f_{\partial I_j})$. 
\end{Lem}
\begin{proof}
For each $j$, the critical values are consistently ordered over $I_j$, so we can separate critical loci by families of level surfaces. The compactification of the moduli space of trajectories that lie in a piece between level surfaces can be done as before, by means of the parametrized Morse lemma (e.g., \cite[Appendix]{Ig2}) and by the same argument as \S\ref{ss:mod_short}. 

Recall that in Lemma~\ref{lem:glue_M2}, the Morse--Smale condition is required. However, the Morse--Smale condition may not be satisfied for all $s\in I_j$. For example, it fails at an $i/i$-intersection bifurcation, as we have seen at Lemma~\ref{lem:generic_1-para}. Instead, we require the parametrized Morse--Smale condition and this suffices for the moduli space to be a smooth submanifold of a fiber bundle over $I_j$ (with fiber $\Conf_2(M)$), though the moduli space may not be a subbundle. Using the parametrized Morse--Smale condition in the fiber-product constructions, we may get a compactification $\overline{\calM}_2(f_{I_j})$ as desired. 
\end{proof}

\begin{proof}[Proof of Proposition~\ref{prop:bM_2(f_J)}]
\begin{figure}
\fig{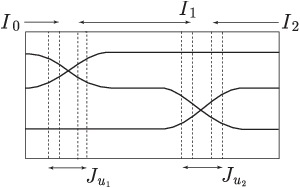}
\caption{}\label{fig:graphic_2}
\end{figure}
It remains to check that the compactifications obtained on $J_{u_j}$ and $I_j$ in Lemma~\ref{lem:bM_2_level_exchg} and \ref{lem:bM_2_ordered} respectively can be glued smoothly on the overlapping intervals $I_j\cap J_{u_j}$ and $I_j\cap J_{u_{j+1}}$. Let $\widetilde{L}_1,\ldots,\widetilde{L}_{N-2}$ be the loci of level surfaces for $f_{J_{u_{j}}}$ that are used to define $\bcalM_2(f_{u_j})$ and let $\widetilde{L}_1',\ldots,\widetilde{L}_{N-1}'$ be the loci of level surfaces for $f_{I_j}$ that are used to define $\bcalM_2(f_{I_j})$. We may assume without loss of generality that $\widetilde{L}_i$ and $\widetilde{L}_j'$ are disjoint for any $i,j$ (see Remark~\ref{rem:M2_nomfd}(3)). Let $\bcalM_2^\cap(f_{J_{u_j}\cap I_j})$ be the compactification of $\calM_2(f_{J_{u_j}\cap I_j})$ defined by using the loci of level surfaces $\widetilde{L}_1,\ldots,\widetilde{L}_{N-2},\widetilde{L}_1',\ldots,\widetilde{L}_{N-1}'$. Then there are natural embeddings
\[ \bcalM_2^\cap(f_{J_{u_j}\cap I_j})\to \bcalM_2(f_{J_{u_j}}),\quad
 \bcalM_2^\cap(f_{J_{u_j}\cap I_j})\to \bcalM_2(f_{I_j}) \]
which gives a strata preserving gluing map between $\bcalM_2(f_{J_{u_j}})$ and $\bcalM_2(f_{I_j})$. We consider $\bcalM_2^\cap(f_{J_{u_j}\cap I_j})$ as a subspace of both $\bcalM_2(f_{J_{u_j}})$ and $\bcalM_2(f_{I_j})$. Let
\[ \bcalM_2(f_{J_{u_j}\cup I_j})=\bcalM_2(f_{J_{u_j}})\cup_{\bcalM_2^\cap(f_{J_{u_j}\cap I_j})}\bcalM_2(f_{I_j}). \]
For other overlapping intervals, we also glue compactifications similarly.
\end{proof}

%%%%%%%%%%%%%%%%%%%%%%%%%%%%%%
\subsection{Gluing of a separated trajectory at birth-death bifurcation}\label{ss:gluing_separate}

Let $s_0\in [0,1]$ be a birth-death bifurcation in a generic 1-parameter family $(f_I,\mu_I)=\{(f_s,\mu_s)\}_{s\in [0,1]}$. Let $p_+$ and $p_-$ be the critical loci of $f_I$ that are involved in the birth or death bifurcation $s_0$, such that $i(p_+)=i(p_-)+1$. The space $\calN_{p_+p_-}(f_I)\subset [0,1]\times \Conf_2(M)$ can be considered as the moduli space of `separated' trajectories. In this subsection we shall see that $\calN_{p_+p_-}(f_I)$ and $\calM_2(f_I)$ are smoothly glued together at the time $s=s_0$. Here, we shall only study a death point since a birth point is symmetric. 

Let $s_0\in [0,1]$ be a death parameter in a generic 1-parameter family and let $J_{s_0}\subset [0,1]$ be a small open interval including $s_0$. Let $v\in M_0$ be the death point at $s_0$. By the normal form lemma for an unfolding of a birth-death singularity (e.g., \cite[Appendix]{Ig2}, \cite{Ce}), there is a local coordinate on a neighborhood $M_v$ of $v$ in $J_{s_0}\times M$ on which $f_s$ agrees with
\[ h_u(x)=c(u)+\frac{x_1^3}{3}+ux_1-\frac{x_2^2}{2}-\cdots-\frac{x_i^2}{2}
	+\frac{x_{i+1}^2}{2}+\cdots+\frac{x_d^2}{2},\quad u\in\R, \]
where $u$ is a reparametrization of $s$ and $c(u)$ is a smooth function of $u$, and one can choose a metric on $J_{s_0}\times M_0$ whose restriction on $M_v$ agrees with the restriction of the standard metric on $\R\times\R^d$. The negative gradient of $h_u$ with respect to the standard metric is
\[ -\mathrm{grad}\,h_u=(-x_1^2-u,x_2,\ldots,x_i,-x_{i+1},\ldots,-x_d). \]
On $u>0$, there are no critical points of $h_u$. At $u=0$, there is only one critical point of $h_u$ at the origin, and on $u<0$, there are exactly two critical points $p_\pm=(\pm\sqrt{|u|},0,\ldots,0)$ of $h_u$. From now on we shall describe how a pair of trajectories going from/to critical points of $h_u$ on $u<0$ are glued together into a single trajectory on $u>0$. It gives a gluing of a moduli space of a separated edge and that of a compact edge.

%%%%%
\subsubsection{Gradient trajectories of $h_u$ in $u>0$}

Here, we assume for simplicity that $c(u)=0$ for $u\in \R$, which does not affect the gradients. The integral curve $\gamma:\R\to \R^d$, $\gamma(t)=(\gamma_1(t),\ldots,\gamma_d(t))$ of $-\mathrm{grad}\,h_u$ is determined by the differential equations:
\begin{equation}\label{eq:ode}
\begin{split}
\dot{\gamma}_1(t)&=-\gamma_1(t)^2-u,\ \dot{\gamma}_2(t)=\gamma_2(t),\ \ldots,\ \dot{\gamma}_i(t)=\gamma_i(t),\\
	\dot{\gamma}_{i+1}(t)&=-\gamma_{i+1}(t),\ \ldots,\ \dot{\gamma}_d(t)=-\gamma_d(t),
	\end{split}
\end{equation}
for each given initial point $(\gamma_1(0),\ldots,\gamma_d(0))$. In $u>0$, the solution of (\ref{eq:ode}) is given explicitly by
\begin{equation}\label{eq:solution}
 \begin{split}
\gamma_1(t)&=\frac{\sqrt{u}\,\gamma_1(0)-u\tan{\sqrt{u}\,t}}{\sqrt{u}+\gamma_1(0)\tan{\sqrt{u}\,t}},\ \gamma_2(t)=\gamma_2(0)e^t,\ \ldots,\ \gamma_i(t)=\gamma_i(0)e^t,\\
	\gamma_{i+1}(t)&=\gamma_{i+1}(0)e^{-t},\ \ldots,\ \gamma_d(t)=\gamma_d(0)e^{-t}. 
\end{split}
\end{equation}

For a small number $\varepsilon>0$, let $L_\varepsilon$ and $L_{-\varepsilon}$ be the subsets of $\R^d$ given by
\[ \begin{split}
	L_\varepsilon&=\{(\varepsilon,x_2,\ldots,x_d)\in\R^d\,;\,x_2,\ldots,x_d\in\R\},\\
	L_{-\varepsilon}&=\{(-\varepsilon,x_2,\ldots,x_d)\in\R^d\,;\, x_2,\ldots,x_d\in\R\}.
	\end{split}
\]
These are approximations of level surfaces at the levels $\pm \ve$ in a neighborhood of the origin. Since $-\mathrm{grad}_{(0,x_2,\ldots,x_d)}\,h_u$ $=$ $(-u,x_2,\ldots,x_i,-x_{i+1},\ldots,-x_d)$, one may see that any trajectory of $h_u$ in $u>0$ and in $M_v$ intersects both $L_\varepsilon$ and $L_{-\varepsilon}$. Conversely, for any point $a$ of $L_\varepsilon\cap M_v$ (resp. $L_{-\varepsilon}\cap M_v$), there exists a unique (shift equivalence class of) gradient trajectory of $h_u$ which intersects $L_\varepsilon$ (resp. $L_{-\varepsilon}$) at $a$. So there is a one-to-one correspondence between a point on $L_\varepsilon$ or $L_{-\varepsilon}$ and a gradient trajectory of $h_u$ that is close to the origin. We identify a gradient trajectory with the pair of its intersection points with $L_{-\ve}\tcoprod L_\ve$.

Now suppose that an integral curve $\gamma(t)$ of $-\mathrm{grad}\,h_u$ starts at a point of $L_\ve$. We shall describe the point $\mathrm{Im}\,\gamma \cap L_{-\varepsilon}$. If $\gamma(t_{-\varepsilon})\in L_{-\varepsilon}$ at $t_{-\varepsilon}>0$, then by (\ref{eq:ode}),
\[ t_{-\varepsilon}=-\int_{\ve}^{-\varepsilon}\frac{dx}{x^2+u}
	=\frac{2}{\sqrt{u}}\Tan^{-1}\frac{\ve}{\sqrt{u}}
\]
for $0<u<\varepsilon^2$. We put $\tau_\varepsilon(u)=\frac{2}{\sqrt{u}}\Tan^{-1}\frac{\ve}{\sqrt{u}}$. Then $\tau_\varepsilon(u)$ has the following expansion (convergent on $0<u<\ve^2$):
\begin{equation}\label{eq:tau}
 \tau_\varepsilon(u)=\frac{\pi}{\sqrt{u}}-2\sum_{k=0}^\infty \frac{(-1)^k}{(2k+1)\varepsilon^{2k+1}}u^k. 
\end{equation}
Indeed, by the identity $\frac{1}{\tan{\alpha}}=-\tan(\alpha+\frac{\pi}{2})$, we have 
\[ \Tan^{-1}\frac{\ve}{\sqrt{u}}+\frac{\pi}{2}=-\Tan^{-1}\frac{\sqrt{u}}{\ve}+\pi=\pi-\sum_{k=0}^\infty \frac{(-1)^k}{2k+1}\left(\frac{\sqrt{u}}{\ve}\right)^{2k+1}. \]
The point $\gamma(t_{-\varepsilon})$ can be expressed by using $\tau_\varepsilon(u)$ as follows.
\[ \gamma(t_{-\varepsilon})
	=\bigl(
	-\varepsilon,
	\gamma_2(0)e^{\tau_\varepsilon(u)},\ldots,\gamma_i(0)e^{\tau_\varepsilon(u)},
	\gamma_{i+1}(0)e^{-\tau_\varepsilon(u)},\ldots,\gamma_d(0)e^{-\tau_\varepsilon(u)}	
	\bigl).
\]
If we put $\varepsilon_2=\gamma_2(0)e^{\tau_\varepsilon(u)},\ldots,\varepsilon_i=\gamma_i(0)e^{\tau_\varepsilon(u)},\varepsilon_{i+1}=\gamma_{i+1}(0),\ldots,\varepsilon_d=\gamma_d(0)$, then the integral curve starting at the point
\begin{equation}\label{eq:start}
 \bigl(\varepsilon,\varepsilon_2 e^{-\tau_\varepsilon(u)},\ldots,\varepsilon_i e^{-\tau_\varepsilon(u)},\varepsilon_{i+1},\ldots,\varepsilon_d\bigr)\in L_\varepsilon 
\end{equation}
intersects $L_{-\varepsilon}$ at the point
\begin{equation}\label{eq:goal}
 \bigl(-\varepsilon,\varepsilon_2,\ldots,\varepsilon_i,
	\varepsilon_{i+1}e^{-\tau_\varepsilon(u)},\ldots,\varepsilon_d e^{-\tau_\varepsilon(u)}\bigr)\in L_{-\varepsilon}. 
\end{equation}
This observation motivates the gluing formula below.

%%%%%
\subsubsection{Gradient trajectories of $h_u$ going from/to critical points in $u\leq 0$}

In $u<0$, the ascending and descending manifolds of $h_u$ are described as follows. 
\[\begin{split}
	\calA_{p_+}(h_u)&=\{x_2=\cdots=x_i=0,x_1\geq -\sqrt{|u|}\},\\
        \calD_{p_+}(h_u)&=\{x_{i+1}=\cdots=x_d=0,x_1=\sqrt{|u|}\},\\
	\calA_{p_-}(h_u)&=\{x_2=\cdots=x_i=0,x_1=-\sqrt{|u|}\},\\
	\calD_{p_-}(h_u)&=\{x_{i+1}=\cdots=x_d=0,x_1\leq\sqrt{|u|}\}.
\end{split}
\]
See Figure~\ref{fig:p_pm}. Hence
\begin{equation}\label{eq:LALD}
 \begin{split}
L_\varepsilon\cap \calA_{p_+}(h_u)&=
	\{(\varepsilon,0,\ldots,0,\varepsilon_{i+1},\ldots,\varepsilon_d)\in\R^d\,;\,\varepsilon_{i+1},\ldots,\varepsilon_d\in\R\},\\
L_{-\varepsilon}\cap \calD_{p_-}(h_u)&=
	\{(-\varepsilon,\varepsilon_2,\ldots,\varepsilon_i,0,\ldots,0)\in\R^d\,;\,\varepsilon_2,\ldots,\varepsilon_i\in\R\}.
\end{split}
\end{equation}
\begin{figure}%
\fig{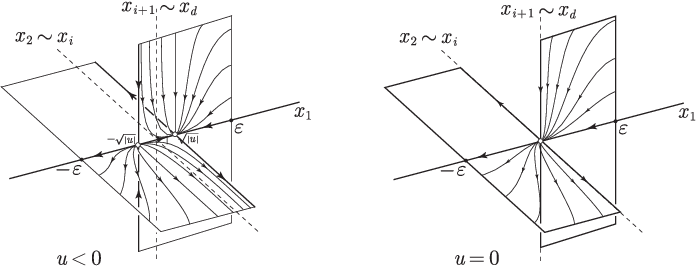}%
\caption{}\label{fig:p_pm}%
\end{figure}%
One may check that this also holds for $u=0$, in which case $p_+=p_-=v$.

%%%%%
\subsubsection{Gluing formula at $u=0$}

We define an injective map
\[ \varphi:(-\varepsilon^2,\varepsilon^2)\times\R^{d-1}\to (-\varepsilon^2,\varepsilon^2)
	\times \R^{d-1}\times \R^{d-1}\quad \mbox{by}\]
\[ \varphi(u;\ve_2,\ldots,\ve_d)=\left\{\begin{array}{ll}
	u\times\bigl(\ve_2 e^{-\tau_\ve(u)},\ldots,\ve_i e^{-\tau_\ve(u)},\ve_{i+1},\ldots,\ve_d\bigr)&\\
	\hspace{5mm}
	\times\bigl(\ve_2,\ldots,\ve_i,\ve_{i+1}e^{-\tau_\ve(u)},\ldots,\ve_d e^{-\tau_\ve(u)}\bigr),&
		\mbox{if $u>0$}\\
	u\times\bigl(0,\ldots,0,\ve_{i+1},\ldots,\ve_d\bigr)\times\bigl(\ve_2,\ldots,\ve_i,0,\ldots,0\bigr),&	\mbox{if $u\leq 0$}\\	
	\end{array}\right. \]
For any $u$ we may identify the space of gradient trajectories of $h_u$ or pairs of gradient trajectories of $h_u$ intersecting both $L_\ve$ and $L_{-\ve}$, with a subspace of $L_\ve\times L_{-\ve}$ through $\varphi$. By (\ref{eq:LALD}), the non-positive part $\varphi((-\ve^2,0]\times \R^{d-1})$ is the space of pairs $(\gamma_{p_+}(t),\gamma_{p_-}(t))$ of integral curves of $-\mathrm{grad}\,h_u$ satisfying the conditions
\begin{equation}\label{eq:g_pm}
	\lim_{t\to\infty}\gamma_{p_+}(t)=p_+,\ \gamma_{p_+}(0)\in L_{\ve},\quad \lim_{t\to -\infty}\gamma_{p_-}(t)=p_-,\ \gamma_{p_-}(0)\in L_{-\ve}.
\end{equation}
On the other hand, by (\ref{eq:start}) and (\ref{eq:goal}), the positive part $\varphi((0,\ve^2)\times\R^{d-1})$ is the space of negative gradient trajectories of $h_u$, $u>0$, near the origin. In other words,
\[ \begin{split}
  \varphi((-\ve^2,0]\times \R^{d-1})&=\calN_{p_-p_+}(\{h_u\}_{u\in(-\ve^2,0]})\cap (-\ve^2,0]\times(L_{\ve}\times L_{-\ve}),\\
\varphi((0,\ve^2)\times\R^{d-1})&=\calM_2(\{h_u\}_{u\in(0,\ve^2)})\cap (0,\ve^2)\times(L_{\ve}\times L_{-\ve}).
\end{split}\]

The following proposition gives a gluing of moduli spaces of short trajectories. 
\begin{Prop}\label{prop:gluing}
The map $\varphi$ is smooth and is an embedding. Hence $\mathrm{Im}\,\varphi$ is a smooth submanifold of $(-\ve^2,\ve^2)\times(L_\ve\times L_{-\ve})$ without boundary.
\end{Prop}
\begin{proof}
Let $\sigma_\ve:(-\ve^2,\ve^2)\to \R$ be the function defined by
\[ \sigma_\ve(u)=\left\{\begin{array}{ll}
	e^{-\tau_\ve(u)} & \mbox{if $u>0$}\\
	0 & \mbox{if $u\leq 0$}
	\end{array}\right. \]
The map $\varphi$ can be rewritten as
\[ \begin{split}
\varphi(u;\ve_2,\ldots,\ve_d)
	&=u\times (\ve_2\sigma_\ve(u),\ldots,\ve_i\sigma_\ve(u),\ve_{i+1},\ldots,\ve_d)\\
	&\hspace{15mm}
	\times(\ve_2,\ldots,\ve_i,\ve_{i+1}\sigma_\ve(u),\ldots,\ve_d\sigma_\ve(u)).
\end{split}
\]
We will see in the two lemmas below that $\sigma_\ve$ is $C^\infty$ differentiable. Hence $\varphi$ is $C^\infty$ differentiable. That the Jacobian matrix has full rank is obvious from the definition of $\varphi$ on $u\leq 0$. Hence $\varphi$ is an embedding.
\end{proof}

\begin{Lem}For any integer $n\geq 1$, there exist power series $P_n(u),Q_n(u)\in \R[[u]][\sqrt{u}]$, which are well-defined as $C^\infty$ differentiable functions on $(0,\ve^2)$ such that $P_n(u)=\bigl(\sqrt{u}(2u^2+2\ve^2 u)\bigr)^n$, $\lim_{u\to 0}Q_n(u)=0$ and
\[ \frac{d^n}{du^n}e^{-\tau_\ve(u)}
	=\frac{Q_n(u)+\pi^n\ve^{2n}}{P_n(u)}e^{-\tau_\ve(u)}.\]
\end{Lem}
\begin{proof}
We prove the lemma by induction on $n$. The case $n=0$ is obvious. Suppose the assertion holds true for $n$. Then on $(0,\ve^2)$, $\displaystyle\frac{d^{n+1}}{du^{n+1}}e^{-\tau_\ve(u)}$ equals 
\begin{equation}\label{eq:dn}
 \begin{split}
&\frac{d}{du}\left(\frac{Q_n(u)+\pi^n\ve^{2n}}{P_n(u)}e^{-\tau_\ve(u)}\right)=-\frac{e^{-\tau_\ve(u)}}{\sqrt{u}(2u^2+2\ve^2 u)P_n(u)}\\
	&\hspace{7mm}\times\left[
		\sqrt{u}\Bigl\{\Bigl((2u+2\ve^2)\frac{uP_n'(u)}{P_n(u)}-2\ve\Bigr)\pi^n\ve^{2n}
			+(-2u-2\ve^2)uQ_n'(u)\right.\\
	&\hspace{45mm}
		+(2u+2\ve^2)Q_n(u)\frac{uP'_n(u)}{P_n(u)}-2\ve Q_n(u)\Bigr\}\\
	&\hspace{18mm}
		+(-2u-2\ve^2)(Q_n(u)+\pi^n\ve^{2n}\left.)\Tan^{-1}\frac{\ve}{\sqrt{u}}\right].
\end{split}
\end{equation}
By (\ref{eq:tau}), one may see that 
\[ \begin{split}
	&\Tan^{-1}\frac{\ve}{\sqrt{u}}\in\R[[u]][\sqrt{u}],\quad
		\frac{uP_n'(u)}{P_n(u)}\in\R[[u]][\sqrt{u}],\\
	&\lim_{u\to 0}\Tan^{-1}\frac{\ve}{\sqrt{u}}=\frac{\pi}{2},\quad
		\lim_{u\to 0}uQ'_n(u)=0,\quad \lim_{u\to 0}\frac{uP'_n(u)}{P_n(u)}=\frac{3n}{2}.
\end{split}
\]
Indeed, putting $Q_n(u)=b_{\frac{1}{2}}\sqrt{u}+b_1u+b_{\frac{3}{2}}u\sqrt{u}+b_2u^2+\cdots$, we have $uQ_n'(u)=\frac{b_{\frac{1}{2}}}{2}\sqrt{u}+b_1u+\frac{3}{2}b_{\frac{3}{2}}u\sqrt{u}+2b_2u^2+\cdots$ and 
\[ \begin{split}
\frac{uP_n'(u)}{P_n(u)}&=\frac{n\sqrt{u}(5u+3\ve^2)(\sqrt{u}(2u^2+2\ve^2u))^{n-1}}{(\sqrt{u}(2u^2+2\ve^2u))^n}=\frac{n(5u+3\ve^2)}{2(u+\ve^2)}. 
\end{split}\]
This implies that the right hand side of (\ref{eq:dn}) is of the form
\[ \frac{Q_{n+1}(u)+\pi^{n+1}\ve^{2n+2}}{\sqrt{u}(2u^2+2\ve^2 u)P_n(u)}e^{-\tau_\ve(u)} \]
for a $C^\infty$ function $Q_{n+1}(u)\in\R[[u]][\sqrt{u}]$ with $\displaystyle\lim_{u\to 0}Q_{n+1}(u)=0$ that is well-defined on $(0,\ve^2)$. The proof completes if we put $P_{n+1}(u)=\sqrt{u}(2u^2+2\ve^2 u)P_n(u)$.
\end{proof}

\begin{Lem}\label{lem:dn=0}
For all $n\geq 0$, $\displaystyle\lim_{u\to 0}\frac{d^n}{du^n}e^{-\tau_\ve(u)}=0$.
\end{Lem}
\begin{proof}
Since $\lim_{u\to 0}(Q_n(u)+\pi^n\ve^{2n})=\pi^n\ve^{2n}$, it suffices to show that 
\[ \lim_{u\to 0}\frac{e^{-\tau_\ve(u)}}{P_n(u)}=\lim_{u\to 0}\frac{\exp(-\frac{\pi}{\sqrt{u}})\exp(2\sum_{k=0}^\infty \frac{(-1)^k}{(2k+1)\ve^{2k+1}}u^{k})}{(\sqrt{u}(2u^2+2\ve^2u))^n}=0. \]
Since $\lim_{u\to 0}\exp(2\sum_{k=0}^\infty \frac{(-1)^k}{(2k+1)\ve^{2k+1}}u^{k})=\frac{2}{\ve}$, the result follows by 
\[ \lim_{u\to 0}\frac{\exp(-\frac{\pi}{\sqrt{u}})}{(\sqrt{u}(2u^2+2\ve^2u))^n}=0. \]
\end{proof}

%%%%%%%%%%%%%%%%%%%%%%%%%%%%%%%%%%%%%%%%%%%%%%%%%%%%%%%%%%%%%%%%
\subsection{Compactification of $\calM_\Gamma(\vec{f}_I)$ in generic 1-parameter family}

%%%%%
\subsubsection{Compactification of the moduli space $\calM_\Gamma(\vec{f}_I)$ in 1-parameter family of Morse pairs}

By using the compactification $\bcalM_2(f_J)$ and $\bcalN_{pq}(f_J)$ given in \S\ref{ss:compact_M2}, one can also define the compactification $\bcalM_\Gamma(f_J)$ of $\calM_\Gamma(f_J)$ in a similar way as \S\ref{ss:comp_fuk}. We have the following proposition.
\begin{Prop}\label{prop:comp_M_1-para}
Suppose $d=3$, $\Gamma\in\calG_{2k,3k}^0(\vec{C})$ and that $\Gamma$ does not have a bivalent vertex. After a small perturbation of the family $(f_J,\mu_J)=\{(f_s,\mu_s)\}_{s\in J}$ of Morse pairs fixing the endpoints, we may arrange that $\bcalM_\Gamma(\vec{f}_J)$ is a compact smooth 1-manifold with boundary. The boundary consists of flow graphs with a once broken trajectory or with a subgraph collapsed to a point.
\end{Prop}

The proof of Proposition~\ref{prop:comp_M_1-para} is analogous to Proposition~\ref{prop:M_G-1-mfd} (proof in \S\ref{ss:comp_fuk}). Namely, we construct a singular compactification $\bcalM_\Gamma^\times(\vec{f}_J)$ in $\prod_{j=1}^{3k}\widetilde{Q}_j$, where $\widetilde{Q}_j$ is either $\bcalM_2(\vec{f}_J)$ or $\bcalN_{pq}(\vec{f}_J)$. Then a sequence of blowing-ups along the diagonals yields $\bcalM_\Gamma(\vec{f}_J)$.

%%%%%
\subsubsection{Gluing of $\bcalM_\Gamma(\vec{f}_I)$ at birth-death point}

Let $s_0\in[0,1]$ be a death parameter in a generic 1-parameter family $\{(f_s,\mu_s)\}_{s\in[0,1]}$. For sufficiently small number $\ve'>0$, let $(p_1,q_1)$ be the pair of critical points of $f_{s_0-\ve'}$, such that $i(p_1)=i(q_1)+1$ and such that they are eliminated on $s>s_0$ after passing through the death point $v$. Then we have the following proposition.

\begin{Prop}\label{prop:glue_bd}
Suppose $d=3$ and that $s_0$ is as above. Let $\Gamma(p_1,q_1)_1\in\calG_{2k,3k}(\vec{C}^{(s_0-\ve')})$ be a graph with no bivalent vertices and let $\Gamma(\emptyset,\emptyset)_1$ be the graph obtained from $\Gamma(p_1,q_1)_1$ by replacing the edge $\beta(1)$ with a compact edge. If $\ve'$ is sufficiently small, then the embedding $\varphi$ of Proposition~\ref{prop:gluing} induces a smooth compact 1-dimensional cobordism between 
\[ \bcalM_{\Gamma(\emptyset,\emptyset)_1}(\vec{f}_{s_0+\ve'})\mbox{ and }\bcalM_{\Gamma(p_1,q_1)_1}(\vec{f}_{s_0-\ve'})\tcoprod \bcalM_{\Gamma(\emptyset,\emptyset)_1}(\vec{f}_{s_0-\ve'}).\]
\end{Prop}
\begin{proof}
If $d=3$, then by Proposition~\ref{prop:M_mfd}, $\dim\bcalM_\Gamma(\vec{f}_s)=0$ for $\Gamma\in\calG_{2k,3k}^0(\vec{C}^{(s)})$. If $\ve'$ is sufficiently small, then there exists $\ve>0$ such that the pair of half trajectories that converge to $p_1$ and $q_1$ intersects $M_v\cap L_{-\ve}$ and $M_v\cap L_{\ve}$ respectively, since the broken trajectory at the limit $s=s_0$ satisfies this property. Thus we may use $\mathrm{Im}\,\varphi$ of Proposition~\ref{prop:gluing} to construct the desired cobordism by a fiber-product construction similar to Lemma~\ref{lem:glue_M2}, \ref{lem:hyp1} and \ref{lem:hyp_p}.
\end{proof}
%\clearpage

%%%%%%%%%%%%%%%%%%%%%%%%%%%%%%%%%%%%%%%%%%%%5
%%%%%%%%%%%%%%%%%%%%%%%%%%%%%%%%%%%%%%%%%%%%
\mysection{(Co)orientation of the moduli spaces in 1-parameter family}{s:coori_1-para}

%%%%%%%%%%%%%%%%%%%%%%%%%%%%%%
\subsection{Convention for (co)orientations in 1-parameter family}

Let $J=[s_0,s_1]$ and let $(f_J,\mu_J)$ be a 1-paremeter family of Morse pairs. In this section, we assume without loss of generality that $f_s=f_{s_0}$ for all $s\in [s_0,s_0+\ve)$ ($\ve>0$ small) and $f_s=f_{s_1}$ for all $s\in (s_1-\ve,s_1]$. We orient $J\times M$ and $J\times \bConf_{2k}(M)$ by
\[ o(J\times M)_{(s,x)}=ds\wedge o(M)_x,\quad o(J\times M^{2k})_{(s,\vec{x})}=ds\wedge o(M^{2k})_{\vec{x}}. \]
We define the coorientations $o^*_{J\times M}(\wcalD_p(f_J))$ and $o^*_{J\times M}(\wcalA_p(f_J))$ so that their restrictions to $\{s_0\}\times M$ are equivalent to $o^*_{\{s_0\}\times M}(\calD_p(f_{s_0}))$ and $o^*_{\{s_0\}\times M}(\calA_p(f_{s_0}))$ respectively. Similarly, we define the coorientations $o^*_{J\times M^2}(\calM_2(f_J))$ and $o^*_{J\times M^2}(\calN_{pq}(f_J))$ so that their restrictions to $\{s_0\}\times M^2$ are equivalent to $o^*_{\{s_0\}\times M^2}(\calM_2(f_{s_0}))$ and $o^*_{\{s_0\}\times M^2}(\calN_{pq}(f_{s_0}))$ respectively. Thus
\[ o^*_{J\times M^2}(\calN_{pq}(f_J))=o^*_{J\times M}(\wcalA_q(f_J))\wedge o^*_{J\times M}(\wcalD_p(f_J)). \]
For $\Gamma\in\calG_{2k,3k}^0(\vec{C})$, we define the coorientation $o^*_{J\times M^{2k}}(\calM_\Gamma(\vec{f}_J))$ so that its restriction to $\{s_0\}\times M^{2k}$ is equivalent to $o^*_{\{s_0\}\times M^{2k}}(\calM_\Gamma(\vec{f}_{s_0}))$. 

If $p$ and $r$ are critical loci of $f_J$ such that $i(p)=i(r)$, then the moduli space
\[ \bbcalM{f_J}{p}{r}=(\wcalD_p(f_J)\pitchfork \wcalA_r(f_J))\pitchfork \widetilde{L}, \]
where $\widetilde{L}$ is the level surface locus that lies just below $p$, is a compact 0-manifold in $\mathrm{Int}\,J\times M$ for a generic family $f_J$. At each point $b\in \bbcalM{f_J}{p}{r}$, the wedge product $o^*_{J\times M}(\wcalD_p(f_J))_b\wedge o^*_{J\times M}(\wcalA_r(f_J))_b\in \bigwedge^d T^*_b\widetilde{L}\subset \bigwedge^d T^*_b(J\times M)$ defines a coorientation of the flow line passing through $b$ (see Appendix~\ref{s:ori} (\ref{eq:coori_int})). We define the sign $\ve_{f_J}(p,r)_b=\pm 1$ so that the following equivalence holds.
\[ o^*_{J\times M}(\wcalD_p(f_J))_b\wedge o^*_{J\times M}(\wcalA_r(f_J))_b \sim \ve_{f_J}(p,r)_b\,\, \iota(-\mathrm{grad}\,f_{s_0})\,o(J\times M)_b. \]

%%%%%%%%%%%%%%%%%%%%%%%%%%%%%%
\subsection{(Co)orientations induced on the boundaries of $\wcalD$, $\wcalA$ at $i/i$-intersection}

Suppose that an $i/i$-intersection occurs at $s=u$. For a small number $\ve>0$, let $J=[u-\ve,u+\ve]$. 
For a parametrized Morse--Smale pair $(f_J,\mu_J)$ and its critical loci $p,q$, we shall describe the induced (co)orientations of the faces $\calF_r\calC\wcalD_p(f_J)$ (resp. $\calF_r\calC\wcalA_q(f_J)$) of $\partial_1\calC\wcalD_p(f_J)$ (resp. $\partial_1\calC\wcalA_q(f_J)$) of flow lines broken at a critical locus $r$, which are induced from the (co)orientation of $\calC\wcalD_p(f_J)$ (resp. $\calC\wcalA_p(f_J)$). 

Let $\bar{b}:\calC\wcalD_p(f_J)\to J\times M$ be the map that assigns to each (possibly broken) flow sequence the terminal endpoint. If $i(p)-i(r)=0$ and if $a$ is a point of $J\times M$ that is the image of $\bar{b}$ from a once broken flow sequence $\hat{a}$ in $\partial_1\calC\wcalD_p(f_J)$ broken at a critical locus $r$, then by Corollary~\ref{cor:compactification_tD} there is an open neighborhood $N_a$ of $a$ in $J\times M$ such that $\bar{b}^{-1}(N_a)$ is a disjoint union of finitely many half-disks whose set of components naturally corresponds to the finite set $\bbcalM{f_J}{p}{r}$. Let $\widehat{N}_{\hat{a}}$ be the component of $\bar{b}^{-1}(N_a)$ on which $\hat{a}$ lies. The restriction of $\bar{b}$ to $\widehat{N}_{\hat{a}}$ is an embedding and hence the coorientation $o^*_{J\times M}(\partial_1\calC\wcalD_p(f_J))_a$ makes sense by identifying $\widehat{N}_{\hat{a}}$ with $\bar{b}(\widehat{N}_{\hat{a}})$. The same is also true for $\partial_1\calC\wcalA_q(f_J)$ at a once broken flow sequence broken at $r$ such that $i(r)-i(q)=0$. 

Note that $\mathrm{Int}\,\bar{b}(\widehat{N}_{\hat{a}})$ is an open subset of $\wcalD_p(f_J)$ and its closure in $N_a$ is $\bar{b}(\widehat{N}_{\hat{a}})$. Hence the (co)orientation of $\wcalD_p(f_J)$ induces a (co)orientation of the boundary $\partial \bar{b}(\widehat{N}_{\hat{a}})$ at $a$. We define $o^*_{J\times M}(\partial_1\calC\wcalD_p(f_J))_a$ to be the one induced in this way. We also define $o^*_{J\times M}(\partial_1\calC\wcalA_q(f_J))_a$ similarly. 

\begin{Lem}\label{lem:ind_ori_CD}
Under the assumption above, let $p,r$ be critical loci of $f_J$ such that $f_J(p)>f_J(r)$ and $i(p)-i(r)=0$. Let $N_a$ and $a\in \bar{b}(\widehat{N}_{\hat{a}})$ be as above. Let $b$ be a point of $\bbcalM{f_J}{p}{r}$ such that $\widehat{N}_{\hat{a}}$ corresponds to $b$. Then the following identity in $\bigwedge^\bullet T_a^*(J\times M)$ holds.
\[ o^*_{J\times M}(\partial_1\calC\wcalD_p(f_J))_a=(-1)^{i(r)+1}\ve_{f_J}(p,r)_b\,ds\wedge  o^*_{J\times M}(\wcalD_r(f_J))_{a}.\]
\end{Lem}
\begin{proof}
Let $i=i(r)$. By assumptions $f_J(p)>f_J(r)$ and $i(p)-i(r)=0$, the index of $r$ is in $1\leq i(r)\leq d-1$. It suffices to check the assertion for one broken flow line. By parametrized Morse Lemma there is a local coordinate $(x_1,\ldots,x_d)$ around $r$ on which $f_s$ agrees with $\displaystyle f_s(r)-\frac{x_1^2}{2}-\cdots-\frac{x_i^2}{2}+\frac{x_{i+1}^2}{2}+\cdots+\frac{x_d^2}{2}$. In this coordinate, $\wcalD_r(f_J)$ agrees with $\{(s,x_1,\ldots,x_d)\in J\times \R^d;x_{i+1}=\cdots=x_d=0\}$ and $\wcalA_r(f_J)$ agrees with $\{(s,x_1,\ldots,x_d)\in J\times \R^d;x_1=\cdots=x_i=0\}$. We may put 
\[ o(\wcalD_r(f_J))=\beta\,ds\,dx_1\cdots dx_i\quad (\beta=\pm 1). \]
We may assume that the intersection of $\wcalD_p(f_J)$ with the plane $\{(s,x_1,\ldots,x_d)\in J\times\R^d;x_d=1\}$ agrees with the set
\[ \{(s,(s-u)\lambda,a_2,\ldots,a_i,0,\ldots,0,1);s\in J,a_2,\ldots,a_i\in\R\} \]
for some $\lambda\neq 0$. Hence $\wcalD_p(f_J)$ agrees locally with the set of points
\[ (s,(s-u)\lambda e^t,a_2e^t,\ldots,a_i e^t,0,\ldots,0,e^{-t}),\quad t\in \R. \]
By putting $a_1'=(s-u)\lambda e^t$, $a_2'=a_2 e^t,\ldots,a_i'=a_i e^t$, $s'=(s-u)/a_1'$, one may see that the closure of this agrees with the set of points 
\[ (s'a_1'+u,a_1',a_2',\ldots,a_i',0,\ldots,0,s'\lambda),\quad a_1',a_2',\ldots,a_i'\in\R,s'\in [-\ve/a_1',\ve/a_1']. \]
Hence for $a=(u,0,\ldots,0)\in J\times \R^d$, we may put
\[ o(\bar{b}(\widehat{N}_{\hat{a}}))_a=\alpha\lambda\, dx_1dx_2\cdots dx_idx_d\quad(\alpha=\pm 1). \]
Then 
\[ o(\partial \bar{b}(\widehat{N}_{\hat{a}}))_a = \iota\left(\frac{\partial}{\partial x_d}\right)\alpha\lambda\,dx_1\cdots dx_idx_d=(-1)^i\alpha\lambda\,dx_1\cdots dx_i=(-1)^i\alpha\beta\lambda\,o(\calD_r(f_{u}))_a. \]

On the other hand, by assumption we have
\[ \begin{split}
	o^*_{J\times M}(\wcalD_p(f_J))_b&=(-1)^{d}\alpha\,ds\,dx_{i+1}\cdots dx_{d-1},\quad\\
	o^*_{J\times M}(\wcalA_r(f_J))_b&=(-1)^{i(d-i)}\beta\,dx_1\cdots dx_i
\end{split} \]
for $b=(0,\ldots,0,1)$. Hence
\[ \begin{split}
	o^*_{J\times M}(\wcalD_p(f_J))_b\wedge o^*_{J\times M}(\wcalA_r(f_J))_b&=(-1)^{i+d}\alpha\beta\,ds\,dx_1\wedge\cdots dx_{d-1}\\
	&=(-1)^{i+1}\alpha\beta\, \iota\Bigl(-\frac{\partial}{\partial x_d}\Bigr)o(J\times M)
\end{split} \]
and we have $\ve_{f_J}(p,r)_b=(-1)^{i+1}\alpha\beta$. This together with the equality above, we obtain
\[ \begin{split}
  o(\partial_1\calC\wcalD_p(f_J))_a &= o(\partial \bar{b}(\widehat{N}_{\hat{a}}))_a = -\ve_{f_J}(p,r)_b\,o(\calD_r(f_{u}))_a,\\
  o^*_{J\times M}(\partial_1\calC\wcalD_p(f_J))_a &= -\ve_{f_J}(p,r)_b\,o^*_{J\times M}(\calD_r(f_{u}))_a\\
  &=(-1)^{i+1}\ve_{f_J}(p,r)_b\, ds\wedge o^*_{J\times M}(\wcalD_r(f_J))_a.
\end{split} \]
\end{proof}

\begin{Lem}\label{lem:ind_ori_CA}
Under the assumption above, let $q,r$ be critical points of $f$ such that $f_J(q)<f_J(r)$ and $i(r)-i(q)=0$. Let $N_a$ and $a\in \bar{b}(\widehat{N}_{\hat{a}})$ be as above. Let $b$ be a point of $\bbcalM{f}{r}{q}$ such that $\widehat{N}_{\hat{a}}$ corresponds to $b$. Then the following identity in $\bigwedge^\bullet T_a^*(J\times M)$ holds. 
\[ o^*_{J\times M}(\partial_1\calC\wcalA_q(f_J))_a=(-1)^{i(r)+d}\ve_{f_J}(r,q)_b\,ds\wedge o^*_{J\times M}(\wcalA_r(f_J))_a. \] 
\end{Lem}
\begin{proof}
Let $i=i(r)$. By assumptions $f_J(r)>f_J(q)$ and $i(r)-i(q)=0$, the index of $r$ is in $1\leq i(r)\leq d-1$. It suffices to check the assertion for one broken flow line. By parametrized Morse Lemma there is a local coordinate $(x_1,\ldots,x_d)$ around $r$ on which $f_s$ agrees with $\displaystyle f_s(r)-\frac{x_1^2}{2}-\cdots-\frac{x_i^2}{2}+\frac{x_{i+1}^2}{2}+\cdots+\frac{x_d^2}{2}$. In this coordinate, $\wcalD_r(f_J)$ agrees with $\{(s,x_1,\ldots,x_d)\in J\times \R^d;x_{i+1}=\cdots=x_d=0\}$ and $\wcalA_r(f_J)$ agrees with $\{(s,x_1,\ldots,x_d)\in J\times \R^d;x_1=\cdots=x_i=0\}$. We may put 
\[ o(\wcalA_r(f_J))=\beta\,ds\,dx_{i+1}\cdots dx_d\quad (\beta=\pm 1). \]
We may assume that the intersection of $\wcalA_q(f_J)$ with the plane $\{(s,x_1,\ldots,x_d)\in J\times\R^d;x_1=1\}$ agrees with the set
\[ \{(s,1,0,\ldots,0,a_{i+1},\ldots,a_{d-1},(s-u)\lambda);s\in J,a_{i+1},\ldots,a_{d-1}\in\R\} \]
for some $\lambda\neq 0$. Hence $\wcalA_q(f_J)$ agrees locally with the set of points
\[ (s,e^{-t},0,\ldots,0,a_{i+1}e^t,\ldots,a_{d-1}e^t,(s-u)\lambda e^t),\quad t\in \R. \]
By putting $a_{i+1}'=a_{i+1} e^t,\ldots,a_{d-1}'=a_{d-1} e^t$, $a_d'=(s-u)\lambda e^t$, $s'=(s-u)/a_d'$, one may see that the closure of this agrees with the set of points 
\[ (s'a_d'+u,s'\lambda,0,\ldots,0,a_{i+1}',\ldots,a_d'),\quad a_{i+1}',\ldots,a_d'\in\R,s'\in [-\ve/a_d',\ve/a_d']. \]
Hence for $a=(u,0,\ldots,0)\in J\times \R^d$, we may put
\[ o(\bar{b}(\widehat{N}_{\hat{a}}))_a=\alpha\lambda\, dx_1\,dx_{i+1}\cdots dx_d\quad(\alpha=\pm 1). \]
Then 
\[ o(\partial \bar{b}(\widehat{N}_{\hat{a}}))_a = \iota\left(\frac{\partial}{\partial x_1}\right)\alpha\lambda\,dx_1\,dx_{i+1}\cdots dx_d=\alpha\lambda\,dx_{i+1}\cdots dx_d=\alpha\beta\lambda\,o(\calA_r(f_{u}))_a. \]

On the other hand, by assumption we have
\[ \begin{split}
	o^*_{J\times M}(\wcalA_q(f_J))_b&=(-1)^{di+i+1}\alpha\,dx_2\cdots dx_i,\quad \\
	o^*_{J\times M}(\wcalD_r(f_J))_b&=\beta\,dx_{i+1}\cdots dx_d
\end{split} \]
for $b=(b_1,0,\ldots,0)$, $b_1>0$. Hence
\[ \begin{split}
	o^*_{J\times M}(\wcalD_r(f_J))_b\wedge o^*_{J\times M}(\wcalA_q(f_J))_b&=(-1)^{di+i+1}\alpha\beta\,dx_{i+1}\cdots dx_d\, ds \,dx_2\cdots dx_i\\
	&=-\alpha\beta\,ds\,dx_2\cdots dx_d=\alpha\beta\, \iota\Bigl(\frac{\partial}{\partial x_1}\Bigr)ds\,dx_1\cdots dx_d
\end{split} \]
and we have $\ve_{f_J}(r,q)_b=\alpha\beta$. This together with the equality above, we obtain
\[ \begin{split}
  o(\partial_1\calC\wcalA_q(f_J))_a &= o(\partial \bar{b}(\widehat{N}_{\hat{a}}))_a = \ve_{f_J}(r,q)_b\,o(\calA_r(f_{u}))_a,\\
  o^*_{J\times M}(\partial_1\calC\wcalA_q(f_J))_a &= \ve_{f_J}(r,q)_b\,o^*_{J\times M}(\calA_r(f_{u}))_a\\
  &=(-1)^{i+d}\ve_{f_J}(r,q)_b\,ds\wedge o^*_{J\times M}(\wcalA_r(f_J))_a.
\end{split} \]
\end{proof}

\subsection{Change of combinatorial propagator at $i/i$-intersection}\label{ss:change_g}

Suppose that an $i/i$-intersection between critical points (loci) $p$ and $q$ occurs at $s=u$. For a small number $\ve>0$, we may assume that the underlying $\Z$-modules of $C_*^{(u-\ve)}$ and $C_*^{(u+\ve)}$ are the same and we identify critical points and critical loci. We put $J=[u-\ve,u+\ve]$, $C_*^{(J)}=C_*^{(u-\ve)}=C_*^{(u+\ve)}$ and $P_*^{(J)}=P_*^{(u-\ve)}=P_*^{(u+\ve)}$. Let $h:C_*^{(J)}\to C_*^{(J)}$ be the homomorphism of homogeneous degree 0, defined for each critical point (locus) $x\in P_i^{(J)}$ by
\[ h(x)=\sum_{y\in P_i^{(J)}}\#\bbcalM{f_J}{x}{y}\cdot y,\quad \#\bbcalM{f_J}{x}{y}=\sum_{b\in\bbcalM{f_J}{x}{y}}\ve_{f_J}(x,y)_b.\]
Since the moduli space $\bbcalM{f_J}{x}{y}$ corresponds to an $i/i$-intersection, $h$ is non-zero only if $x=p$. Then for $b\in \bbcalM{f_J}{p}{q}$, we have $h(p)=\ve_{f_J}(p,q)_b\cdot q$. We denote the boundary operators of $C_*^{(u-\ve)}$ and $C_*^{(u+\ve)}$ by $\partial$ and $\partial'$ respectively. The following lemma describes the bifurcation of Morse complex at the $i/i$-intersection and is stated in several papers (e.g. \cite{Lau, Hu} and \cite[Lemma~5.1]{Fuk2}).

\begin{Lem}\label{lem:d-d'} Under the assumption above, we have
\[ \partial-\partial'=\partial h-h\partial'=\partial' h - h \partial, \]
or equivalently, $(1-h)\circ \partial'=\partial \circ (1-h)$ and $(1+h)\circ \partial=\partial'\circ (1+h)$, or $1+h:C^{(u-\ve)}_*\to C^{(u+\ve)}_*$ is a chain map. 
\end{Lem}
\begin{proof}
Let $p,q$ be critical loci of $f_J$ such that $i(p)-i(q)=0$. We check the identities
\[ \partial - \partial' -\partial h + h\partial' =0,\quad 
\partial - \partial' -\partial' h + h\partial =0.  \]
 We consider the boundary of the moduli spaces $\bbcalM{f_J}{p}{r}$ and $\bbcalM{f_J}{r'}{q}$ compactified using $\calC\wcalD$ and $\calC\wcalA$. The contribution of $\partial J$ is $\partial-\partial'$. The other contributions come from the broken flow lines of the $i/i$-intersection at $s=u$. For a critical locus $r$ with $i(r)=i(p)-1$, the broken flow line from $p$ to $r$ broken at $q$ contributes as $-\ve_{f_J}(p,q)_b\,\ve_{f_{u}}(q,r)_a$. Indeed, the coorientation of the boundary of $\bbcalM{f_J}{p}{r}$ is   
\[ \begin{split}
  &(-1)^{i(r)}o^*_{J\times M}(\partial\calC\wcalD_p(f_J))_a\wedge o^*_{J\times M}(\wcalA_r(f_J))_a\\
  &=(-1)^{i(r)}(-1)^{i(q)+1}\ve_{f_J}(p,q)_b\,ds\wedge o^*_{J\times M}(\wcalD_q(f_J))_a\wedge o^*_{J\times M}(\wcalA_r(f_J))_a\\
  &=(-1)^{i(r)+i(q)+1}\ve_{f_J}(p,q)_b\,ds\wedge \ve_{f_{u}}(q,r)_a\,\iota(-\mathrm{grad}\,f_{u})\,o(M)_a\\
  &=(-1)^{i(r)+i(q)}\ve_{f_J}(p,q)_b\,\ve_{f_{u}}(q,r)_a\,\iota(-\mathrm{grad}\,f_{u})\,o(J\times M)_a\\
  &=-\ve_{f_J}(p,q)_b\,\ve_{f_{u}}(q,r)_a\,\iota(-\mathrm{grad}\,f_{u})\,o(J\times M)_a.
\end{split}\]
Here we have used (\ref{eq:ind_ori_bd_int}) and Lemma~\ref{lem:ind_ori_CD}. This gives rise to $-\partial h$ ($=-\partial' h$). For a critical locus $r'$ with $i(r')=i(q)+1$, the broken flow line from $r'$ to $q$ broken at $p$ contributes as $\ve_{f_{u}}(r',p)_a\,\ve_{f_J}(p,q)_b$. Indeed, the coorientation of the boundary of $\bbcalM{f_J}{r'}{q}$ is 
\[ \begin{split}
  &o^*_{J\times M}(\wcalD_{r'}(f_J))_a\wedge o^*_{J\times M}(\partial\calC\wcalA_p(f_J))_a\\
  &=(-1)^{i(p)+d}\ve_{f_J}(p,q)_b\,o^*_{J\times M}(\wcalD_{r'}(f_J))_a\wedge ds\wedge o^*_{J\times M}(\wcalA_p(f_J))_a\\
  &=(-1)^{i(p)+d}(-1)^{d-i(r')}\ve_{f_J}(p,q)_b\,ds\wedge o^*_{J\times M}(\wcalD_{r'}(f_J))_a\wedge o^*_{J\times M}(\wcalA_p(f_J))_a\\
  &=(-1)^{i(r')+i(p)}\ve_{f_{u}}(r',p)_a\,ds\wedge \ve_{f_J}(p,q)_b\,\iota(-\mathrm{grad}\,f_{u})\,o(M)_a\\
  &=(-1)^{i(r')+i(p)+1}\ve_{f_{u}}(r',p)_a\,\ve_{f_J}(p,q)_b\,\iota(-\mathrm{grad}\,f_{u})\,o(J\times M)_a\\
  &=\ve_{f_{u}}(r',p)_a\,\ve_{f_J}(p,q)_b\,\iota(-\mathrm{grad}\,f_{u})\,o(J\times M)_a.
\end{split} \]
Here we have used (\ref{eq:ind_ori_bd_int}) and Lemma~\ref{lem:ind_ori_CA}. This gives rise to $+h\partial'$ ($=+h\partial$).
\end{proof}

The following corollary follows immediately from Lemma~\ref{lem:d-d'}. 
\begin{Cor}[Lemma~5.7 of \cite{Fuk2}]\label{cor:g'-g}
Let $g$ be a combinatorial propagator for $(C_*^{(u-\ve)},\partial)$. Then the endomorphism
\[ g'=(1+h)\circ g\circ(1-h)\in\End_1(C_*^{(u+\ve)}) \]
is a combinatorial propagator for $(C_*^{(u+\ve)},\partial')$. Moreover, by $hgh=hg'h=0$, 
\[g'-g=hg-gh=hg'-g'h.\]
\end{Cor}
%\clearpage

%%%%%%%%%%%%%%%%%%%%%%%%%%%%%%
\subsection{Orientations of some faces of $J\times \partial \bConf_{2k}(M)$}

The orientations of the principal face $J\times \partial_{ij}\bConf_{2k}(M)$ and the anomalous face $J\times \partial^a\bConf_{2k}(M)$ induced from the standard orientation $ds\wedge o(M)_{x_1}\wedge\cdots\wedge o(M)_{x_{2k}}$ of $J\times M^{2k}$ are as follows.
\begin{equation}\label{eq:o(JBl)}
 \begin{split}
  o(J\times \partial B\ell_{\Delta_{ij}}(M^{2k}))&=-\omega_2\wedge ds\wedge o(\Delta_{ij}),\\
  o(J\times \partial B\ell_{\Delta_a}(M^{2k}))&=\omega_{6k-4}\wedge ds\wedge o(\Delta_a).
\end{split}
\end{equation}
This can be checked as follows. Let $\Delta_a=\{(x_1,\ldots,x_{2k})\in M^{2k};x_1=\cdots=x_{2k}\}$. 
For $\vec{x}=(x_1,x_2,\ldots,x_{2k})\in \Delta_a$,
\[\begin{split}
&ds \wedge o(\Delta_a)_{\vec{x}}
\wedge \bigwedge_{i=2}^{2k}(du_i^{(1)}-du_1^{(1)})\wedge(du_i^{(2)}-du_1^{(2)})\wedge(du_i^{(3)}-du_1^{(3)})\\
&=(2k)^3 \, ds \wedge o(M^{2k})_{(s,\vec{x})},
\end{split}\]
where $(u_i^{(1)},u_i^{(2)},u_i^{(3)})$ is a local coordinate around $x_i$ and $o(\Delta_a)_{\vec{x}}=\bigwedge_{\ell=1}^3(du_1^{(\ell)}+du_2^{(\ell)}+\cdots+du_{2k}^{(\ell)})$. The part $\bigwedge_{i=2}^{2k}(du_i^{(1)}-du_1^{(1)})\wedge(du_i^{(2)}-du_1^{(2)})\wedge(du_i^{(3)}-du_1^{(3)})$ gives an orientation of the fiber of the normal bundle $N_{\Delta_a}\to \Delta_a$ and the part $ds \wedge o(\Delta_a)_{\vec{x}}$ is a 4-form. Hence the orientation of the unit sphere bundle of $N_{\Delta_a}$ induced from the left hand side of the above expression is $ds\wedge o(\Delta_a)\wedge \omega_{6k-4}$. The orientation of $J\times \partial B\ell_{\Delta_{ij}}(M^{2k})$ is similar to that of $\partial B\ell_{\Delta_{ij}}(M^{2k})$ given in \S\ref{ss:ori_face}. 

Note that if $n^*$ is the metric dual of an inward normal vector field on a face of $J\times \partial\bConf_{2k}(M)$ of the type considered above, then by (\ref{eq:o(JBl)}), the products $n^*\wedge o(J\times \partial B\ell_{\Delta_{ij}}(M^{2k}))$ and $n^*\wedge o(J\times \partial B\ell_{\Delta_a}(M^{2k}))$ are both equivalent to the standard orientation of $J\times \bConf_{2k}(M)$. 

Now the integer $\#\calM_\Gamma^\loc(-\mathrm{grad}\,\vec{f}_J)$ is defined by the sum of signs determined by exterior products of coorientations of submanifolds of the $\bConf_{2k}^\loc(\R^3)$-bundle over $J\times M_0$ as in Definition~\ref{def:Mloc} and by (\ref{eq:o(JBl)}).

%%%%%%%%%%%%%%%%%%%%%%%%%%%%%%
\subsection{Standard co-orientations of $\calM_\Gamma$ in 1-parameter family}

Let $\Gamma$ be a trivalent graph with $2k$ vertices and without bivalent vertices such that $0\leq 3(2k-3k)+\sum_{i=1}^{3k}\eta_i +1\leq 1$. In a generic 1-parameter family $(f_J,\mu_J)$, the moduli space $\calM_\Gamma(\vec{f}_J)$ is a smooth manifold of dimension $3(2k-3k)+\sum_{i=1}^{3k}\eta_i +1$ and is the transversal intersection of the preimages of $\calM_2(f_j)$'s and $\calN_{pq}(f_j)$'s in $J\times \Conf_{2k}(M)$. We may define $o^*_{J\times \Conf_{2k}(M)}(\calM_\Gamma(\vec{f}_J))$ by the exterior product of coorientations of the preimages of $\calM_2(f_j)$'s and $\calN_{pq}(f_j)$'s in $J\times \Conf_{2k}(M)$ as in \S\ref{sss:std_ori}.

%%%%%%%%%%%%%%%%%%%%%%%%%%%%%%
\subsection{(Co)orientations induced on $\partial \bcalM_\Gamma$}\label{ss:ddd}

Let $(f_J,\mu_J)=\{(f_s,\mu_s)\}_{s\in J}$ be a generic 1-parameter family of Morse pairs. Let $\vec{f}_J$ be a sequence of 1-parameter families of Morse pairs that is obtained from $\vec{f}$ by replacing $f_1$ with $f_J$. For a graph $\Gamma\in\calG_{2k,3k}^0(\vec{C})$, we consider the co-orientation of $\partial \bcalM_\Gamma(\vec{f}_J)$ induced from $o^*_{J\times M^{2k}}(\calM_\Gamma(\vec{f}_J))$ defined above. Let $d''\Gamma=\sum_{e\in\Se(\Gamma)}d_e''\Gamma$, where

\begin{align*}
d_e''\fig{d_3.eps}&=-\sum_{{{r_i\in P_*^{(i)}}\atop{i(r_i)=i(p_i)}}}\fig{d_4.eps}
	+\sum_{{{s_i\in P_*^{(i)}}\atop{i(s_i)=i(q_i)}}}\fig{d_5.eps}
	\qquad (\beta(i)=e)
\end{align*}
and $\mathrm{or}(d_e''\Gamma)$ is the induced one. 

\begin{Prop}\label{prop:dbM_1-para}
Suppose that $d=3$ and that $(\vec{f}_J,\vec{\mu}_J)$ is generic as in Proposition~\ref{prop:comp_M_1-para}. Let $\Gamma$ be a graph in $\calG_{2k,3k}^0(\vec{C})$. We have
\[\begin{split}
 & \sum_{\sigma\in\mathfrak{S}_{3k}}\sum_{\tau\subset E(\Gamma)}(\#\calM_{\Gamma_\sigma^\tau}(\vec{f}_{s_1})-\#\calM_{\Gamma_\sigma^\tau}(\vec{f}_{s_0}))\\
&=\left\{\begin{array}{ll}
  \displaystyle\sum_{\sigma,\tau}(\#\calM_{(-d+d'+d'')\Gamma_\sigma^\tau}(\vec{f}_J)+\#\calM_{\Gamma_\sigma^\tau}^\loc(-\mathrm{grad}\,\vec{f}_J)) & \mbox{if $E(\Gamma)=\Comp(\Gamma)$}\\
  \displaystyle\sum_{\sigma,\tau}\#\calM_{(-d+d'+d'')\Gamma_\sigma^\tau}(\vec{f}_J) & \mbox{if $E(\Gamma)\neq \Comp(\Gamma)$}
  \end{array}\right. 
\end{split}\]
\end{Prop}
\begin{proof}
By Proposition~\ref{prop:comp_M_1-para}, we know the types of the graphs that may occur at the boundary of $\bcalM_\Gamma(\vec{f}_J)$. We check that $0=\#\partial\bcalM_\Gamma(\vec{f}_J)$ is the sum of $-\#\calM_{\Gamma_\sigma^\tau}(\vec{f}_{s_1})+\#\calM_{\Gamma_\sigma^\tau}(\vec{f}_{s_0})+\#\calM_{(-d+d'+d'')\Gamma_\sigma^\tau}(\vec{f}_J)$ and the contribution of $J\times \partial^\mathrm{hi}\bConf_{2k}(M)$. 

Suppose for simplicity that separated edges of $\Gamma$ are labeled $1,2,\ldots,a$. For a number $\ell$ in $1\leq \ell\leq 3k$, put
\[ \begin{split}
	\widetilde{\Sigma}_\ell&=\left\{\begin{array}{ll}
	\bigcap_{{{1\leq j\leq a}\atop{j\neq \ell}}}\widetilde{H}_j\cap\bigcap_{j=a+1}^{3k}\widetilde{\Theta}_j
		& \mbox{if $1\leq \ell\leq a$}\\
	\bigcap_{j=1}^a\widetilde{H}_j\cap\bigcap_{{{a+1\leq j\leq 3k}\atop{j\neq \ell}}}\widetilde{\Theta}_j
		& \mbox{if $a+1\leq \ell\leq 3k$}
	\end{array}\right. 
\end{split}\]
Then $\codim\widetilde{\Sigma}_\ell=\codim\calM_\Gamma(\vec{f}_J)-\codim\widetilde{H}_\ell=6k-4\equiv 0$ (mod 2). 

First, we consider the contribution of $J\times \partial\bConf_{2k}(M)$. The vanishing of the contributions of the hidden faces $\partial_A\bConf_{2k}(M)$ with $A\subsetneqq \{1,2,\ldots,2k\}$ follows from Lemmas~\ref{lem:symmetry}, \ref{lem:symmetry2} and \ref{lem:dilation}. The contributions of the principal face and of the anomalous face are $-\#\calM_{d\Gamma}(\vec{f}_J)$ and $\#\calM_{\Gamma}^\loc(-\mathrm{grad}\,\vec{f}_J)$ respectively. This is immediate from the sign convention and from (\ref{eq:o(JBl)}). The contribution of the hidden faces $\partial_{A\cup\{\infty\}}\bConf_{2k}(M)$ are as follows. Recall that the interior of the hidden face $\partial_{A\cup\{\infty\}}\bConf_{2k}(M)$ is diffeomorphic to the space $\Conf_{2k-j}(M)\times \Conf_j^\infty(\R^3)$ (Proof of Proposition~\ref{prop:dM=Md}). Let $f_j^\infty:\R^3\to \R$ ($j=2,\ldots,3k$) be the linear map such that $\varphi_\infty^*f_j^\infty$ agrees with $f_j$ near $\infty_M$ and let $f_s^\infty:\R^3\to \R$ ($s\in J$) be the linear map such that $\varphi_{\infty s}^*f_s^\infty$ agrees with $f_s$ near $\infty_M$. Let $\vec{f}_s^\infty=(f_s^\infty,f_2^\infty,\ldots,f_{3k}^\infty)$, $\vec{f}_J^\infty=\{\vec{f}_s^\infty\}_s$ and let $B=V(\Gamma)\setminus A$. Suppose that $E(\Gamma/\Gamma_B)=\Comp(\Gamma/\Gamma_B)$. Let $\calM_{\Gamma/\Gamma_B}^\infty(\vec{f}_J^\infty\setminus(\vec{f}_J^\infty)_B)$ be the space of linear graphs in $J\times \R^3$ modulo the dilation of $\R^3$ whose edge labeled $\ell\neq 1$ (resp. $\ell=1$) follows the negative gradient of $f_\ell^\infty$ (resp. $f_s^\infty$). Let $\pi_1:J\times \Conf_{2k-j}(M)\times \Conf_j^\infty(\R^3)\to J\times \Conf_{2k-j}(M)$ and $\pi_2:J\times \Conf_{2k-j}(M)\times \Conf_j^\infty(\R^3)\to J\times \Conf_j^\infty(\R^3)$ be the projections. Then the face of $\partial\bcalM_\Gamma(\vec{f}_J)$ coming from $\partial_{A\cup\{\infty\}}\bConf_{2k}(M)$ is diffeomorphic to
\[ \pi_1^{-1}\calM_{\Gamma_B}((\vec{f}_J)_B) \cap \pi_2^{-1}\calM_{\Gamma/\Gamma_B}^\infty(\vec{f}_J^\infty\setminus(\vec{f}_J^\infty)_B). \]
If the number of edges in $E(\Gamma)$ that intersect both $V(\Gamma_A)$ and $V(\Gamma_B)$ is $m$, then the codimension of $\calM_{\Gamma/\Gamma_B}^\infty(\vec{f}_J^\infty\setminus(\vec{f}_J^\infty)_B)$ is $3j+m$. Since $\dim{J\times \Conf_j^\infty(\R^3)}=3j$, $m$ must be zero if $\calM_{\Gamma/\Gamma_B}^\infty(\vec{f}_J^\infty\setminus(\vec{f}_J^\infty)_B)\neq \emptyset$. That $m=0$ implies that $A=\{1,2,\ldots,2k\}$. But in such a case $\Gamma_B$ is empty and the translation in $\R^3$ acts on $\calM_{\Gamma/\Gamma_B}^\infty(\vec{f}_J^\infty\setminus(\vec{f}_J^\infty)_B)$ freely. By a dimensional reason, this shows that $\calM_{\Gamma/\Gamma_B}^\infty(\vec{f}_J^\infty\setminus(\vec{f}_J^\infty)_B)$ must be empty. 

Next, we consider the contributions of the inner boundaries. Let $X$ be a graph obtained from $\Gamma$ by replacing an edge labeled 1 with a broken edge such that $\calM_X(\vec{f}_J)$ is 0-dimensional. We shall describe the co-orientation of the face $\calS_X$ of $\partial \bcalM_\Gamma(\vec{f}_J)$ corresponding to $X$ induced from the standard co-orientation of $\calM_\Gamma(\vec{f}_J)$ using (\ref{eq:ind_ori_boundary}) and (\ref{eq:ind_ori_bd_int}). In the following, we let $\ell=1$. 

(1) $X=\fig{d_4_1.eps}$. For $\calM_X(\vec{f}_J)$ to be 0-dimensional, $i(r_\ell)=i(p_\ell)-1$ or $i(r_\ell)=i(p_\ell)$. When $i(r_\ell)=i(p_\ell)$, the co-orientation of $\calF_{r_\ell}\bcalM_\Gamma(\vec{f}_J)$ induced from the standard one
\begin{equation}\label{eq:std_ori_2}
 o^*_{J\times M^{2k}}(\calM_\Gamma(\vec{f}_J))
	=o^*_{J\times M}(\widetilde{\calD}_{p_\ell}(f_\ell))\wedge o^*_{J\times M}(\widetilde{\calA}_{q_\ell}(f_\ell))\wedge o^*_{J\times M^{2k}}(\widetilde{\Sigma}_\ell) 
\end{equation}
is given by
\[ \begin{split}
	&(-1)^{(d+1)-1}(-1)^{(2kd+1)-1} o^*_{J\times M}(\widetilde{\calA}_{q_\ell}(f_J))\\
	&\hspace{10mm}\wedge (-1)^{i(r_\ell)+1}\ve_{f_J}(p_\ell,r_\ell)\,ds\wedge o^*_{J\times M}(\widetilde{\calD}_{r_\ell}(f_J))\wedge o^*_{J\times M^{2k}}(\widetilde{\Sigma}_\ell)\\
	&=(-1)^{i(q_\ell)+i(r_\ell)+d+1}\ve_{f_J}(p_\ell,r_\ell)\,ds\wedge o^*_{J\times M}(\widetilde{\calA}_{q_\ell}(f_J))
		\wedge o^*_{J\times M}(\widetilde{\calD}_{r_\ell}(f_J))\wedge o^*_{J\times M^{2k}}(\widetilde{\Sigma}_\ell)\\
	&=-\ve_{f_J}(p_\ell,r_\ell)\,ds\wedge o^*_{J\times M}(\widetilde{\calA}_{q_\ell}(f_J))
		\wedge o^*_{J\times M}(\widetilde{\calD}_{r_\ell}(f_J))\wedge o^*_{J\times M^{2k}}(\widetilde{\Sigma}_\ell).
	\end{split}\]
Here we used Lemma~\ref{lem:ind_ori_CD}. This is opposite to the standard co-orientation $o^*_{J\times M^{2k}}(\calM_X(\vec{f}_J))$.

When $i(r_\ell)=i(p_\ell)-1$, the co-orientation of $\calF_{r_\ell}\bcalM_\Gamma(\vec{f}_J)$ induced from the standard one is given by
\[ \begin{split}
	&(-1)^{d-1}(-1)^{(2kd+1)-1} \ve_{f_{s_0}}(p_\ell,r_\ell)\,o^*_{J\times M}(\widetilde{\calA}_{q_\ell}(f_J))\\
	&\hspace{10mm}\wedge (-1)^{i(r_\ell)+1} o^*_{J\times M}(\widetilde{\calD}_{r_\ell}(f_J))\wedge o^*_{J\times M^{2k}}(\widetilde{\Sigma}_\ell)\\
	&=(-1)^{i(r_\ell)+1} \ve_{f_{s_0}}(p_\ell,r_\ell)\, o^*_{J\times M}(\widetilde{\calA}_{q_\ell}(f_J))
		\wedge o^*_{J\times M}(\widetilde{\calD}_{r_\ell}(f_J))\wedge o^*_{J\times M^{2k}}(\widetilde{\Sigma}_\ell)\\
	&=(-1)^{i(q_\ell)+1} \ve_{f_{s_0}}(p_\ell,r_\ell)\, o^*_{J\times M}(\widetilde{\calA}_{q_\ell}(f_J))
		\wedge o^*_{J\times M}(\widetilde{\calD}_{r_\ell}(f_J))\wedge o^*_{J\times M^{2k}}(\widetilde{\Sigma}_\ell).
	\end{split}\]
Here we used Lemma~\ref{lem:ind_ori_D}.

(2) $X=\fig{d_5_1.eps}$. For $\calM_X(\vec{f}_J)$ to be 0-dimensional, $i(s_\ell)=i(q_\ell)+1$ or $i(s_\ell)=i(q_\ell)$. When $i(s_\ell)=i(q_\ell)$, the co-orientation of $\calF_{s_\ell}\bcalM_\Gamma(\vec{f}_J)$ induced from the standard one (\ref{eq:std_ori_2}) is given by
\[ \begin{split}
	&(-1)^{(d+1)-1}(-1)^{(2kd+1)-1}(-1)^{d-i(p_\ell)}(-1)^{i(s_\ell)+d} \ve_{f_J}(s_\ell,q_\ell)\,ds\wedge o^*_{J\times M}(\widetilde{\calA}_{s_\ell}(f_J))\\
	&\hspace{10mm}\wedge o^*_{J\times M}(\widetilde{\calD}_{p_\ell}(f_J))\wedge o^*_{J\times M^{2k}}(\widetilde{\Sigma}_\ell)\\
	&=(-1)^{i(p_\ell)+i(s_\ell)+d} \ve_{f_J}(s_\ell,q_\ell)\,ds\wedge o^*_{J\times M}(\widetilde{\calA}_{s_\ell}(f_J))
	\wedge o^*_{J\times M}(\widetilde{\calD}_{p_\ell}(f_J))\wedge o^*_{J\times M^{2k}}(\widetilde{\Sigma}_\ell)\\
	&= \ve_{f_J}(s_\ell,q_\ell)\,ds\wedge o^*_{J\times M}(\widetilde{\calA}_{s_\ell}(f_J))
	\wedge o^*_{J\times M}(\widetilde{\calD}_{p_\ell}(f_J))\wedge o^*_{J\times M^{2k}}(\widetilde{\Sigma}_\ell).
\end{split}\]
Here we used Lemma~\ref{lem:ind_ori_CA}. This agrees with the standard co-orientation. 

When $i(s_\ell)=i(q_\ell)+1$, the co-orientation of $\calF_{s_\ell}\bcalM_\Gamma(\vec{f}_J)$ induced from the standard one is given by
\[ \begin{split}
	&(-1)^{d-1}(-1)^{(2kd+1)-1}(-1)^{d-i(p_\ell)} \ve_{f_{s_0}}(s_\ell,q_\ell)\, o^*_{J\times M}(\widetilde{\calA}_{s_\ell}(f_J))\\
	&\hspace{10mm}\wedge o^*_{J\times M}(\widetilde{\calD}_{p_\ell}(f_J))\wedge o^*_{J\times M^{2k}}(\widetilde{\Sigma}_\ell)\\
	&=(-1)^{i(s_\ell)+1} \ve_{f_{s_0}}(s_\ell,q_\ell)\, o^*_{J\times M}(\widetilde{\calA}_{s_\ell}(f_J))
	\wedge o^*_{J\times M}(\widetilde{\calD}_{p_\ell}(f_J))\wedge o^*_{J\times M^{2k}}(\widetilde{\Sigma}_\ell).
\end{split}\]
Here, we used Lemma~\ref{lem:ind_ori_A}.

(3) $X=\fig{d_2.eps}$. The induced co-orientation on the boundary is as in Lemma~\ref{lem:ind_ori_M2}, which differs from the standard co-orientation by $(-1)^{2d-1}(-1)^{(2kd+1)-1}(-1)^{i(r_\ell)}=(-1)^{i(r_\ell)+1}$. 

Now we have seen that the signs in the formula of the definitions of $d'$ and $d''$ are consistent with the induced co-orientations on the boundary of $\bcalM_\Gamma(\vec{f}_J)$.
\end{proof}
%\clearpage

%%%%%%%%%%%%%%%%%%%%%%%%%%%%%%
%%%%%%%%%%%%%%%%%%%%%%%%%%%%%%
\mysection{Proof of main theorem}{s:proof_main}

We shall prove that $\widehat{Z}_{2k,3k}(\vec{f})$ is invariant under bifurcations of types (1), (2), (3), (4) in Lemma~\ref{lem:generic_1-para} and complete the proof of Theorem~\ref{thm:Z_invariant}.

%%%%%%%%%%%%%%%%%%%%%%%%%%%%%%
\subsection{Invariance on ordered 1-parameter family without $i/i$-intersections}

We check the invariance of $\widehat{Z}_{2k,3k}(\vec{f})$ with respect to bifurcations of type (4) for different Morse indices in Lemma~\ref{lem:generic_1-para}.

\begin{Lem}\label{lem:inv_ordered}
Suppose that a generic 1-parameter family $\{(f_s,\mu_s)\}_{s\in J}$, $J=[s_0,s_1]$, of Morse pairs is ordered and has no $i/i$-intersections over $J$. Then
\[ \widehat{Z}_{2k,3k}(\vec{f}_{s_0})=\widehat{Z}_{2k,3k}(\vec{f}_{s_1}).\]
\end{Lem}
\begin{proof}
Note that the moduli space $\calM_{d''\Gamma}(\vec{f}_J)$ is empty since $f_J$ has no $i/i$-intersections. By Proposition~\ref{prop:dbM_1-para}, the difference $Z_{2k,3k}(\vec{f}_{s_1})-Z_{2k,3k}(\vec{f}_{s_0})$ equals
\[ \Tr_{\vec{g}}\Bigl[ \sum_{\Gamma\in\calG^0_{2k,3k}(\vec{C})} \#\calM_{(-d+d')\Gamma}(\vec{f}_J)\,\Gamma \Bigr]+Z_{2k,3k}^{\mathrm{anomaly}}(-\mathrm{grad}\,\vec{f}_J). \]
As in the proof of Lemma~\ref{lem:indep_g}, the sum $\Tr_{\vec{g}}\Bigl[\sum_{\Gamma\in \calG_{2k,3k}^0(\vec{C})}\#\calM_{-d\Gamma}(\vec{f}_J)\Gamma\Bigr]$ vanishes by the IHX relation. Moreover, $\Tr_{\vec{g}}\Bigl[\sum_{\Gamma\in \calG_{2k,3k}^0(\vec{C})}\#\calM_{d'\Gamma}(\vec{f}_J)\Gamma\Bigr]$ equals
\[\Tr_{\vec{g}}\Bigl[\sum_{i=1}^{3k}\sum_{{{\Gamma'(\tilde{p}_i,\tilde{q}_i)_i}\atop{i(\tilde{p}_i)=i(\tilde{q}_i)}}}
		(-1)^{i(\tilde{p}_i)+1}\#\calM_{\Gamma'(\tilde{p}_i,\tilde{q}_i)_i}(\vec{f}_J)\,d'^*\Gamma'(\tilde{p}_i,\tilde{q}_i)_i\Bigr],\]
where the second sum is taken over graphs of degree $(\eta_1,\ldots,\eta_{3k})$, $\eta_j=1$ ($j\neq i$), $\eta_i=0$, such that $\beta(i)\in\Se(\Gamma)$, and $d'^*\Gamma'(\tilde{p}_i,\tilde{q}_i)_i$ denotes
\[ \sum_{{{x_i\in P_*^{(i)}}\atop{i(x_i)=i(\tilde{q}_i)+1}}}\partial_{x_i\tilde{p}_i}^{(i)}\Gamma(x_i,\tilde{q}_i)_i
	+\sum_{{{y_i\in P_*^{(i)}}\atop{i(y_i)=i(\tilde{p}_i)-1}}}\partial_{\tilde{q}_iy_i}^{(i)}\Gamma(\tilde{p}_i,y_i)_i
	+\delta_{\tilde{p}_i\tilde{q}_i}\Gamma(\emptyset,\emptyset)_i.\]
For each $\tilde{p}_i,\tilde{q}_i\in P_*^{(i)}$ with $i(\tilde{p}_i)=i(\tilde{q}_i)$, we have
\[\begin{split}
\Tr_{\vec{g}}&\Bigl[
	\sum_{{{x_i\in P_*^{(i)}}\atop{i(x_i)=i(\tilde{q}_i)+1}}}\partial_{x_i\tilde{p}_i}^{(i)}\Gamma(x_i,\tilde{q}_i)_i
		+\sum_{{{y_i\in P_*^{(i)}}\atop{i(y_i)=i(\tilde{p}_i)-1}}}\partial_{\tilde{q}_iy_i}^{(i)}\Gamma(\tilde{p}_i,y_i)_i
		+\delta_{\tilde{p}_i\tilde{q}_i}\Gamma(\emptyset,\emptyset)_i
		\Bigr]\\
	=&\Tr_{\ldots,\partial^{(i)}g^{(i)}+g^{(i)}\partial^{(i)},\ldots}\Bigl[ \Gamma(\tilde{p}_i,\tilde{q}_i)_i\Bigr]+\Tr_{\vec{g}}\Bigl[\delta_{\tilde{p}_i\tilde{q}_i}\Gamma(\emptyset,\emptyset)_i\Bigr]\\
	=&\Tr_{\ldots,\mathrm{id},\ldots}\Bigl[\Gamma(\tilde{p}_i,\tilde{q}_i)_i\Bigr]+\Tr_{\vec{g}}\Bigl[\delta_{\tilde{p}_i\tilde{q}_i}\Gamma(\emptyset,\emptyset)_i\Bigr]=\delta_{\tilde{p}_i\tilde{q}_i}\Tr_{\vec{g}}\Bigl[-\Gamma(\emptyset,\emptyset)_i+\Gamma(\emptyset,\emptyset)_i\Bigr]=0.
\end{split}\]
Hence we have
\[ Z_{2k,3k}(\vec{f}_{s_1})-Z_{2k,3k}(\vec{f}_{s_0})=Z_{2k,3k}^\mathrm{anomaly}(-\mathrm{grad}\,\vec{f}_J). \]

For the correction terms of $\widehat{Z}_{2k,3k}(\vec{f}_{s_0})$ and $\widehat{Z}_{2k,3k}(\vec{f}_{s_1})$, we may choose the same spin 4-manifold $W$. Then we choose generic GM sections $\vec{\gamma}_W$ and ${\vec{\gamma}\phantom{|}'_{\kern-1mm W}}$ of $\Gamma(T^vW)^{3k}$ as in \S\ref{ss:framing} that are extensions of $-\mathrm{grad}\,\vec{f}_{s_0}$ and $-\mathrm{grad}\,\vec{f}_{s_1}$ respectively. Then it follows from Lemma~\ref{lem:Z(rho)=0} that
\[ Z_{2k,3k}^\mathrm{anomaly}(-\mathrm{grad}\,\vec{f}_J)-Z_{2k,3k}^\mathrm{anomaly}(\vec{\gamma}_W)+Z_{2k,3k}^\mathrm{anomaly}(\vec{\gamma}\phantom{|}'_{\kern-1mm W})=0. \]
This completes the proof.
\end{proof}

%%%%%%%%%%%%%%%%%%%%%%%%%%%%%%
\subsection{Invariance at level exchange bifurcation}

We check the invariance of $\widehat{Z}_{2k,3k}(\vec{f})$ with respect to bifurcations of type (1) in Lemma~\ref{lem:generic_1-para}.

\begin{Lem}\label{lem:inv_exchange}
Suppose that $s_0\in J$ is a level exchange bifurcation for the generic 1-parameter family $\{(f_s,\mu_s)\}_{s\in J}$ of Morse pairs. If $\ve$ is sufficiently small,
\[ \widehat{Z}_{2k,3k}(\vec{f}_{s_0-\ve})=\widehat{Z}_{2k,3k}(\vec{f}_{s_0+\ve}).\]
\end{Lem}
The proof is the same as Lemma~\ref{lem:inv_ordered}.

%%%%%%%%%%%%%%%%%%%%%%%%%%%%%%
\subsection{Invariance at birth-death bifurcation}

We check the invariance of $\widehat{Z}_{2k,3k}(\vec{f})$ with respect to bifurcations of type (2) in Lemma~\ref{lem:generic_1-para}.

We say that a birth-death point $v$ in a 1-parameter family $\{(f_s,\mu_s)\}_{s\in [0,1]}$, say at $s=s_0$, is {\it independent} if on a neighborhood of $s_0$ in $[0,1]$ the descending and the ascending manifold of $v$ are disjoint from all the other descending and ascending manifolds of critical points. 

\begin{Lem}[\cite{HW}, page 62]
A 1-parameter family $\{(f_s,\mu_s)\}_{s\in[0,1]}$ of pairs of generalized Morse functions and metrics on $M_0$ can be deformed so that every birth-death points are independent.
\end{Lem}
\begin{Lem}\label{lem:inv_bd}
Suppose that $s_0\in J$ is a parameter on which an independent birth-death point $v$ occurs in a generic 1-parameter family. If $\ve$ is sufficiently small,
\[ \widehat{Z}_{2k,3k}(\vec{f}_{s_0-\ve})=\widehat{Z}_{2k,3k}(\vec{f}_{s_0+\ve}).\]
\end{Lem}

\begin{proof}
We prove the lemma only for death point since the case of birth point is symmetric. If $(p,q)$ is the critical point pair at $s=s_0-\ve$ that disappears on $s\geq s_0$, then the Morse complex at $s=s_0-\ve$ is the direct sum $C_*^{(s_0+\ve)}\oplus C_*^\mathrm{elem}$, where $C_*^{(s_0+\ve)}=(C_*^{(s_0+\ve)},\partial^{(s_0+\ve)})$ is the Morse complex at $s=s_0+\ve$ and 
\[ C_*^\mathrm{elem}=\{0\to C_{i+1}^\mathrm{elem}=\Z^{\{p\}} \stackrel{\approx}{\to} C_{i}^\mathrm{elem}=\Z^{\{q\}}\to 0\}. \]

Choose a combinatorial propagator $g$ of $C_*^{(s_0+\ve)}$. The acyclic complex $C_*^\mathrm{elem}$ has a unique combinatorial propagator $g^\mathrm{elem}$ defined by $g^\mathrm{elem}(q)=p$. We consider $g$ and $g^\mathrm{elem}$ as homogeneous degree 1 maps of $C_*^{(s_0+\ve)}\oplus C_*^\mathrm{elem}$ by setting $g(C_*^\mathrm{elem})=0$ and $g^\mathrm{elem}(C_*^{(s_0+\ve)})=0$. Then one can check that $g'=g+g^\mathrm{elem}$ is a combinatorial propagator for $C_*^{(s_0+\ve)}\oplus C_*^\mathrm{elem}$.

We need only to check the identity of the lemma in the case where a gluing of trajectories happens at $v$. Suppose that the separated edge labeled by 1 in a flow graph of $\Gamma(p,q)_1\in\calG_{2k,3k}^0(\vec{C}^{(s_0-\ve)})$ converges to a broken edge as $s\to s_0$ and that $p$ and $q$ converges to $v$. Then by Proposition~\ref{prop:glue_bd} we have
\begin{equation}\label{eq:sign_glue}
 \begin{split}
	&\Tr_{g',\cdots}\bigl(\Gamma(p,q)_1\bigr)\cdot
	\#\bcalM_{\Gamma(p,q)_1}(\vec{f}_{s_0-\ve})
	+\Tr_{g',\cdots}\bigl(\Gamma(\emptyset,\emptyset)_1\bigr)
	\cdot\#\bcalM_{\Gamma(\emptyset,\emptyset)_1}(\vec{f}_{s_0-\ve})\\
	&=\Tr_{g,\cdots}\bigl(\Gamma(\emptyset,\emptyset)_1\bigr)
	\Bigl(-\#\bcalM_{\Gamma(p,q)_1}(\vec{f}_{s_0-\ve})
	+\#\bcalM_{\Gamma(\emptyset,\emptyset)_1}(\vec{f}_{s_0-\ve})\Bigr)\\
	&=\Tr_{g,\cdots}\bigl(\Gamma(\emptyset,\emptyset)_1\bigr)\cdot
	\#\bcalM_{\Gamma(\emptyset,\emptyset)_1}(\vec{f}_{s_0+\ve}).
\end{split} 
\end{equation}
Here, we must check that the signs of the boundaries of the 1-cobordism are correct. It suffices to check the coorientations for the standard model $h_u$ in \S\ref{ss:gluing_separate} for a 1-parameter family around a death bifurcation. For $u<0$ with $|u|$ small and for $x=(x_1,\ldots,x_d)\in \calA_{p_+}(h_u), y=(y_1,\ldots,y_d)\in \calD_{p_-}(h_u)$, put
\[ o^*_M(\calA_{p_+}(h_u))_x=\alpha\,dx_2\cdots dx_i,\quad
o^*_M(\calD_{p_-}(h_u))_y=\beta\,dy_{i+1}\cdots dy_d\quad (\alpha,\beta\in\{-1,1\}).
\]
By convention, $o^*_{M\times M}(\calN_{p_-p_+}(h_u))_{(x,y)}=o^*_M(\calA_{p_+}(h_u))_x\wedge o^*_M(\calD_{p_-}(h_u))_y$.
On the other hand,
\[ \begin{split}
  &o^*_M(\calD_{p_-}(h_u))_x\wedge o^*_M(\calA_{p_+}(h_u))_x=\alpha\beta dx_{i+1}\cdots dx_ddx_2\cdots dx_i\\
  &=(-1)^{i-1}\alpha\beta dx_2\cdots dx_d=(-1)^{i-1}\alpha\beta\, \iota\Bigl(\frac{\partial}{\partial x_1}\Bigr)o(\R^d)_x.
\end{split} \]
Hence $\ve_{h_u}(p_-,p_+)=(-1)^{i-1}\alpha\beta$. For $u>0$ small, consider points $x'=(x_1',\ldots,x_d'),y'=(y_1',\ldots,y_d')\in \R^d$ such that $x'$ is close to $x$, $y'$ is close to $y$ and $y'=\Phi_{h_u}^t(x')$. By using the explicit solution (\ref{eq:solution}) and by convention, $o(\calM_2(h_u))_{(x',y')}$ is given by
\[ \displaystyle(-dh_u)_{y'}\wedge (dx_1'+\delta(x_1')dy_1')\wedge\bigwedge_{k=2}^i(dx_k'+e^tdy_k')\wedge\bigwedge_{k=i+1}^d(dx_k'+e^{-t}dy_k'), \]
where $\delta(x_1')=\gamma_1(t)$. Then
\[ \begin{split}
  &o(\calM_2(h_u))_{(x',y')}\wedge dx_2'\cdots dx_i'\wedge dy_{i+1}'\cdots dy_d'\\
  &=(-1)^{i-1}(e^t)^{i-1}(-y_1^2-u)dy_1'\cdots dy_d'dx_1'\cdots dx_d'.
\end{split}\]
Assuming that $x',y'$ converge to $x,y$ respectively as $u\to 0$, the coorientation $\lim_{u\to 0+}o^*_{M\times M}(\calM_2(h_u))_{(x',y')}$ is equivalent to 
\[ \lim_{u\to 0-}-\ve_{h_u}(p_-,p_+)\,o^*_M(\calA_{p_+}(h_u))_x\wedge o^*_M(\calD_{p_-}(h_u))_y.\]
This shows that the signs in (\ref{eq:sign_glue}) are correct. The proof of the invariance of the other terms in $\widehat{Z}_{2k,3k}(\vec{f}_{s_0-\ve})$ is the same as Lemma~\ref{lem:inv_ordered} since $v$ is independent.
\end{proof}

%%%%%%%%%%%%%%%%%%%%%%%%%%%%%%
\subsection{Invariance at $i/i$-intersection}

We check the invariance of $\widehat{Z}_{2k,3k}(\vec{f})$ with respect to bifurcations of type (3) in Lemma~\ref{lem:generic_1-para}.

\begin{Lem}\label{lem:inv_i/i}
Suppose that $s_0\in J$ is a point on which an $i/i$-intersection between critical points (loci) $p$ and $q$ occurs in a generic 1-parameter family $\{(f_s,\mu_s)\}_{s\in J}$. If $\ve$ is sufficiently small, then
\[ \widehat{Z}_{2k,3k}(\vec{f}_{s_0-\ve})
	=\widehat{Z}_{2k,3k}(\vec{f}_{s_0+\ve}). \]
\end{Lem}
\begin{proof}
By Proposition~\ref{prop:dbM_1-para}, we may assume without loss of generality that 
\[\sum_{\sigma\in\mathfrak{S}_{3k}}\sum_{\tau\subset E(\Gamma)}\Bigl(\#\bcalM_{\Gamma_\sigma^\tau}(\vec{f}_{s_0+\ve})-\#\bcalM_{\Gamma_\sigma^\tau}(\vec{f}_{s_0-\ve})
	-\#\bcalM_{d''\Gamma_\sigma^\tau}(\vec{f}_J)\Bigr)=0 \]
 if $\ve$ is sufficiently small. Let $g,g'$ be the combinatorial propagators considered in \S\ref{ss:change_g} and put $\vec{g}=(g,g_2,\ldots,g_{3k})$, $\vec{g}'=(g',g_2,\ldots,g_{3k})$, $\vec{C}=(C_*^{(s_0-\ve)},C_*^{(2)},\ldots,C_*^{(3k)})$ and $\vec{C}'=(C_*^{(s_0+\ve)},C_*^{(2)},\ldots,C_*^{(3k)})$. Using Corollary~\ref{cor:g'-g} we have
\[ \begin{split}
	&\sum_{\Gamma\in\calG_{2k,3k}^0(\vec{C})}\Tr_{\vec{g}'}\bigl(\Gamma\bigr)\cdot\#\bcalM_\Gamma(\vec{f}_{s_0+\ve})
	-\sum_{\Gamma\in\calG_{2k,3k}^0(\vec{C}')}\Tr_{\vec{g}}\bigl(\Gamma\bigr)\cdot\#\bcalM_\Gamma(\vec{f}_{s_0-\ve})\\
	&=\sum_\Gamma\Tr_{\vec{g}'}\bigl(\Gamma\bigr)\cdot\Bigl(\#\bcalM_\Gamma(\vec{f}_{s_0-\ve})+\#\bcalM_{d''\Gamma}(\vec{f}_J)\Bigr)-\sum_\Gamma\Tr_{\vec{g}}\bigl(\Gamma\bigr)\cdot\#\bcalM_\Gamma(\vec{f}_{s_0-\ve})\\
	&=\sum_\Gamma\Tr_{g'-g,\ldots}\bigl(\Gamma\bigr)\cdot\#\bcalM_\Gamma(\vec{f}_{s_0-\ve})+\sum_\Gamma\Tr_{\vec{g}'}\bigl(\Gamma\bigr)\cdot\#\bcalM_{d''\Gamma}(\vec{f}_J)\\
	&=\sum_\Gamma\Tr_{hg'-g'h,\ldots}\bigl(\Gamma\bigr)\cdot\#\bcalM_\Gamma(\vec{f}_{s_0-\ve})
		+\sum_\Gamma\Tr_{\vec{g}'}\bigl(\Gamma\bigr)\cdot\#\bcalM_{d''\Gamma}(\vec{f}_J)\\
	&=\Tr_{g',\ldots}\Bigl[\sum_{{{p_1,q_1}\atop{i(p_1)=i(q_1)+1}}}
		\sum_{\Gamma(p_1,q_1)_1}
		\Gamma(p_1,q_1)_1\cdot\#\bcalM_{(h*\Gamma-\Gamma*h)(p_1,q_1)_1}(\vec{f}_J)\Bigr]\\
		&\hspace{5mm}+\Tr_{\vec{g}'}\Bigl[\sum_{{{p_1,q_1}\atop{i(p_1)=i(q_1)+1}}}
		\sum_{\Gamma(p_1,q_1)_1}
		\Gamma(p_1,q_1)_1\cdot\#\bcalM_{d''\Gamma(p_1,q_1)_1}(\vec{f}_J)\Bigr]=0.
\end{split}\]
This completes the proof.
\end{proof}

\begin{proof}[Proof of Theorem~\ref{thm:Z_invariant}]
Lemmas~\ref{lem:inv_ordered}, \ref{lem:inv_exchange}, \ref{lem:inv_bd}, \ref{lem:inv_i/i} show that $\widehat{Z}_{2k,3k}(\vec{f}_s)$ is invariant under all possible bifurcations listed in \S\ref{ss:bifurcations}. This completes the proof.
\end{proof}

%%%%%%%%%%%%%%%%%%%%%%%%%%%%%%%
\appendix

%%%%%%%%%%%%%%%%%%%%%%%%%%%%
\mysection{Some facts on smooth manifolds with corners}{s:mfd_corners}

We follow the convention in \cite[Appendix]{BT} for manifolds with corners, smooth maps between them and their transversality. We write down some necessary terms from \cite[Appendix]{BT}, some of which are specialized than those in \cite[Appendix]{BT}.
\begin{Def}
\begin{enumerate}
\item A map between manifolds with corners is {\it smooth} if it has a local extension, at any point of the domain, to a smooth map from a manifold without boundary, as usual. 
\item Let $Y,Z$ be smooth manifolds with corners, and let $f:Y\to Z$ be a bijective smooth map. This map is a {\it diffeomorphism} if both $f$ and $f^{-1}$ are smooth.
\item Let $Y,Z$ be smooth manifolds with corners, and let $f:Y\to Z$ be a smooth map. This map is {\it strata preserving} if the inverse image by $f$ of a connected component $S$ of a stratum of $Z$ is a union of connected components of strata of $Y$. 
\item Let $X,Y$ be smooth manifolds with corners and $Z$ be a smooth manifold without boundary. Let $f:X\to Z$ and $g:Y\to Z$ be smooth maps. Say that $f$ and $g$ are {\it (strata) transversal} when the following is true: Let $U$ and $V$ be connected components in stratums of $X$ and $Y$ respectively. Then $f:U\to S$ and $g:V\to S$ are transversal.
\end{enumerate}
\end{Def}
We use the following proposition, which is a corollary of \cite[Proposition~A.5]{BT}.
\begin{Prop}\label{prop:BT}
Let $X,Y$ be smooth manifolds with corners and $Z$ be a smooth manifold without boundary. Let $f:X\to Z$ and $g:Y\to Z$ be smooth maps that are transversal. Then the fiber product
\[ X\times_Z Y=\{(x,y);f(x)=g(y)\}\subset X\times Y\]
is a smooth manifold with corners, whose strata have the form $U\times_Z V$ where $U\subset X$ and $V\subset Y$ are strata.
\end{Prop}

If $f,g$ are inclusions then $X\times_Z Y=(X\times Y)\cap \Delta_Z=\Delta_{X\cap Y}$, which is canonically diffeomorphic to $X\cap Y$. Thus we obtain the following corollary.

\begin{Cor}\label{cor:BT2}
Let $X,Y$ be smooth manifolds with corners that are submanifolds of a smooth manifold $Z$ without boundary. Suppose that the inclusions $X\to Z$ and $Y\to Z$ are transversal. Then the intersection $X\cap Y$ is a smooth manifold with corners, whose strata have the form $U\cap V$ where $U\subset X$ and $V\subset Y$ are strata.
\end{Cor}

The following elementary proposition is useful.

\begin{Prop}\label{prop:closure_corners}
Let $Z$ be a smooth manifold without boundary and let $X$ be a compact smooth submanifold of $Z$ with corners. Suppose that $\dim{X}>0$. Then the closure of the codimension 0 stratum $\mathrm{Int}\,X$ of $X$ in $Z$ agrees with $X$.
\end{Prop}
\begin{proof}
Let $n=\dim{X}$ and $N=\dim{Z}$. Let 
\[ \R^n\langle m\rangle=\{(x_1,\ldots,x_n); x_1\geq 0,\ldots,x_m\geq 0\}\subset\R^n\qquad(m\leq n). \]
Choose an open covering $\{O_\lambda\}_\lambda$ of $X$ by small open $N$-disks $O_\lambda$ in $Z$, say by open $\varepsilon$-balls with respect to the geodesic distance for a Riemannian metric on $Z$ for small $\ve$, so that for each $\lambda$ there is a chart $\varphi_\lambda:O_\lambda\to \varphi_\lambda(O_\lambda)\subset \R^N$ such that the restriction $\varphi_\lambda|_{O_\lambda\cap X}:O_\lambda\cap X\to \R^N$ factors as $\iota\circ \phi_\lambda$ where $\phi_\lambda:O_\lambda\cap X\to \R^n\langle m_\lambda\rangle$ is a chart and $\iota:\R^n\to \R^N$ is the inclusion $(x_1,\ldots,x_n)\mapsto (x_1,\ldots,x_n,0,\ldots,0)$. 

The codimension 0 stratum $\mathrm{Int}\,X$ of $X$ is the union of preimages of $\iota(\mathrm{Int}\,\R^n\langle m_\lambda \rangle)$ under charts $\varphi_\lambda$: $\mathrm{Int}\,X=\bigcup_\lambda O_\lambda \cap \varphi_\lambda^{-1}\iota(\mathrm{Int}\,\R^n\langle m_\lambda\rangle)$. The relation $\overline{\mathrm{Int}\,X}\subset X$ follows immediately from definition of the closure and the compactness of $X$. We prove the converse. Since $X$ is compact in $Z$, there is a finite subcovering $\{O_{\lambda_1},\ldots,O_{\lambda_r}\}$ of $X$. Then we have
\[ \begin{split}
  \overline{\mathrm{Int}\,X}&=
  \textstyle\bigcup_{i=1}^r \overline{O_{\lambda_i}\cap \varphi_{\lambda_i}^{-1}\iota(\mathrm{Int}\,\R^n\langle m_{\lambda_i}\rangle)}
  =\bigcup_{i=1}^r\varphi_{\lambda_i}^{-1}\overline{\varphi_{\lambda_i}(O_{\lambda_i})\cap\iota(\mathrm{Int}\,\R^n\langle m_{\lambda_i}\rangle)}\\
  &\textstyle\supset\bigcup_{i=1}^r\varphi_{\lambda_i}^{-1}(\varphi_{\lambda_i}(O_{\lambda_i})\cap\iota(\R^n\langle m_{\lambda_i}\rangle))
  =\bigcup_{i=1}^rO_{\lambda_i}\cap \varphi_{\lambda_i}^{-1}\iota(\R^n\langle m_{\lambda_i}\rangle)=X.
\end{split}\]
Here at the first equality we have used the identity $\overline{A_1\cup \cdots \cup A_r}=\overline{A_1}\cup\cdots\cup\overline{A_r}$ for arbitrary subsets $A_1,\ldots,A_r$ ($r<\infty$) of a topological space, and between the second and the third line we have used the relation $\overline{O\cap A}\supset O\cap \overline{A}$ for $O$ open, $A$ arbitrary, and the assumption $n\geq 1$. 
\end{proof}

\section{\bf Orientations on manifolds and their intersections}\label{s:ori}

%%%%%

For a $d$-dimensional orientable manifold $M$, we will represent an orientation on $M$ by a nowhere vanishing $d$-form of $\Omega_{\mathrm{dR}}^d(M)$ and denote by $o(M)$. If $M$ is a submanifold of an oriented Riemannian $e$-dimensional manifold $E$, then we may alternatively define $o(M)$ from an orientation $o^*_E(M)$ of the normal bundle of $M$ by the rule
\begin{equation}\label{eq:coori}
 o(M)\wedge o^*_E(M)\sim o(E). 
\end{equation}
Note that $o^*_E(M)$ is defined canonically by the Hodge star operator: $o^*_E(M)=*o(M)$. $o^*_E(M)$ is called a {\it coorientation} of $M$ in $E$. We assume that (\ref{eq:coori}) is always satisfied so that coorienation is just an alternative way to represent orientation. 

Let $N$ be an oriented smooth manifold and let $\pi:N\to E$ be a smooth map that is transversal to $M$. Then the preimage $\pi^{-1}M$ is naturally an oriented submanifold of $N$. We may define the coorientation of $\pi^{-1}M$ by $\pi^* o^*_E(M)$. We denote simply by $o^*_E(M)$ the coorientation $\pi^* o^*_E(M)$. For example, if $N=D\times E$ for an oriented manifold $D$ and if $\pi:D\times E\to E$ is the projection, then $D\times M=\pi^{-1}M$ is naturally cooriented by $o^*_E(M)$.  

If $M$ has boundary $\partial M$, we provide an induced orientation on $\partial M$ from $o(M)$ as follows: let $n$ be an inward normal vector field on $\partial M$, then we define 
\begin{equation}\label{eq:inward_first}
 o(\partial M)_x=\iota(n_x)o(M)_x. 
\end{equation}
In other words, if $n_x^*$ is the dual of $n_x$ with respect to the metric and if $o(M)_x=n_x^*\wedge \alpha_x$ for $\alpha_x\in\Omega^{d-1}_\mathrm{dR}(\partial M)$, then $o(\partial M)_x=\alpha_x$. This gives
\begin{equation}\label{eq:ind_ori_boundary}
o^*_E(\partial M)_x=(-1)^{e-1}o^*_E(M)_x\wedge n_x^*.
\end{equation}

Suppose $M$ and $M'$ are two cooriented submanifolds of $E$ of dimension $i$ and $j$ that intersect transversally. The transversality implies that at an intersection point $x$, the form $o^*_E(M)_x\wedge o^*_E(M')_x$ is a non-trivial $(2e-i-j)$-form. We define 
\begin{equation}\label{eq:coori_int}
 o^*_E(M\pitchfork M')_x=o^*_E(M)_x\wedge o^*_E(M')_x. 
\end{equation}
This depends on the order of the product. Note that if $M$ and $M'$ may have boundaries, then by (\ref{eq:ind_ori_boundary}), the induced coorientations on $\partial(M\pitchfork M')$ is
\begin{equation}\label{eq:ind_ori_bd_int}
 \begin{split}
    o^*_E(\partial M\pitchfork M')_x&=(-1)^{\mathrm{deg}\,o^*_E(M')_x}o^*_E(\partial M)_x\wedge o^*_E(M')_x,\\
    o^*_E(M\pitchfork \partial M')_x&=o^*_E(M)_x\wedge o^*_E(\partial M')_x. 
\end{split}
\end{equation}

%%%%%%%%%%%%%%%%%%%%%%%%%%%%%%
\section{\bf The complex of endomorphisms of an acyclic complex}\label{s:acyclic_cpx}

For a finitely generated, based free acyclic chain complex $(C_*,\partial)$, $C_i=\Z^{P_i}$, we consider the $\Z$-module $\End_k(C_*)$ of endomorphisms $C_*\to C_{*+k}$ of homogeneous degree $k$. The boundary operator $\partial':\End_k(C_*)\to \End_{k-1}(C_*)$ is defined by
\[ \partial'f=\partial\circ f+(-1)^{k+1}f\circ \partial. \]
Then the pair $(\End_*(C_*),\partial')$ forms a chain complex. By the canonical isomorphism $\End_k(C_*)\cong \bigoplus_{i\in\Z}C_{i+k}\otimes\Hom(C_i,\Z)$ of chain complexes and the K\"unneth theorem, one can show that the complex $(\End_*(C_*),\partial')$ is acyclic. For example, $f\in\End_0(C_*)$ is a cycle iff $\partial'f=\partial f-f\partial=0$. In particular, $\mathrm{id}\in\End_0(C_*)$ is a cycle and hence is a boundary. So there exists $g\in\End_1(C_*)$ such that 
\[ \partial' g=\partial g+g\partial =\mathrm{id}. \]
If two such endomorphisms $g,g'$ are given, then the difference $g-g'$ is a $\partial'$-cycle, since $\partial'(g-g')=\mathrm{id}-\mathrm{id}=0$. So there exists $h\in\End_2(C_*)$ such that
\[ \partial'h=\partial h-h\partial=g-g'. \]

%%%%%%%%%%%%%%%%%%%%%%%%%%%%%%%%%%%%%%%%%%%%%%%%%%%
\section{\bf Blow-up}\label{s:blow-up}

\subsection{Blow-up of the origin in $\R^i$}

Let $\wgamma^1(\R^i)$ denote the total space of the tautological oriented half-line ($[0,\infty)$) bundle over the oriented Grassmannian $\widetilde{G}_1(\R^i)\cong S^{i-1}$. Namely, $\wgamma^1(\R^i)=\{(x,y)\in S^{i-1}\times \R^i; \exists t\in[0,\infty), y=tx\}$. Then the tautological bundle is trivial and that $\wgamma^1(\R^i)$ is diffeomorphic to $S^{i-1}\times [0,\infty)$. Let
\[ B\ell(\R^i,\{0\})=\wgamma^1(\R^i) \]
and call $B\ell(\R^i,\{0\})$ the {\it blow-up} of $0$ in $\R^i$. Let $\pi:\wgamma^1(\R^i)\to \R^i$ be the map defined by $\pi=\mathrm{pr}_2\circ \varphi$ in the following commutative diagram:
\[ \xymatrix{
	\wgamma^1(\R^i) \ar[r]^-{\varphi} \ar[rd]_-{\pi}&
		S^{i-1}\times \R^i  \ar[d]^-{\mathrm{pr}_2} \ar[r]^-{\mathrm{pr}_1} & S^{i-1}\\
	& \R^i &
	}\]
where $\varphi:\wgamma^1(\R^i)\to S^{i-1}\times\R^i$ is the embedding which maps a pair $(x,y)\in S^{i-1}\times \R^i$ with $y=tx$ to $(x,y)$. If $y\neq 0$, then $\varphi(x,y)=(\frac{y}{|y|},y)$. We call $\pi$ the {\it projection} of the blow-up. Here, $\pi^{-1}(0)=\partial\wgamma^1(\R^i)$ is the image of the zero section of the tautological bundle $\mathrm{pr}_1\circ \varphi:\wgamma^1(\R^i)\to S^{i-1}$ and is diffeomorphic to $S^{i-1}$. 
\begin{Lem}\label{lem:bl_extention}\begin{enumerate}
\item The restriction of $\pi$ to the complement of $\pi^{-1}(0)=\partial\wgamma^1(\R^i)$ is a diffeomorphism onto $\R^i-\{0\}$. 
\item The restriction of $\varphi$ to the complement of $\pi^{-1}(0)$ has the image in $S^{i-1}\times \R^i$ whose closure agrees with the full image of $\varphi$ from $\wgamma^1(\R^i)$. 
\end{enumerate}
\end{Lem}
\subsection{Blow-up of $\R^i\subset \R^d$}

When $d>i\geq 0$, we put $B\ell(\R^d,\R^i)=\wgamma^1(\R^i)\times \R^{d-i}$ (the blow-up of $\R^i$ in $\R^d$) and define the projection $\varpi:B\ell(\R^d,\R^i)\to \R^d$ by $\pi\times\mathrm{id}_{\R^{d-i}}$. This can be straightforwardly extended to the blow-up $B\ell(Y,X)$ of a submanifold $X$ in a manifold $Y$ having oriented normal bundle, by replacing the normal bundle with the associated $\wgamma^1(\R^d)$-bundle over $X$.

\subsection*{Acknowledgments}
I would like to thank Professor Kenji Fukaya for encouraging me to write my result for publication. I would also like to thank Professor Masamichi Takase for valuable comments on spin 4-manifolds and would like to thank Professor Katrin Wehrheim and Tatsuro Shimizu for helpful comments. I would like to thank the referee for the careful reading and for lots of helpful and constructive comments. I am supported by JSPS Grant-in-Aid for Young Scientists (B) 23740040.

\end{document}